\documentclass[11pt]{article}

\usepackage{amsmath, amsthm}
\usepackage{amsmath, amsfonts}
\usepackage{amsmath, amssymb}
\usepackage{amsmath}
\usepackage{graphics}
\usepackage{graphicx}
\usepackage{color}
\usepackage{hyperref}
\usepackage[all]{xy}

\newtheorem{theorem}{Theorem}[subsection]
\newtheorem{definition}[theorem]{Definition}
\newtheorem{proto-definition}[theorem]{Proto-Definition}
\newtheorem{pseudo-definition}[theorem]{Pseudo-Definition}
\newtheorem{definition-lemma}[theorem]{Definition/Lemma}
\newtheorem{definition-explanation}[theorem]{Definition/Explanation}
\newtheorem{explanation-definition}[theorem]{Explanation/Definition}
\newtheorem{definition-fact}[theorem]{Definition/Fact}
\newtheorem{definition-notation}[theorem]{Definition/Notation}
\newtheorem{definition-conjecture}[theorem]{Definition/Conjecture}
\newtheorem{definition-theorem}[theorem]{Definition/Theorem}
\newtheorem{lemma}[theorem]{Lemma}
\newtheorem{lemma-definition}[theorem]{Lemma/Definition}

\newtheorem{corollary}[theorem]{Corollary}
\newtheorem{remark}[theorem]{\it Remark}
\newtheorem{remark-notation}[theorem]{\it Remark/Notation}

\newtheorem{application-lemma}[theorem]{Application/Lemma}

\newtheorem{example}[theorem]{Example}
\newtheorem{example-definition}[theorem]{Example/Definition}

\newtheorem{notation}[theorem]{Notation}

\newtheorem{definition-prototype}[theorem]{Definition-Prototype}

\numberwithin{equation}{subsection}

\newtheorem{sproto-definition}[stheorem]{Proto-Definition}
\newtheorem{spseudo-definition}[stheorem]{Pseudo-Definition}
\newtheorem{sdefinition-lemma}[stheorem]{Definition/Lemma}
\newtheorem{sdefinition-explanation}[stheorem]{Definition/Explanation}
\newtheorem{sexplanation-definition}[stheorem]{Explanation/Definition}
\newtheorem{sdefinition-fact}[stheorem]{Definition/Fact}
\newtheorem{sdefinition-notation}[stheorem]{Definition/Notation}
\newtheorem{sdefinition-conjecture}[stheorem]{Definition/Conjecture}
\newtheorem{sdefinition-theorem}[stheorem]{Definition/Theorem}

\newtheorem{slemma-definition}[stheorem]{Lemma/Definition}

\newtheorem{sremark-notation}[stheorem]{\it Remark/Notation}

\newtheorem{sapplication-lemma}[stheorem]{Application/Lemma}

\newtheorem{sexample-definition}[stheorem]{Example/Definition}

\newtheorem{sdefinition-prototype}[stheorem]{Definition-Prototype}

\newtheorem{sstheorem}{Theorem}[subsubsection]
\newtheorem{ssdefinition}[sstheorem]{Definition}
\newtheorem{ssproto-definition}[sstheorem]{Proto-Definition}
\newtheorem{sspseudo-definition}[sstheorem]{Pseudo-Definition}
\newtheorem{ssdefinition-lemma}[sstheorem]{Definition/Lemma}
\newtheorem{ssdefinition-explanation}[sstheorem]{Definition/Explanation}
\newtheorem{ssexplanation-definition}[sstheorem]{Explanation/Definition}
\newtheorem{ssdefinition-fact}[sstheorem]{Definition/Fact}
\newtheorem{ssdefinition-notation}[sstheorem]{Definition/Notation}
\newtheorem{ssdefinition-conjecture}[sstheorem]{Definition/Conjecture}
\newtheorem{ssdefinition-theorem}[sstheorem]{Definition/Theorem}
\newtheorem{sslemma}[sstheorem]{Lemma}
\newtheorem{sslemma-definition}[sstheorem]{Lemma/Definition}

\newtheorem{sscorollary}[sstheorem]{Corollary}
\newtheorem{ssremark}[sstheorem]{\it Remark}
\newtheorem{ssremark-notation}[sstheorem]{\it Remark/Notation}

\newtheorem{ssapplication-lemma}[sstheorem]{Application/Lemma}

\newtheorem{ssexample-definition}[sstheorem]{Example/Definition}

\newtheorem{ssdefinition-prototype}[sstheorem]{Definition-Prototype}

\newcommand{\Aut}{\mbox{\it Aut}\,}

\newcommand{\End}{\mbox{\it End}\,}

\newcommand{\Endsheaf}{\mbox{\it ${\cal E}\!$nd}\,}

\newcommand{\Hom}{\mbox{\it Hom}\,}
\newcommand{\Homsheaf}{\mbox{\it ${\cal H}$om}\,}

\newcommand{\Id}{\mbox{\it Id}\,}

\newcommand{\Image}{\mbox{\it Im}\,}

\newcommand{\SO}{\mbox{\it SO}\,}

\newcommand{\Span}{\mbox{\it Span}\,}
\newcommand{\Spec}{\mbox{\it Spec}\,}
 \newcommand{\boldSpec}{\mbox{\it\bf Spec}\,}

\newcommand{\Sym}{\mbox{\it Sym}}

\newcommand{\anticommuting}{\mbox{\scriptsize\it anti-c}}

\newcommand{\determinant}{\mbox{\it det}\,}

\newcommand{\dimm}{\mbox{\it dim}\,}

\newcommand{\even}{\mbox{\scriptsize\rm even}\,}
\newcommand{\tinyeven}{\mbox{\tiny\rm even}\,}

\newcommand{\odd}{\mbox{\scriptsize\rm odd}\,}

\newcommand{\pr}{\mbox{\it pr}}

\newcommand{\boldy}{\mbox{\boldmath $y$}}
  \newcommand{\scriptsizeboldy}{\mbox{\scriptsize\boldmath $y$}}
  \newcommand{\tinyboldy}{\mbox{\tiny\boldmath $y$}}

\newcommand{\rightaarrow}{\rightarrow\hspace{-2ex}\rightarrow}

\begin{document}

\enlargethispage{24cm}

\begin{titlepage}

$ $

\vspace{-1.5cm} 

\noindent\hspace{-1cm}
\parbox{6cm}{\small August 2017}\
   \hspace{7cm}\
   \parbox[t]{6cm}{yymm.nnnnn [math.DG] \\
                D(11.4.1): smooth map; super\\
				}

\vspace{2cm}

\centerline{\large\bf
 Further studies of}
\vspace{1ex}
\centerline{\large\bf
 the notion of differentiable maps from Azumaya/matrix supermanifolds}
\vspace{1ex}
\centerline{\large\bf
 I.\ The smooth case:} 
\vspace{1ex}
\centerline{\large\bf 
 Ramond-Neveu-Schwarz and Green-Schwarz meeting Grothendieck}

\vspace{3em}

\centerline{\large
  Chien-Hao Liu
            \hspace{1ex} and \hspace{1ex}
  Shing-Tung Yau
}

\vspace{4em}

\begin{quotation}
\centerline{\bf Abstract}

\vspace{0.3cm}

\baselineskip 12pt  
{\small
 In this sequel to works
    D(11.1) (arXiv:1406.0929 [math.DG]), 
    D(11.2) (arXiv:1412.0771 [hep-th]), and 
    D(11.3.1) (arXiv:1508.02347 [math.DG]),
 we re-examine --- and reformulate when in need --- several basic notions in super $C^{\infty}$-algebraic geometry
   as guided by the mathematical formulation
      of Ramond-Neveu-Schwarz fermionic strings  and of Green-Schwarz fermionic strings
    from the viewpoint of Grothendieck on Algebraic Geometry.
 Two theorems
     that are the super counterpart of Theorem~3.1.1 and Theorem~3.2.1 of D(11.3.1)
   are proved.
 They unify the notion of
   `smooth maps from an Azumaya/matrix super smooth manifold with a fundamental module to a super smooth manifold'
  introduced in D(11.2),
   making it a complete super parallel to the setting for D-branes
     in the realm of algebraic geometry in
	  D(1) (arXiv:0709.1515 [math.AG]) and
	  D(2) (arXiv:0809.2121 [math.AG]),    and 	
     in the realm of differential or $C^{\infty}$-algebraic geometry in D(11.1) and D(11.3.1).	 	
 A prototypical definition of dynamical fermionic stacked D-brane world-volume on a space-time
    in the same spirit of RNS fermionic strings or GS fermionic strings is thus laid down.	
 Similar to D(11.3.1), which paved the path to
      the construction of non-Abelian Dirac-Born-Infeld action (D(13.1) (arXiv:1606.08529 [hep-th]))
	 and the standard action (D(13.3) (arXiv:1704.03237 [hep-th])) for fundamental bosonic stacked D-branes,
 the current notes shall serve the same for the construction of supersymmetric action
  for fundamental fermionic stacked D-branes of various dimensions --- a theme of another subseries of the D-project.
 A notion of ``noncommutative $C^{\infty}$-rings" and `morphism' between them
  is introduced at the end as a byproduct.
} 
\end{quotation}

\vspace{4em}

\baselineskip 12pt
{\footnotesize
\noindent
{\bf Key words:} \parbox[t]{14cm}
{fermionic D-brane;
 Azumaya/matrix supermanifold, super $C^{\infty}$-scheme,
  $C^{\infty}$-map,\\ general super $C^{\infty}$-ring-homomorphism;
 spectral subscheme,
 Malgrange Division Theorem;\\
 $C^{\infty}$-admissible noncommutative ring, $C^{\infty}$-hull, morphism
 }} 

 \bigskip

\noindent {\small MSC number 2010: 58A40, 14A22, 81T30; 51K10, 16S50, 46L87.
} 

\bigskip

\baselineskip 10pt
{\scriptsize
\noindent{\bf Acknowledgements.}
We thank Andrew Strominger and Cumrun Vafa
   for  influence to our understanding of strings, branes, and gravity.
C.-H.L.\ thanks in addition
 Philip Engel, Karsten Gimre, David Nelson, Adrian Ocneanu
    for topic courses, fall 2017, while editing the notes;
 Jia-Chie Lee
     for her motivating life story in late spring,
 Yu-Chun Liu
     for a special mid-summer,
 Master Guo-Guang Shr for a meeting in late summer,  	
  while in various stages of the work;
 Ling-Miao Chou
     for comments on illustrations and moral support.
S.-T.Y.\ thanks in addition
 Mathematical Science Center and Department of Mathematical Sciences, Tsinghua University, Beijing, China,
     for hospitality. 	
The project is supported by NSF grants DMS-9803347 and DMS-0074329.
} 

\end{titlepage}

\newpage

\begin{titlepage}

$ $

\vspace{12em}

\centerline{\small\it
 Chien-Hao Liu dedicates the current notes to former librarians}
\centerline{\small\it
 Diane Pochini and Martha Wooster,}
\centerline{\small\it
 and the several student assistants for the evenings and weekends}
\centerline{\small\it
  who come and go and whose name I never know,}
\centerline{\small\it
 of the Blue Hill/Gordon McKay Library for Meteorology, Engineer, and Applied Sciences}
\centerline{\small\it
 (closed permanently in late December 2016) on the third floor of the Pierce Building}
\centerline{\small\it
 for their nodding smiles over these years}
\centerline{\small\it
 and for keeping a very quiet library in an isolated corner of a campus building,}
\centerline{\small\it
 where many ideas in the first decade of the D-project were first conceived.}

%
%

\end{titlepage}


\newpage
$ $

\vspace{-3em}

\centerline{\sc
  Maps from Azumaya/Matrix Supermanifolds I: Smooth Case
 } %

\vspace{2em}


\begin{flushleft}
{\Large\bf 0. Introduction and outline}
\end{flushleft}
In this sequel to
 [L-Y2] (D(11.1)),
 [L-Y3] (D(11.2)), and
 [L-Y4] (D(11.3.1)),
 we re-examine --- and reformulate when in need --- several basic notions in super $C^{\infty}$-algebraic geometry
   as guided by the mathematical formulation
      of Ramond-Neveu-Schwarz fermionic strings  and of Green-Schwarz fermionic strings
    from the viewpoint of Grothendieck on Algebraic Geometry (Sec.$\:$1).	
Two theorems
     that are the super counterpart of [L-Y4: Theorem~3.1.1 \& Theorem~3.2.1] (D(11.3.1))
   are proved (Sec.$\:$2).
They unify the notion of
   `smooth maps from an Azumaya/matrix super smooth manifold with a fundamental module to a super smooth manifold'
  introduced in [L-Y3] (D(11.2)),
   making it a complete super parallel to the setting for D-branes
     in the realm of algebraic geometry in
	  [L-Y1] ( D(1)) and
	  [L-L-S-Y] (D(2)),    and 	
     in the realm of differential or $C^{\infty}$-algebraic geometry in [L-Y2] (D(11.1)) and [L-Y4] (D(11.3.1));
  (cf.\ Sec.$\:$3.1).
A prototypical definition of a dynamical fermionic stacked D-brane world-volume on a space-time
    in the same spirit of RNS fermionic strings or GS fermionic strings is thus laid down (Definition-Prototype~3.1.1).	
Similar to [L-Y4] (D(11.3.1)), which paved the path to
      the construction of non-Abelian Dirac-Born-Infeld action [L-Y5] (D(13.1))
	 and the standard action [L-Y7] (D(13.3)) for a fundamental bosonic stacked D-brane,
 the current notes shall serve the same for the construction of a supersymmetric action
  for fundamental fermionic stacked D-branes in various dimensions --- a theme of another subseries of the D-project.

As a byproduct from the study, we introduce a notion of ``noncommutative $C^{\infty}$-rings"
 and `morphism' between them, which covers all we have ran into in the project
(Sec.$\:$3.2).

\bigskip

\bigskip

\noindent
{\bf Convention.}
 References for standard notations, terminology, operations and facts in
    (1) algebraic geometry;
    (2) synthetic geometry, $C^{\infty}$-algebraic geometry;
    (3) string theory and D-branes;
	(4) supersymmetry
 can be found respectively in
    (1) [Ha];
    (2) [Du], [Jo], [Ko], [M-R];
    (3) [G-S-W], [Po];
	(4) [S-W], [West]. [Wi], [W-B].
 \begin{itemize}
  \item[$\cdot$]
   For clarity, the {\it real line} as a real $1$-dimensional manifold is denoted by ${\Bbb R}^1$,
    while the {\it field of real numbers} is denoted by ${\Bbb R}$.
   Similarly, the {\it complex line} as a complex $1$-dimensional manifold is denoted by ${\Bbb C}^1$,
    while the {\it field of complex numbers} is denoted by ${\Bbb C}$.
	
  \item[$\cdot$]	
  The inclusion `${\Bbb R}\hookrightarrow{\Bbb C}$' is referred to the {\it field extension
   of ${\Bbb R}$ to ${\Bbb C}$} by adding $\sqrt{-1}$, unless otherwise noted.

  \item[$\cdot$]
   The {\it complexification} of an ${\Bbb R}$-module $M$ is denoted by
    $M^{\Bbb C}\;(:= M\otimes_{\Bbb R}{\Bbb C})$.

 \item[$\cdot$]	
  The {\it real $n$-dimensional vector spaces} ${\Bbb R}^{\oplus n}$
      vs.\ the {\it real $n$-manifold} $\,{\Bbb R}^n$; \\
  similarly, the {\it complex $r$-dimensional vector space ${\Bbb C}^{\oplus r}$}
     vs.\ the {\it complex $r$-fold} $\,{\Bbb C}^r$.

 \item[$\cdot$]
  All $C^{\infty}$-manifolds
     are paracompact, Hausdorff, admitting a (locally finite) partition of unity,
     and embeddable into some ${\Bbb R}^N$ as closed smooth submanifolds.
  We adopt the {\it index convention for tensors} from differential geometry.
   In particular, the tuple coordinate functions on an $n$-manifold is denoted by, for example,
   $(y^1,\,\cdots\,y^n)$.
  However, no up-low index summation convention is used.

  \item[$\cdot$]
   `{\it smooth}' $=C^{\infty}$;
    the set (or group, or ring, or module) of smooth sections of a bundle or sheaf is denoted by
	 $C^{\infty}(\,\mbox{\LARGE $\cdot$}\,)$.

  \item[$\cdot$]
   $\Spec R $ ($:=\{\mbox{prime ideals of $R$}\}$)
         of a commutative Noetherian ring $R$  in algebraic geometry\\
   vs.\ $\Spec R$ of a $C^{\infty}$-ring $R$
  ($:=\Spec^{\Bbb R}R :=\{\mbox{$C^{\infty}$-ring homomorphisms $R\rightarrow {\Bbb R}$}\}$).

  \item[$\cdot$]
  {\it morphism} between schemes in algebraic geometry
    vs.\ {\it $C^{\infty}$-map} between $C^{\infty}$-manifolds or $C^{\infty}$-schemes
         	in differential topology, differential geometry, and $C^{\infty}$-algebraic geometry.
			
  \item[$\cdot$]			
   {\it matrix} $m$ vs.\ manifold of {\it dimension} $m$.
   
  \item[$\cdot$]
  {\it coordinate tuple} $(y^1,\,\cdots\,,\, y^n)$
    vs.\ {\it ideal} $(y^1,\,\cdots\,,\, y^n)$ generated by $y^1,\,\cdots\,,\, y^n$.
\end{itemize}

\bigskip

\begin{flushleft}
{\bf Outline}
\end{flushleft}
\nopagebreak
{\small
 \baselineskip 12pt  
 \begin{itemize}
    \item[1]
     From fermionic strings to general morphisms in super $C^{\infty}$-algebraic geometry:\\
	 Ramond-Neveu-Schwarz and  Green-Schwarz meeting Grothendieck
	  \vspace{-.6ex}
	  \begin{itemize}
	   \item[1.1]
	    Fermionic strings from the aspect of Grothendieck's modern Algebraic Geometry
		\begin{itemize}
		 \item[\Large $\cdot$\;]
          Ramond-Neveu-Schwarz and  Green-Schwarz meeting Grothendieck
		
		 \item[\Large $\cdot$\;]
          An issue on morphisms between superrings brought out by RNS fermionic strings	
        \end{itemize}
		
	   \item[1.2]	
		Extensions of $C^{\infty}$-ring structure and $C^{\infty}$-ring-homomorphism
		
	   \item[1.3]
	    Basics of super algebraic geometry
  
       \item[1.4]
	    Super $C^{\infty}$-rings, super $C^{\infty}$-schemes, and general morphisms
	  \end{itemize}
  
   \item[2]
     A further study of the notion of smooth maps from an Azumaya/matrix supermanifold
	  \vspace{-.6ex}
	  \begin{itemize}
	   \item[2.1]
	    The setup and the statement of two main theorems
		 \begin{itemize}
		  \item[\Large $\cdot$\;]
		   Azumaya/matrix super $C^{\infty}$-manifolds with a fundamental module
		 
		  \item[\Large $\cdot$\;]
		   Two Main Theorems 
		     on maps from $(\widehat{X}_{[s_1]}, \widehat{E})$ to $\widehat{Y}_{[s_2]}$
		 \end{itemize}
		 		 
	   \item[2.2]
		Preliminaries on endomorphisms and primary decompositions
		 \begin{itemize}		
		  \item[2.2.1]
	       Endomorphisms of a free module over a complex Grassmann algebra
			 \vspace{.6ex}
		     \begin{itemize}
			  \item[\Large $\cdot$\;]
		       The automorphism group $\mbox{\it Aut$\,$}_{\widehat{\Bbb C}_{[s]}}(\widehat{E})$
		       of $\widehat{E}$	
			   
		      \item[\Large $\cdot$\;]
		       Primary decomposition of $\widehat{E}$
		       under an $\widehat{m}\in \mbox{\it End$\,$}_{\widehat{\Bbb C}_{[s]}}(\widehat{E})$
		
		      \item[\Large $\cdot$\;]
			   Primary decomposition of $\widehat{E}$ under a commuting system of endomorphisms					   
	         \end{itemize}
		
          \vspace{.3ex}		
		  \item[2.2.2]
            Generalization of Sec.$\:$2.2.1 to
            $C^{\infty}(\mbox{\it End$\,$}_{\widehat{\Bbb C}_{[s]}}(\widehat{E}))$
            for general $X$			
         \end{itemize}
   		 
       \item[2.3]	
	    $C^{\infty}$-maps from an Azumaya/matrix superpoint to a real supermanifold
         \begin{itemize}		
		  \item[2.3.1]
		   Proof of Theorem~2.1.5 when $X$ is a point

		  \item[2.3.2]
           Proof of Theorem~2.1.8 when $X$ is a point
         \end{itemize}
	  	  	    	
	   \item[2.4]
	    $C^{\infty}$-maps from an Azumaya/matrix supermanifold to a real supermanifold	
         \begin{itemize}		
		  \item[2.4.1]
		   Proof of Theorem~2.1.5

		  \item[2.4.2]
           Proof of Theorem~2.1.8
         \end{itemize}
	  \end{itemize}	
	
	\item[3]
	 Remarks on fermionic D-branes and on ``noncommutative $C^{\infty}$-rings" after the study
	  \vspace{-.6ex}
	  \begin{itemize}
	   \item[3.1]
	    Fermionic D-branes as dynamical objects \`{a} la RNS or GS fermionic strings
	
	   \item[3.2]
	    Remarks on the notion of `$C^{\infty}$-admissible noncommutative rings'
      \end{itemize}
 \end{itemize}
} 

\newpage

\section{From fermionic strings to general morphisms in super\\ $C^{\infty}$-algebraic geometry:
	               Ramond-Neveu-Schwarz and\\  Green-Schwarz meeting Grothendieck}
				
In this section we pick up where we were in [L-Y3: Sec.$\:$5.1] (D(11.2)) and
 review
  how one should think of fermionic strings from the aspect of Grothendieck's formulation of Algebraic Geometry
 (Sec.$\:$1.1).
This brings out a notion of `{\sl general morphism}' between superrings
  (resp.\ super $C^{\infty}$-rings, super $C^{\infty}$-schemes) (Sec.$\:$1.3 and Sec.$\:$1.4),
  if one wants to take physicists' notion of supersymmetry into account and
     generalize the notion of fermionic strings to fermionic D-branes as fundamental dynamical objects in string theory.
Some basic results we need on extensions of super $C^{\infty}$-ring structure and general morphisms are presented
 in Sec.$\:$1.2.

\bigskip

\subsection{Fermionic strings from the aspect of Grothendieck's modern Algebraic Geometry}

Fermionic strings are fundamental/dynamical objects in superstring theory.
There are two formulations of fermionic strings (either open or closed):
\begin{itemize}
  \item[(1)] {\it Ramond-Neveu-Schwarz} ({\it RNS}) {\it fermionic string},
   for which world-sheet spinors and world-sheet supersymmetry are manifestly involved
   ([N-S] of  Andr\'{e} Neveu and John Schwarz
   and [Ra] of Pierre Ramond);
   
  \item[(2)]  {\it Green-Schwarz} ({\it GS}) {\it fermionic string},
   for which space-time spinors and space-time supersymmetry are manifestly involved
   ([G-S] of Michael Green and John Schwarz).
\end{itemize}
Mathematicians are referred particularly to [G-S-W: Chap.\  4 \& Chap.\ 5] of Green, Schwarz, and Witten
 for thorough explanations.

\bigskip

\begin{flushleft}
{\bf  Ramond-Neveu-Schwarz and  Green-Schwarz meeting Grothendieck}
\end{flushleft}
We now relook at each from the viewpoint of Grothendieck's Algebraic Geometry.
The discussion here follows [G-S-W: Chap.\ 4 \& Chap.\ 5]
   (with possibly some mild change of notations to be compatible with the current notes)  and [Ha: Chap.\ II].
Let
  ${\Bbb M}^{(d-1)+1}$   be the $d$-dimensional Minkowski space-time with coordinates
    $y:= (y^{\mu})_{\mu}=(y^0, y^1,\,\cdots\,, y^{d-1})$ and
  $\Sigma\simeq {\Bbb R}^1\times S^1$ or ${\Bbb R}^1\times [0,2\pi]$ be a string world-sheet 	
   with coordinates $\sigma:=(\sigma^0,\sigma^1)$.
   
\bigskip

\noindent
$(a)$ {\it Ramond-Neveu-Schwarz} ({\it RNS}) {\it fermionic string}

\medskip

\noindent
In this setting,
 there are both bosonic (world-sheet scalar) fields $y^{\mu}(\sigma)$
   and fermionic (world-sheet spinor) fields $\psi^{\mu}(\sigma)$
 on the string world-sheet $\Sigma$ for $\mu=0,1, \,\cdots\,, d-1$.
The former  collectively describe a map $f:\Sigma \rightarrow {\Bbb M}^{(d-1)+1}$ and
the latter as its superpartner.

Consider the supermanifold $\widehat{\Sigma}$ that have the same topology as $\Sigma$
 but with additional Grassmann coordinates $\theta := (\theta^A)_A =(\theta^1, \theta^2)$
   forming $2$-component Majorana spinor on $\Sigma$.
Then, after adding auxiliary (nondynamical) fields $B^{\mu}(\sigma)$ to the world-sheet,
these fields on $\Sigma$ can be grouped to superfields:(Cf.\ [G-S-W: Sec.\ 4.1.2; Eq.\ (4.1.16)].)
 $$
	 Y^{\mu}(\sigma) \;
		=\; y^{\mu}(\sigma)\,+\,\bar{\theta}\psi^{\mu}(\sigma)\,
		        +\, \frac{1}{2}\,\bar{\theta}\theta\,B^{\mu}(\sigma)\,.		
 $$	
  
 {From} the viewpoint of Grothendieck's Algebraic Geometry,
  a map $\widehat{f} : \widehat{\Sigma}\rightarrow {\Bbb M}^{(d-1)+1}$ is specified
  contravariantly by a homomorphism
   $$
    \begin{array}{ccccc}
      \widehat{f}^{\sharp}  &  :
	    & C^{\infty}({\Bbb M}^{(d-1)+1})
        & \longrightarrow     &   C^{\infty}(\widehat{\Sigma})\\[1.2ex]
	 && y^{\mu}  & \longmapsto   & \widehat{f}^{\sharp}(y^{\mu})
	\end{array}
   $$
   of the function-rings in question.
 Since $C^{\infty}(\widehat{\Sigma})=C^{\infty}(\Sigma)[\theta^1,\theta^2]$
    (with $\theta^1$, $\theta^2$ anticommuting)
     a superpolynomial ring over $C^{\infty}(\Sigma)$,
 $\hat{f}^{\sharp}(y^{\mu})$ must be of the form
  $$
    \widehat{f}^{\sharp}(y^{\mu})\;
	  =\; f^{\mu}(\sigma)\,
	          +\,\bar{\theta}\psi^{\mu}(\sigma)\,
		          +\, \frac{1}{2}\,\bar{\theta}\theta\,B^{\mu}(\sigma)\,,
  $$
  which is exactly the previous quoted expression [G-S-W: Eq.\ (4.1.16)].
 In conclusion,
  \begin{itemize}
   \item[{\Large $\cdot$}]
    {\it A Ramond-Neveu-Schwarz fermionic string moving in a Minkowski space-time ${\Bbb M}^{(d-1)+1}$
	    as studied in {\rm [G-S-W: Chap.\ 4]}
      can be described by
	   a map $\widehat{f}: \widehat{\Sigma} \rightarrow {\Bbb M}^{(d-1)+1}$
	    in the sense of Grothendieck's Algebraic Geometry.}
  \end{itemize}
   
\bigskip

\noindent
$(b)$ {\it Green-Schwarz} ({\it GS}) {\it fermionic string}

\medskip

\noindent
In this setting,
 in addition to the ordinary bosonic (world-sheet scalar) fields
   $y^{\mu}(\sigma)$, $\mu=0,1,\,\cdots\,, d-1$, on $\Sigma$
   that collectively describe a map $f: \Sigma\rightarrow {\Bbb M}^{(d-1)+1}$,
 there are also a set of {\it world-sheet scalar yet mutually anticommuting} fields
   $\theta^{Aa}(\sigma)$, $A=1,\,\cdots\,, N$ and $a=1,\,\cdots\,, s$, on $\Sigma$.
Here $s$ is the dimension of a spinor representation of the Lorentz group $\SO(d-1,1)$
 of the target Minkowski space-time ${\Bbb M}^{(d-1)+1}$.

{\it Differential geometrically} intuitively, one would think of these (world-sheet scalar) fields on $\Sigma$
 collectively as follows:
 \begin{itemize}
  \item[{\Large $\cdot$}]
   Let $\widehat{\Bbb M}^{(d-1)+1}$ be a superspace
     with coordinates
	   the original coordinates $y:=(y^{\mu})_{\mu}$ of ${\Bbb M}^{(d-1)+1}$
	  and additional anticommuting coordinates $\theta^{Aa}$,
	     $A=1,\,\cdots\,, N$ and $a=1,\,\cdots\,, s$,
     such that
   	 each tuple $(\theta^{A1}, \,\cdots\,, \theta^{As})$, $A=1,\,\cdots\,,N$,
   	   is in a spinor representation of the Lorentz group $\SO(d-1,1)$,
     	   the symmetry of the space-time ${\Bbb M}^{(d-1)+1}$ .
   Note that $\widehat{\Bbb M}^{(d-1)+1}\simeq {\Bbb R}^{d|Ns}$ as supermanifolds.
	
  \item[{\Large $\cdot$}]
   The collection $(y^{\mu}(\sigma), \theta^{Aa}(\sigma))_{\mu, A, a}$
    of (world-sheet scalar) fields on $\Sigma$ describe collectively
    a map $\widehat{f}: \Sigma \rightarrow \widehat{\Bbb M}^{(d-1)+1}$.
  In other words,
  a Green-Schwarz fermionic string moving in ${\Bbb M}^{(d-1)+1}$  is described
   by a map from an ordinary world-sheet to a super-Minkowski space-time.
 \end{itemize}
 
However, {\it algebraic geometrically} some revision to this naive differential geometric picture has to be made.
 \begin{itemize}
  \item[{\Large $\cdot$}]
   One would like a contravariant  equivalence between spaces and their function-ring:
    $$
	 \begin{array}{ccccc}
       \widehat{f}& : & \Sigma  & \longrightarrow & \widehat{\Bbb M}^{(d-1)+1}
	 \end{array}
    $$
   with
   $$
    \begin{array}{cccccl}
      \widehat{f}^{\sharp}& :
	    &  C^{\infty}({\Bbb M}^{(d-1)+1})[\theta^{Aa}\,:\,  1\le A \le N,\, 1\le a\le s]
		& \longrightarrow &  C^{\infty}(\Sigma) \\[1.2ex]
	  &&	y^{\mu}       & \longmapsto     &   y^{\mu}(\sigma) \\[1.2ex]
	  &&    \theta^{Aa}& \longmapsto     &  ?                                         &.
	 \end{array}
   $$
  Here,
    $C^{\infty}({\Bbb M}^{(d-1)+1})[\theta^{Aa}\,:\,  1\le A \le N,\, 1\le a\le s]$
       is the superpolynomial ring over the $C^{\infty}$-ring $C^{\infty}({\Bbb M}^{(d-1)+1})$
	  with anticummuting generators in $\{\theta^{Aa}\}_{A, a}$.
   
  \item[{\Large $\cdot$}]
   The natural candidate for $\widehat{f}(\theta^{Aa})$ is certainly the world-sheet scalar field
     $\theta^{Aa}(\sigma)$ regarded as an element in the function-ring  of $\Sigma$.
   However, the anticommuting nature of fields $\theta^{Aa}$, $1\le A\le N$ and $1\le a\le s$,
    among themselves forbids them to lie in $C^{\infty}(\Sigma)$.

  \item[{\Large $\cdot$}]	
   The way out of this from the viewpoint of Grothendieck's Algebraic Geometry
    is to extend the world-sheet $\Sigma$ also
    to a superworld-sheet $\widehat{\Sigma}$ with the function-ring the superpolynomial ring
    $C^{\infty}(\Sigma)[\theta^{\prime Aa}\,:\, 1\le A\le N,\, 1\le a\le s]$.
	
  \item[{\Large $\cdot$}]	
   One now has a well-defined super-$C^{\infty}$-ring-homomorphism
   $$
    \begin{array}{cccccl}
      \widehat{f}^{\sharp}& :
	   & C^{\infty}({\Bbb M}^{(d-1)+1})[\,\theta^{Aa}\,:\,  A,\, a\,]
	   & \longrightarrow
	   &  C^{\infty}(\Sigma)[\,\theta^{\prime Aa}\,:\, A,\, a\,]       \\[1.2ex]
	  &&	y^{\mu}       & \longmapsto     &   y^{\mu}(\sigma) \\[1.2ex]
	  &&    \theta^{Aa}& \longmapsto     &   \theta^{Aa}(\sigma)                                         &.
	 \end{array}
   $$
   
  \item[{\Large $\cdot$}]
   Furthermore, since all the fields $\theta^{Aa}(\sigma)$ are dynamical,
   in comparison with the setting for the RNS fermionic string, it is reasonable to require in addition that
   $$
    \widehat{f}^{\sharp}(\theta^{Aa})\; =\; \theta^{Aa}(\sigma)\;
	  \in \; \Span_{C^{\infty}(\Sigma)}\{\,\theta^{\prime Aa}\,|\, A, a\,\}\,.
   $$
 \end{itemize}

In conclusion,
  \begin{itemize}
   \item[{\Large $\cdot$}] {\it
    Assuming the notation from the above discussion.
	A Green-Schwarz fermionic string moving in a Minkowski space-time ${\Bbb M}^{(d-1)+1}$
	    as studied in {\rm [G-S-W: Chap.\ 5]}
      can be described in the sense of Grothendieck's Algebraic Geometry	
	  by a map $\widehat{f}: \widehat{\Sigma} \rightarrow  \widehat{\Bbb M}^{(d-1)+1}$,
	   defined by a super-$C^{\infty}$-ring-homomorphism
      $$
       \begin{array}{cccccl}
         \widehat{f}^{\sharp}& :
	      & C^{\infty}({\Bbb M}^{(d-1)+1})[\,\theta^{Aa}\,:\,  A,\, a\,]
	      & \longrightarrow
	      &  C^{\infty}(\Sigma)[\,\theta^{\prime Aa}\,:\, A,\, a\,]       \\[1.2ex]
	      &&	y^{\mu}       & \longmapsto     &   y^{\mu}(\sigma) \\[1.2ex]
	      &&    \theta^{Aa}& \longmapsto     &   \theta^{Aa}(\sigma)                                         
	   \end{array}
      $$
	 such that
      $$
        \widehat{f}^{\sharp}(\theta^{Aa})\; =\; \theta^{Aa}(\sigma)\;
	      \in \; \Span_{C^{\infty}(\Sigma)}\{\,\theta^{\prime Aa}\,|\, A, a\,\}\,.
      $$	
    }\end{itemize}

\bigskip

\begin{flushleft}
{\bf An issue on morphisms between superrings brought out by RNS fermionic strings}
\end{flushleft}
The above Grothendieck-reformat of Ramond-Neveu-Schwarz and Green-Schwarz
 turns the underlying basic mathematical notion of a fermionic string
 contravariantly as a morphism $\widehat{f}^{\sharp}$ between function-rings of  supermanifolds.
However, there is an issue here for RNS fermionic strings.
 \begin{itemize}
  \item[\LARGE $\cdot$]
   {\it Function-rings of supermanifolds are ${\Bbb Z}/2$-graded (i.e.\ even-odd). That is, they are superrings.
   Mathematically most naturally, a morphism between superrings are required to be ${\Bbb Z}/2$-grading preserving.}
   Yet, to have a bosonic-fermionic partner pair of maps (i.e.\ (map, ``mappino")-pair)
     and a world-sheet supersymmetry in the Ramond-Neveu-Schwarz fermionic string,
   $\widehat{f}^{\sharp}$ cannot be ${\Bbb Z}/2$-grading preserving.
   Otherwise,
     since $y^{\mu}$ in the function-ring $C^{\infty}({\Bbb M}^{(d-1)+1})$
	   of the target ${\Bbb M}^{(d-1)+1}$ is even,
     the world-sheet fermionic partner $(\psi^{\mu})_{\mu}$	of the map $f$ has to vanish
	   to keep $\widehat{f}^{\sharp}(y^{\mu})$ even as well.
   This would then leave no room for world-sheet supersymmetric transformations on component fields.  	
 \end{itemize}
(While for Green-Schwarz fermionic string,
     it is no harm to require $\widehat{f}^{\sharp}$ to be ${\Bbb Z}/2$-grading-preserving,
	    though physically one may not have to.)

This suggests that
 if one would like to
   combine the basic setup of D-branes as maps from Azumaya/matrix manifold to a target manifold
 (cf.\ [L-Y2] (D(11.1)))
     with the Ramond-Neveu-Schwarz or Green-Schwarz formulation of ferminic strings
   to give a formulation of fermionic D-branes as fundamental (as opposed to solitonic) objects in string theory
   (cf.\ [L-Y3: Sec.$\:$5] (D(11.2)) for a light initiating glimpse),
one needs to reconsider
  a mathematically sound and ``physically correct" notion of `{\sl morphisms}' between superrings.

\bigskip

\subsection{Extensions of $C^{\infty}$-ring structure and $C^{\infty}$-ring-homomorphism}

Readers are referred to the work [Jo] of Dominic Joyce for the fundamentals of $C^{\infty}$-Algebraic  Geometry.
Here we collect three handy lemmas
  concerning extensions of a $C^{\infty}$-ring structure or a $C^{\infty}$-ring-homomorphism,
 and their immediate corollaries on a $C^{\infty}$-scheme or a morphism.
 
\bigskip

\begin{lemma} {\bf [$C^{\infty}$ evaluation after nilpotent perturbation]}$\;$
 Given a $C^{\infty}$-ring $R$, let
   $r_1,\,\cdots\,,\, r_k \in R$   and
   $n_1,\,\cdots\,,\, n_k$ be nilpotent elements in $R$
      with $n_1^{l+1}=\,\cdots\,= n_k^{l+1}=0$.
 Then,
      for any $h\in C^{\infty}({\Bbb R}^k)$,
   the element $h(r_1+n_1,\,\cdots\,,\,r_k+n_k)\in R$
      from the $C^{\infty}$-ring structure of $R$ is given explicitly by
   $$
     h(r_1+n_1,\,\cdots\,,\,r_k+n_k)\;
	   =\;
	    \sum_{d=0}^{kl}\, \frac{1}{d!}\,
	       \sum_{d_1+\,\cdots\,+d_k=d}
	       (\partial_1^{\,d_1}\,\cdots\,\partial_k^{\,d_k}  h)(r_1,\,\cdots\,,\, r_k)
		     \cdot n_1^{d_1}\,\cdots\,n_k^{d_k}\,,	
   $$
   where
     $\partial_1^{\,d_1}\,\cdots\,\partial_n^{\,d_n}  h \in C^{\infty}({\Bbb R}^k)$
      is the partial derivative of $h$ with respect to the first variable $d_1$-times, the second variable $d_2$-times,
	  ..., and the $k$-th variable $d_k$-times.
\end{lemma}

\begin{proof}
 This is an immediate consequence of the fact that being a $C^{\infty}$-ring, $R$ is commutative and
  the Taylor's Formula with Remainder in the following form:
  \begin{itemize}
   \item[\LARGE $\cdot$]
    For any $h\in C^{\infty}({\Bbb R}^k)$ and $l\in {\Bbb Z}_{\ge 1}$,
	 there exist $\bar{h}_{(d_1,\,\cdots\,,\, d_k)}\in C^{\infty}({\Bbb R}^{2k})$,
       $d_i\in {\Bbb Z}_{\ge 0}$ with $d_1+\,\cdots\,+d_k=kl+1$,  	
	 such that
	  {\footnotesize
	  \begin{eqnarray*}
	   \lefteqn{h(y^1,\,\cdots\,,\,y^k) \,-\, h(x^1,\,\cdots\,, x^k)   } \\
	   &&
	    =\; \sum_{d=1}^{kl}\, \frac{1}{d!}\,
	              \sum_{d_1+\,\cdots\,+d_k=d}
	              (\partial_1^{\,d_1}\,\cdots\,\partial_k^{\,d_k}  h)(x^1,\,\cdots\,,\, x^k)
		               \cdot (y^1-x^1)^{d_1}\,\cdots\,(y^k-x^k)^{d_k}   \\
        && \hspace{2em}
		        +\;  \frac{1}{(kl+1)!}
				        \sum_{d_1+\,\cdots\,+d_k=kl+1}
	                       \bar{h}_{(d_1,\,\cdots\,,\, d_k)}(x^1,\,\cdots\,,\, x^k, y^1,\,\cdots\,,\, y^k)
		                    \cdot(y^1-x^1)^{d_1}\,\cdots\,(y^k-x^k)^{d_k}\,.            		
      \end{eqnarray*}}
  \end{itemize}
\end{proof}

\bigskip

\begin{lemma} {\bf [extension of $C^{\infty}$-ring structure]}$\;$
 Let
  $R$ be a $C^{\infty}$-ring and
  $S= R\oplus N$ be a commutative ${\Bbb R}$-algebra
     with $N^{l+1}=0$ for some $l\in {\Bbb Z}_{\ge 1}$.
 Then, $S$ admits a unique $C^{\infty}$-ring structure
   such that
    both the built-in ring-monomorphism $R\hookrightarrow S$ and the built-in ring-epimorphism $S\rightaarrow R$
    are $C^{\infty}$-ring-homomorphisms.
\end{lemma}

\begin{proof}
 For $h\in C^{\infty}({\Bbb R}^k)$, $k\in {\Bbb Z}_{\ge 1}$,
   and $s_1,\, \cdots\,,\, s_k\in S$,
  define  $h(s_1,\,\cdots\,,\,s_k)$ by setting
  $$
    h(s_1,\,\cdots\,,\,s_k)\; :=\;
	  \sum_{d=0}^{kl}\, \frac{1}{d!}\,
	    \sum_{d_1+\,\cdots\,+d_k=d}
	     (\partial_1^{\,d_1}\,\cdots\,\partial_k^{\,d_k}  h)(r_1,\,\cdots\,,\, r_k)
		  \cdot n_1^{d_1}\,\cdots\,n_k^{d_k}\,,	
  $$
  where
    \begin{itemize}
	  \item[\LARGE $\cdot$]
       $s_i=r_i+n_i$, $i=1,\,\ldots\,,k$,
        is the decomposition of $s_i\in S$ according to $S=R \oplus N$;
		
	  \item[\LARGE $\cdot$]	
	  	$\partial_1^{\,d_1}\,\cdots\,\partial_n^{\,d_n}  h \in C^{\infty}({\Bbb R}^k)$
       is the partial derivative of $h$ with respect to the first variable $d_1$-times, the second variable $d_2$-times,
	    ..., and the $k$-th variable $d_k$-times;
		
	  \item[\LARGE $\cdot$]	
	   $(\partial_1^{\,d_1}\,\cdots\,\partial_k^{\,d_k}  h)(r_1,\,\cdots\,,\, r_k)\in R$
	    from the $C^{\infty}$-ring structure on $R$.
	\end{itemize}	
 This defines a $C^{\infty}$-ring structure on $S$ via the $C^{\infty}$-ring structure on $R$.
 By construction,
   both the built-in ring-homomorphisms $R\hookrightarrow S$ and $S\rightaarrow R$
   are $C^{\infty}$-ring-homomorphisms with respect to this $C^{\infty}$-ring structure on $S$.
 
 Uniqueness of this $C^{\infty}$-ring structure with the required property follows from Lemma~1.2.1.
   
 This completes the proof.

\end{proof}

\bigskip

\begin{lemma} {\bf [extension of $C^{\infty}$-ring-homomorphism]}$\;$
 Let $S_1=R_1\oplus N_1$ and $S_2= R_2\oplus N_2$ be commutative ${\Bbb R}$-algebras
   such that $R_1$ and $R_2$  are $C^{\infty}$-rings and $N_1$ and $N_2$ are nilpotent,
     with $N_1^{l_1+1}=0$ and $N_2^{l_2+1}=0$.
 It follows from Lemma~1.2.2
  that $S_1$ and $S_2$ are endowed canonically with a $C^{\infty}$-ring structure
       that extends $R_1$ and $R_2$ respectively.
 Let
   $f:R_1\rightarrow R_2$ be a $C^{\infty}$-ring-homomorphism.
 Then any ring-homomorphism $g:S_1\rightarrow S_2$ that extends $f$
   is a $C^{\infty}$-ring-homomorphism.
\end{lemma}

\begin{proof}
 For all
    $h\in C^{\infty}({\Bbb R}^k)$, $k\in {\Bbb  Z}_{\ge 1}$, and $a_1, \,\cdots\,,\, a_k\in S_1$,
  we need to show that
  $$
    g(h(a_1,\,\cdots\,,\, a_k))\;
	  =\;   h(g(a_1),\,\cdots\,,\, g(a_k))\,.
  $$

 To prove this,
   let $a_i=b_i+c_i$, $i=1,\,\cdots\,,\, k$
      be the decomposition of $a_i$ according to the decomposition $S_1=R_1\oplus N_1$.
 Then observe that
   $g(a_i)=g(b_i)+g(c_i)=f(b_i)+g(c_i)$
      with $g(c_i)$ nilpotent: $(g(c_i))^{l_1+1}=0$ ,
    for $i=1,\,\cdots\,,\, k$,
   since $g$ is a ring-homomorphism extension of $f$.
 It follows from Lemma~1.2.1 that
  %
  \begin{eqnarray*}
    \lefteqn{
    g(h(a_1,\,\cdots\,,\, a_k))\;
	  =\;  g(h(b_1+c_1,\,\cdots\,,\, b_k+c_k)) }\\
     &&
	  =\;  g\left(\rule{0ex}{1.2em}\right.
	             \sum_{d=0}^{kl}\, \frac{1}{d!}\,
	               \sum_{d_1+\,\cdots\,+d_k=d}
	                (\partial_1^{\,d_1}\,\cdots\,\partial_k^{\,d_k}  h)(b_1,\,\cdots\,,\, b_k)
		               \cdot c_1^{d_1}\,\cdots\,c_k^{d_k}
	            \left.\rule{0ex}{1.2em}\right)   \\
     &&
	  =\;   \sum_{d=0}^{kl}\, \frac{1}{d!}\,
	               \sum_{d_1+\,\cdots\,+d_k=d}
	                f \left(
					   (\partial_1^{\,d_1}\,\cdots\,\partial_k^{\,d_k}  h)(b_1,\,\cdots\,,\, b_k)
					  \right)
		               \cdot   g(c_1)^{d_1}\,\cdots\, g(c_k)^{d_k}    \\
     &&
	  =\;   \sum_{d=0}^{kl}\, \frac{1}{d!}\,
	               \sum_{d_1+\,\cdots\,+d_k=d}
					   (\partial_1^{\,d_1}\,\cdots\,\partial_k^{\,d_k}  h)
					        (f(b_1),\,\cdots\,,\, f(b_k))
		               \cdot   g(c_1)^{d_1}\,\cdots\, g(c_k)^{d_k}    \\					
     &&
	   =\;   h(f(b_1)+g(c_1) ,\,\cdots\,,\, f(b_k)+ g(c_k))\;
	   =\;   h(g(a_1),\,\cdots\,,\, g(a_k))\,.
  \end{eqnarray*}

 This completes the proof.

\end{proof}

\bigskip

Passing from local to global via gluing gives the following corollaries for $C^{\infty}$-schemes:

\bigskip

\begin{corollary} {\bf [extension of $C^{\infty}$-scheme]}$\;$
 Let
   $X$ be a $C^{\infty}$-scheme and
   $X \subset \check{X}$ be an inclusion of locally ringed spaces,
      given by a nilpotent ideal sheaf ${\cal N}\subset {\cal O}_{\check{X}}$.
 Suppose that 	
	the short exact sequence
	   $$
	      0\;\longrightarrow\; {\cal N}\; \longrightarrow\; {\cal O}_{\check{X}}\;
		       \longrightarrow\; {\cal O}_X\; \longrightarrow\; 0
	   $$
	 has a built-in splitting ${\cal O}_X\subset {\cal O}_{\check{X}}$
	   as sheaves of ${\Bbb R}$-algebras on the same topological space underlying both $X$ and $\check{X}$.
 Then, $\check{X}$ admits canonically a $C^{\infty}$-scheme structure
   such that both the built-in inclusion $X\hookrightarrow \check{X}$ and
     the built-in dominant morphism $\check{X}\rightarrow X$ are morphisms of $C^{\infty}$-schemes.
\end{corollary}

\bigskip

\begin{corollary} {\bf [extension of morphism between $C^{\infty}$-schemes]}$\;$
 Let
   $X$ and $Y$ be $C^{\infty}$-schemes,
   $X\subset \check{X}$ and $Y\subset\check{Y}$ be as in Corollary~1.2.4
      from some built-in nilpotent split-exact extension of structure sheaves,
   and $f:X\rightarrow Y$ be a morphism of $C^{\infty}$-schemes.
 Then any morphism $\check{f}:\check{X}\rightarrow \check{Y}$ of locally ringed spaces
  that extends $f$ is a morphism of $C^{\infty}$-schemes.
\end{corollary}

\bigskip
 
\subsection{Basics of super algebraic geometry}

The most basic notions in super algebraic geometry for the current notes and their sequel are collected here.
(Some of the settings are more general than [L-Y3: Sec.\ 2.1] (D(11.2).).
Readers are referred to the thesis `{\sl Superrings and supergroups}' [Westra] of Dennis Westra
 for further details and the foundation toward super-algebraic geometry in line with Grothendieck's Algebraic Geometry.
 
\bigskip
 
\begin{definition}{\bf [superring].} {\rm
 A {\it superring} $A$  is a ${\Bbb Z}/2$-graded
    ${\Bbb Z}/2$-commutative (unital associative) ring
  $A=A_0\oplus A_1$ (also denoted $A_{\even}\oplus A_{\odd}$)
  such that the multiplication $A\times A \rightarrow A$ satisfies
  $$
   \begin{array}{llll}
    \mbox{({\it ${\Bbb Z}/2$-graded})}            &&&
      A_0A_0\; \subset\;  A_0\,, \hspace{1em}
	  A_0A_1\; =\; A_1A_0\; \subset\;  A_1\,, \hspace{1em}\mbox{and}\hspace{1em}
	  A_1A_1\; \subset\;  A_0\,, \\[.6ex]
	\mbox{({\it ${\Bbb Z}/2$-commutative})}  &&&
	  aa^{\prime}\; =\; (-1)^{ii^{\prime}}a^{\prime}a
	  \hspace{2em}\mbox{for $a\in A_i$ and $a^{\prime}\in A_{i^{\prime}}$,
	                                                  $i, i^{\prime}=0,\,1\,$.}	
   \end{array}	
  $$
 A {\it morphism} between superrings (i.e.\ {\it superring-homomorphism})
  is a ${\Bbb Z}/2$-grading-preserving ring-homomorphism of the underlying unital associative rings.
 The elements of $A_0$ are called {\it even}, the elements of $A_1$ are called {\it odd},
 and an element that is either even or odd is said to be {\it homogeneous}.
 For a homogeneous element $a\in A$,
  denote by $|a|$ the {\it ${\Bbb Z}/2$-degree} or {\it parity} of $a$;
  $|a|=i$ if $a\in A_i$, for $i=0,\,1$.
 
 An {\it ideal} $I$ of $A$ is said to be {\it ${\Bbb Z}/2$-graded}
  if $I=(I\cap A_0)+(I\cap A_1)$.
 In this case, $A$ induces a superring structure on the quotient ring $A/I$,
   with the ${\Bbb Z}/2$-grading given by\\
   $A/I=(A_0/(I\cap A_0)) \oplus(A_1/(I\cap A_1))  $.
 The converse is also true; cf.\ Definition/Lemma~1.3.3.
}\end{definition}

\bigskip

\begin{example-definition}
{\bf [superpolynomial ring/Grassmann algebra/}\\ {\bf exterior algebra]}$\;$ {\rm
 Let $R$ be a commutative ring.
 Then the following super $R$-algebra
   $$
     R[\theta^1,\,\cdots\,,\,\theta^s]^{\anticommuting}\;
	 :=\;   \frac{R\langle \theta^1,\,\cdots\,,\,\theta^s \rangle}
	                   {( r\theta^{\mu}-\theta^{\mu}r\,,\;
			                    \theta^{\mu}\theta^{\nu}+\theta^{\nu}\theta^{\mu}\,
				                |\; r \in R\,;\;    \mu, \nu =1,\,\ldots\,,\, s)}
   $$
    is called a {\it superpolynomial ring over $R$} {\it with $s$ anticommuting variables}
	(synonymously, {\it Grassmann $R$-algebra with $s$ generators},
	                               {\it exterior $R$-algebra with $s$ generators}).
 Here,
  \begin{itemize}
    \item[\LARGE $\cdot$]
     $R\langle \theta^1,\,\cdots\,,\,\theta^s \rangle$	
      is the noncommutative $R$-algebra freely generated by $\theta^1,\,\cdots\,,\, \theta^s$,
	
    \item[\LARGE $\cdot$]
     $( r\theta^{\mu}-\theta^{\mu}r\,,\;
			                    \theta^{\mu}\theta^{\nu}+\theta^{\nu}\theta^{\mu}\,
				                |\; r \in R\,;\;    \mu, \nu =1,\,\ldots\,,\, s)$
      is the bi-ideal in $R\langle \theta^1,\,\cdots\,,\,\theta^s \rangle$
      generated by elements indicated.	
  \end{itemize}
 The ${\Bbb Z}/2$-grading of $R[\theta^1,\,\cdots\,,\,\theta^s]^{\anticommuting}$
   is given by assigning
     elements of $R$ even, $\theta^1, \,\cdots\,,\,\theta^s$ odd,
	 and the product rule or even or odd homogeneous elements.
	
 The even component $R[\theta^1,\,\cdots\,,,\, \theta^s]^{\anticommuting}_{\even}$
         of $R[\theta^1,\,\cdots\,,,\, \theta^s]^{\anticommuting}$
     consists of polynomials in $\theta^1,\,\cdots\,,\, \theta^s$ with coefficients in $R$
	   whose monomial summands are all of even total $(\theta^1,\,\cdots\,,\, \theta^s)$-degree.
 The odd component $R[\theta^1,\,\cdots\,,,\, \theta^s]^{\anticommuting}_{\odd}$
         of $R[\theta^1,\,\cdots\,,,\, \theta^s]^{\anticommuting}$
     consists of polynomials in $\theta^1,\,\cdots\,,\, \theta^s$ with coefficients in $R$
	   whose monomial summands are all of odd total $(\theta^1,\,\cdots\,,\, \theta^s)$-degree.	
}\end{example-definition}

\bigskip

By construction, there are built-in ring-inclusion $\iota_{A_0}$ and  ring-epimorphism $\pi_{A_0}$
 $$
  \xymatrix{
    A_0\; \ar@{^{(}->}[r]^-{\iota_{A_0}}
	 & \;A\; \ar@{->>}[r]^-{\pi_{A_0}}    & \;A_0
   }
 $$
 with $\pi_{A_0}\circ \iota_{A_0}=\Id_{A_0}$ the identity map on $A_0$.

\bigskip

\begin{definition-lemma}{\bf [${\Bbb Z}/2$-graded ideal = supernormal ideal].}
{\rm
   An ideal $I$ of a superring $A$ is called {\it supernormal}
    if $A$ induces a superring structure on the quotient ring $A/I$.}
 In terms of this, $I$ is ${\Bbb Z}/2$-graded if and only if $I$ is supernormal.	
\end{definition-lemma}
 
\bigskip

\begin{definition}{\bf [module over superring].} {\rm
 Let $A$ be a superring.
 An {\it $A$-module} $M$ is a {\it left module} over the unital associative ring underlying $A$
   that is endowed with a {\it ${\Bbb Z}/2$-grading} $M=M_0\oplus M_1$
    such that
     $$
	   A_0M_0\; \subset\; M_0\,, \hspace{1em}
	   A_1M_0\; \subset\; M_1\,, \hspace{1em}
	   A_0M_1\; \subset\; M_1\,, \hspace{1em}\mbox{and}\hspace{1em}
	   A_1M_1\;\subset\; M_0\,.
	 $$
 The elements of $M_0$ are called {\it even}, the elements of $M_1$ are called {\it odd},
   and an element that is either even or odd is said to be {\it homogeneous}.
 For a homogeneous element $m\in M$,
  denote by $|m|$ the {\it ${\Bbb Z}/2$-degree} or {\it parity} of $m$;
  $|m|=i$ if $m\in M_i$, for $i=0,\,1$.	
		
 For a superring $A$,
  \begin{itemize}
   \item[{\Large $\cdot$}]
 {\it a left $A$-module is canonically a right $A$-module}
    by setting
      $ma := (-1)^{|m||a|}am  $ for homogeneous elements $a\in A$ and $m\in M$
       and then extending ${\Bbb Z}$-linearly to all elements.
  \end{itemize}	
 For that reason, as in the case of commutative rings and modules,
  we don't  distinguish a left-, right-, or bi-module for a module over a superring.
  
 A {\it morphism} (or {\it module-homomorphism})
     $h:M\rightarrow M^{\prime}$ between $A$-modules
   is a right-module-homomorphism between the right-module over the unital associative ring underlying $A$;
 or equivalently
   a left-module-homomorphism between the left-module over the unital associative ring underlying $A$
   but with the sign rule applied to homogeneous components of $h$ and homogeneous elements of $A$.
 Explicitly,
 $h$ is said to be
  {\it even} if it preserves the ${\Bbb Z}/2$-grading   or
  {\it odd}  if it switches the ${\Bbb Z}/2$-grading;
 decompose $h$ to $h=h_0+h_1$  a summation of even and odd components,
 then $h_i(am)=(-1)^{i|a|}a h_i(m)$, $i=0,\,1$, for $a\in A$ homogeneous and $m\in M$.
 
 Subject to the above sign rules when applicable, the notion of
 
      \medskip
	
	{\Large $\cdot$}
       {\it submodule} $M^{\prime} \hookrightarrow M$, (cf.\ {\it monomorphism}),
	
	{\Large $\cdot$}
       {\it quotient module}  $M \twoheadrightarrow M^{\prime}$, (cf.\ {\it epimorphism}), 	
	
	{\Large $\cdot$}
	   {\it direct sum}  $M\oplus M^{\prime}$ of $A$-modules,
	
    {\Large $\cdot$}
      {\it tensor product} $M\otimes_A M^{\prime}$ of $A$-modules,
	
    {\Large $\cdot$}
	    {\it finitely generated}:
    		if $A^{\oplus l} \twoheadrightarrow M$
		     exists for some $l$,
			
    {\Large $\cdot$}
        {\it finitely presented}:
           if $A^{\oplus l^{\prime}}
		          \rightarrow   A^{\oplus l} \rightarrow M \rightarrow 0$
              is exact for some $l$, $l^{\prime}$				
	
       \medskip
	
	\noindent
    are all defined in the ordinary way as in commutative algebra.
}\end{definition}

\bigskip

We introduce now a new notion that is unusual from the aspect of the category  of superrings.
However, as explained in Sec.$\:$1.1, such a notion is required when one wants to bring in the physics of supersymmetry,
 which exchanges the even part $A_0$ and the odd part $A_1$  of a superring $A=A_0\oplus A_1$
 (when $A$ is the function-ring of the super space-time in question):
 
\bigskip

\begin{definition} {\bf [general superring-homomorphism]}$\;$ {\rm
 A {\it general superring-homomorphism} $h:A\rightarrow B$
  is a homomorphism between the unital associative rings that underlie $A$ and $B$ respectively.
}\end{definition}

\bigskip

\noindent
A general superring-homomorphism may mingle, rather than preserve,
 the even part and the odd part of the superrings in question.

\bigskip

\begin{definition} {\bf [superscheme]} {\rm
 A {\it superscheme} is a locally ringed space
   $\widehat{X}:=(X,  {\cal O}_{\widehat{X}})$
  with the structure sheaf
   $$
     {\cal O}_{\widehat{X}}\;    =\; {\cal A}_0\oplus {\cal A}_1\;
	  =:\;   {\cal A}_{\even}\oplus {\cal A}_{\odd}
   $$
	a sheaf of superrings over the underlying topological space $X$ such that
   \begin{itemize}
    \item[(1)]
     The locally ringed space $\widehat{X}_{\even}:=(X,{\cal A}_{\even})$
	  defines a scheme in the sense of Algebraic Geometry as in [E-H] and [Ha].
	
	\item[(2)]
	 The odd component ${\cal A}_{\odd}$ of ${\cal O}_{\widehat{X}}$
	  is a coherent sheaf of ${\cal A}_{\even}$-modules on $\widehat{X}_{\even}$.
   \end{itemize}
  For convenience, we shall denote ${\cal O}_{\widehat{X}}$ also by $\widehat{\cal O}_X$.
  When in need, a usual scheme is regarded as a superscheme whose structure sheaf has zero odd-component.

 A {\it morphism} of superschemes from $(X, \widehat{\cal O}_X)$ to $(Y,\widehat{\cal O}_Y)$
     is a pair $(f, f^{\sharp})$ of
   a continuous map $f:X\rightarrow Y $ and
   a map $f^{\sharp}: \widehat{\cal O}_Y \rightarrow f_{\ast}\widehat{\cal O}_X$
     of sheaves of superrings that preserves the ${\Bbb Z}/2$-grading
	 and induces local homomorphisms
	  $f^{\sharp}_p: \widehat{\cal O}_{Y, f(p)}\rightarrow \widehat{\cal O}_{X,p}$
	  of local rings for all $p\in X$.

 A {\it general morphism} of superschemes
        from $(X, \widehat{\cal O}_X)$ to $(Y,\widehat{\cal O}_Y)$
     is a pair $(f, f^{\sharp})$ as above except that
	 $f^{\sharp}: \widehat{\cal O}_Y \rightarrow f_{\ast}\widehat{\cal O}_X$
       is only a map of sheaves of rings that underlie the superrings.
  $f^{\sharp}$ may not preserve the ${\Bbb Z}/2$-grading.   	
}\end{definition}
 
\bigskip

By construction, the short exact sequence of $\widehat{\cal O}_X$-modules
 $$
   0\; \longrightarrow\; {\cal A}_{\odd}\; \longrightarrow\; \widehat{\cal O}_X\;
         \longrightarrow\; {\cal A _{\even}}\; \longrightarrow\; 0
 $$
 is canonically split by the built-in inclusion ${\cal A}_{\even}\subset \widehat{\cal O}_X$.
This defines built-in morphisms
 $$
  \xymatrix{
    \widehat{X}_{\even}\; \ar@{^{(}->}[r]^-{\iota_{\widehat{X}_{\tinyeven}}}
	 & \;\widehat{X}\; \ar@{->>}[r]^-{\pi_{\widehat{X}_{\tinyeven}}}
	 & \;\widehat{X}_{\even}
   }
 $$
 of superschemes
 with $\pi_{\widehat{X}_{\tinyeven}}\circ \iota_{\widehat{X}_{\tinyeven}}
           =\Id_{\widehat{X}_{\tinyeven}}$ the identity map on $\widehat{X}_{\even}$.

\bigskip

\subsection{Super $C^{\infty}$-rings, super $C^{\infty}$-schemes, and general morphisms}

Closely related to and enforced by the notion of general superring-homomorphisms,
a few basic definitions concerning super $C^{\infty}$-rings and  super $C^{\infty}$-schemes are reset here.
They are more general than those introduced in [L-Y3] (D(11.2))
   and are required on the string-theory side for the problem of constructing
   fundamental fermionic stacked D-branes along the line of the Ramond-Neveu-Schwarz formulation of fermionic strings
	 --- a theme to be addressed in another subseries of the D-project.

\bigskip	

\begin{definition} {\bf [super $C^{\infty}$-ring]}$\;$ {\rm
 A {\it super $C^{\infty}$-ring} is a superring $A=A_0\oplus A_1 =: A_{\even}\oplus A_{\odd}$
    in the sense of Definition~1.3.1 such that $A_{\even}$ is equipped with a  $C^{\infty}$-ring structure
	(cf.$\:$[Jo]).             
 An {\it ideal} of a super $C^{\infty}$-ring $A$ is an ideal of the underlyung ring.
 A (left, right, or bi-){\it module} of $A$ is a (resp.\ left-, right-, bi-)module of the underlying superring.
}\end{definition}	

\bigskip

\begin{definition} {\bf [super-$C^{\infty}$-ring-homomorphism]}$\;$ {\rm
 Let  $A=A_0\oplus A_1$, $B=B_0\oplus B_1$ be super $C^{\infty}$-rings.
 A {\it super-$C^{\infty}$-ring-homomorphism} $h:A\rightarrow B$
  is a ring-homomorphism of the underlying unital associative rings such that
   \begin{itemize}
    \item[(1)]
    $h$ preserves the ${\Bbb Z}/2$-grading:
       $h(A_0)\subset B_0$ and $h(A_1)\subset B_1$;
	
	\item[(2)]
     the restriction
      $h|_{A_0}:A_0\rightarrow B_0$	
	  is a $C^{\infty}$-ring-homomorphism.
   \end{itemize}	
}\end{definition}	

\bigskip

\begin{definition} {\bf [general super-$C^{\infty}$-ring-homomorphism]}$\;$ {\rm
 Let  $A=A_0\oplus A_1$,\\ $B=B_0\oplus B_1$ be super $C^{\infty}$-rings.
 A {\it general super-$C^{\infty}$-ring-homomorphism} $h:A\rightarrow B$
  is a ring-homomorphism of the underlying unital associative rings such that
  the induced ring-homomorphism
   $$
      \pi_{B_0}\circ h \circ \iota_{A_0}\;:\;  A_0\; \longrightarrow\; B_0
   $$
 is a $C^{\infty}$-ring-homomorphism.
}\end{definition}

\bigskip

\noindent
By definition, super-$C^{\infty}$-ring-homomorphisms are
 special examples of general super-$C^{\infty}$-ring-homomorphisms.

\bigskip
 
The above setting passes contravariantly to affine super $C^{\infty}$-schemes,
 and then to super $C^{\infty}$-schemes via gluing.
	
\bigskip

\begin{definition} {\bf [super $C^{\infty}$-scheme]}$\;$ {\rm
 A {\it super $C^{\infty}$-scheme} is a locally ringed space\\
   $\widehat{X}:=(X,  {\cal O}_{\widehat{X}})$
  with the structure sheaf
   $$
     {\cal O}_{\widehat{X}}\;    =\; {\cal A}_0\oplus {\cal A}_1\;
	  =:\;   {\cal A}_{\even}\oplus {\cal A}_{\odd}
   $$
	a sheaf of superrings over a topological space $X$ such that
   \begin{itemize}
    \item[(1)]
     The locally ringed space $\widehat{X}_{\even}:=(X,{\cal A}_{\even})$
	  defines a $C^{\infty}$ scheme in the sense of $C^{\infty}$-Algebraic Geometry as in [Jo].
	
	\item[(2)]
	 The odd component ${\cal A}_{\odd}$ of ${\cal O}_{\widehat{X}}$
	  is a finitely presented ${\cal A}_{\even}$-module on $\widehat{X}_{\even}$.
   \end{itemize}
  For convenience, we shall denote ${\cal O}_{\widehat{X}}$ also by $\widehat{\cal O}_X$.
  When in need, a usual $C^{\infty}$-scheme is regarded as a super $C^{\infty}$-scheme
  whose structure sheaf has zero odd-component.
}\end{definition}

\bigskip

\begin{example} {\bf [from spinor bundle to super $C^{\infty}$-scheme]}$\;$ {\rm
 Let
   $M$ be a Riemannian or Lorentzian $C^{\infty}$-manifold,
   ${\cal O}_M$ be its structure sheaf of smooth functions,
   $S$ be a spinor bundle or a direct sum of spinor bundles over $M$ of total rank $s$,
   $S^{\vee}$ be the dual bundle of $S$,
   ${\cal S}$ be the sheaf of $C^{\infty}$-sections of $S$, and
   ${\cal S}^{\vee}:= \Homsheaf_{{\cal O}_M}({\cal S},{\cal O}_M)$
      be the dual sheaf of ${\cal S}$ or equivalently the sheaf of $C^{\infty}$-sections of $S^{\vee}$.
 Then
  $$
    \mbox{$\bigwedge$}^{\bullet}_{{\cal O}_M}{\cal S}^{\vee}\;
	:=\;  \oplus_{l=0}^s\, \mbox{$\bigwedge$}^l_{{\cal O}_M}{\cal S}^{\vee}\,,
  $$
  where
    $\bigwedge^0_{{\cal O}_M}{\cal S}^{\vee}:={\cal O}_M$ by convention and
    $\bigwedge^l_{{\cal O}_M}$ is the exterior tensor product of degree $l$ of an ${\cal O}_M$-module,
  is a sheaf of ${\Bbb Z}/2$-graded ${\Bbb Z}/2$-commutative ${\cal  O}_M$-algebras, with
   $$
     \begin{array}{cl}
	 & (\mbox{$\bigwedge$}^{\bullet}_{{\cal O}_M}{\cal S}^{\vee})_{\even}\;
                :=\;\mbox{$\bigwedge$}^{\even}_{{\cal O}_M}{\cal S}^{\vee}\;		
	           :=\;  {\cal O}_M
			              \oplus \mbox{$\bigwedge$}^2_{{\cal O}_M}{\cal S}^{\vee} \oplus \cdots\,  
           \hspace{3em}			   \\[1.2ex]
	 \mbox{and}\hspace{2em}	
	 & (\mbox{$\bigwedge$}^{\bullet}_{{\cal O}_M}{\cal S}^{\vee})_{\odd}\;\;
                :=\;\mbox{$\bigwedge$}^{\odd}_{{\cal O}_M}{\cal S}^{\vee}\;\;		
	           :=\;  {\cal S}^{\vee}
			             \oplus \mbox{$\bigwedge$}^3_{{\cal O}_M}{\cal S}^{\vee} \oplus \cdots\,.  	
	 \end{array}
   $$
 $\mbox{$\bigwedge$}^{\bullet}_{{\cal O}_M}{\cal S}^{\vee}$
   is called the {\it sheaf of Grassmann algebras} (synonymously, {\it sheaf of exterior algebras})
   associated to $S$.
 It follows from Lemma~1.2.2 that
  $(\mbox{$\bigwedge$}^{\bullet}_{{\cal O}_M}{\cal S}^{\vee})_{\even}$
   is canonically a sheaf of $C^{\infty}$-${\cal O}_M$-algebras,
    making
     $(M,(\mbox{$\bigwedge$}^{\bullet}_{{\cal O}_M}{\cal S}^{\vee})_{\even})$
	 a $C^{\infty}$-scheme.
 It follows that the locally ringed space
  $$
    \widehat{M}\; :=\;  (M, \mbox{$\bigwedge$}^{\bullet}_{{\cal O}_M}{\cal S}^{\vee})
  $$
  is a super $C^{\infty}$-scheme.
 See Remark~2.1.3 for further explanations.
  
 Explicitly, over an open set $U\subset M$ over which $S|_U$ is trivial.
 Let $\theta^1,\,\cdots\,,\,\theta^s$ be a basis of $S^{\vee}|_U$.
 Then
  $$
    (\mbox{$\bigwedge$}^{\bullet}_{{\cal O}_M}{\cal S}^{\vee})(U)\;
	 =\;  C^{\infty}(U)[\theta^1,\,\cdots\,,\,\theta^s]^{\anticommuting}\,,
  $$
  that is, an algebraic extension of $C^{\infty}(U)$ by variables $\theta^1,\,\cdots\,,\theta^s$,
   subject to the relations
  $$
     h \theta^{\mu}\;=\; \theta^{\mu} h\,,\hspace{2em}
	 \theta^{\mu}\theta^{\nu}\;=\; - \theta^{\nu}\theta^{\mu}\,,
  $$
  for all $h\in C^{\infty}(U)$ and $\mu,\nu =1,\,\ldots\,,\, s$.
 By construction, there are built-in morphisms of super $C^{\infty}$-schemes (cf.$\:$Definition~1.4.7 below)
 $$
  \xymatrix @R=.6ex{
    \widehat{M}_{\even}\; \ar@{^{(}->}[r]^-{\iota_{\widehat{M}_{\tinyeven}}}
	  & \;\widehat{M}\; \ar@{->>}[r]^-{\pi_{\widehat{M}_{\tinyeven}}}
	  & \;\widehat{M}_{\even}\,, \\
    M\;\;\;\;\; \ar@{^{(}->}[r]^-{\iota_M}
	  & \;\widehat{M}\; \ar@{->>}[r]^-{\pi_M}     & \;M\;\;\;\;\;, \\	
   }
 $$
 with
    $\pi_{\widehat{M}_{\tinyeven}}\circ \iota_{\widehat{M}_{\tinyeven}}
           =\Id_{\widehat{M}_{\tinyeven}}$ the identity map on $\widehat{M}_{\even}$   and
	 $\pi_M \circ \iota_M =\Id_M$ the identity map on $M$.						
  
 Note that the same construction works with $S$ replaced by any real vector bundle over $M$.
 The super $C^{\infty}$-scheme thus obtained will also be called by a friendlier name:
 {\it super $C^{\infty}$-manifold}.
 When there is a need to make the rank$_{\Bbb R}$ $s$ of $S$ explicit,
   we will denote $\widehat{M}$ also by $\widehat{M}_{[s]}$.
 The super $C^{\infty}$-scheme $\widehat{{\Bbb R}^n}_{[s]}$,
    which has the underlying topology ${\Bbb R}^n$  and rank$_{\Bbb R}(S)=s$,
  is denoted also by ${\Bbb R}^{n|s}$ conventionally.
}\end{example}

\bigskip

\begin{example} {\bf [product of super $C^{\infty}$-manifolds]}$\;$ {\rm
 Let
   $\widehat{X}_{[s_1]}
      =(X, \widehat{\cal O}_X:=\bigwedge^{\bullet}_{{\cal O}_M}{\cal S}_1^{\vee})$
     and
   $\widehat{Y}_{[s_s]}
      =(Y, \widehat{\cal O}_Y:=\bigwedge^{\bullet}_{{\cal O}_M}{\cal S}_2^{\vee})$
  be supermanifolds as in Example~1.4.5
   but with $S_1$ (resp.$\:S_2$) some arbitrary real vector bundle over $X$ (resp.\ $Y$)
    of rank$_{\Bbb R}$ $s_1$ (resp.$\:s_2$).
 Then,
   the {\it product} $\widehat{X}_{[s_1]}\times \widehat{Y}_{[s_2]}$
     of super $C^{\infty}$-manifolds $\widehat{X}_{[s_1]}$ and $\widehat{Y}_{[s_2]}$ is
	 a super $C^{\infty}$-manifold with the underlying topology $X\times Y$ and the structure sheaf
	 $$
	   \widehat{\cal O}_{X\times Y}\;
	    :=\;  \pr_X^{\ast}\widehat{\cal O}_X
	              \otimes_{\,{\cal O}_{X\times Y}} \pr_Y^{\ast}\widehat{\cal O}_Y\;
				  =:\;  \widehat{\cal O}_X
	                         \boxtimes_{\,{\cal O}_{X\times Y}} \widehat{\cal O}_Y\,,	
	 $$
    where
	  $\pr_X:X\times Y \rightarrow X$ and $\pr_Y:X\times Y\rightarrow  Y$ are projection maps   and
	  the ${\cal O}_{X\times Y}$-algebra structure on $\widehat{\cal O}_{X\times Y}$
        comes from the tensor product of
		  ${\Bbb Z}/2$-graded ${\Bbb R}/2$-commutative ${\cal O}_{X\times Y}$-algebras
	       $\pr_X^{\ast}\widehat{\cal O}_X$ and $\pr_Y^{\ast}\widehat{\cal O}_Y$.
 Explicitly,
   for open sets $U\subset X$ and $V\subset Y$ such that
   $$
     \widehat{\cal O}_X(U)\;=\; C^{\infty}(U)[\theta^1,\,\cdots\,,\theta^{s_1}]^{\anticommuting}
	   \hspace{2em}\mbox{and}\hspace{2em}	
	 \widehat{\cal O}_Y(V)\;
	   =\; C^{\infty}(V)[\vartheta^1,\,\cdots\,,\vartheta^{s_2}]^{\anticommuting}\,,
   $$
    for a choice of generating sections $\theta^1,\,\cdots\,,\theta^{s_1}$ of ${\cal S}^{\vee}|_U$
	  and $\vartheta^1,\,\cdots\,,\,\vartheta^{s_2}$ of ${\cal S}_2^{\vee}|_V$,
 one has
   \begin{eqnarray*}
    \widehat{\cal O}_{X\times Y}(U\times V)
	  & = &  C^\infty(U\times V)[\theta^1,\,\cdots\,,\, \theta^{s_1}\,;\,
	                 \vartheta^1,\,\cdots\,,\, \vartheta^{s_2}]^{\anticommuting}\\[.6ex]      					 
      & :=\: &   C^\infty(U\times V)[\theta^1,\,\cdots\,,\, \theta^{s_1}]^{\anticommuting}
	                    \otimes_{\,C^{\infty}(U\times V)}					
					   C^\infty(U\times V)[\vartheta^1,\,\cdots\,,\, \vartheta^{s_2}]^{\anticommuting}	  \\[.6ex]
      & = &  \frac{C^\infty(U\times V)
	                   \langle \theta^1,\,\cdots\,,\, \theta^{s_1}, \,\vartheta^1,\,\cdots\,,\, \vartheta^{s_2}\rangle}
	               {\left( \left.
				         \mbox{\small
						   $\begin{array}{l}
				             f \theta^\mu-\theta^{\mu}\!f\,,\;\;
                               \theta^\mu\theta^\nu+\theta^\nu\theta^\mu\,, \\[.6ex]		
                                 \hspace{2em}
								   f\vartheta^i-\vartheta^i f\,,\;\;
                                   \vartheta^i\vartheta^j+\vartheta^j\vartheta^i\,, \\[.6ex]
                                 \hspace{4em}							
							       \theta^\mu\vartheta^i-\vartheta^i\theta^\mu
						      \end{array}$}
						    \right|
							  \mbox{\small
							   $\begin{array}{l}
						           f\in C^\infty(U\times V),\\   \;\;\;  \mu,\nu =1,\,\ldots\,,\, s_1,    \\  \;\;\;\;\;\;  i,j=1,\,\cdots\,,\, s_2
							      \end{array}$}
					  \right)}\,.
   \end{eqnarray*}
  Here in the last line,
      the numerator is the free (associative, unital) $C^\infty(U\times V)$-algebra generated by
        $\theta^1,\,\cdots\,,\theta^{s_1},\, \vartheta^1,\,\cdots\,,\,\vartheta^{s_2}$ and
	  the denominator is the bi-ideal generated by the elements indicated.	
 In notation, 	
   $$
      \widehat{X}_{[s_1]}\times \widehat{Y}_{[s_2]}\;
	   =\;  \widehat{X\times Y}_{[s_1,s_2]}\;
	           :=\;  (X\times Y,  \widehat{\cal O}_{X\times Y})\,.
   $$
}\end{example}

\bigskip

\begin{definition} {\bf [morphism between super $C^{\infty}$-schemes]}$\;$ {\rm
 A {\it morphism} of super $C^{\infty}$-schemes
   from $(X, \widehat{\cal O}_X)$ to $(Y,\widehat{\cal O}_Y)$
     is a pair $(f, f^{\sharp})$ of
   a continuous map $f:X\rightarrow Y $ and
   a map $f^{\sharp}: \widehat{\cal O}_Y \rightarrow f_{\ast}\widehat{\cal O}_X$
     of sheaves of super $C^{\infty}$-rings that preserves the ${\Bbb Z}/2$-grading
	 and induces local homomorphisms
	  $f^{\sharp}_p: \widehat{\cal O}_{Y, f(p)}\rightarrow \widehat{\cal O}_{X,p}$
	  of local rings for all $p\in X$.
}\end{definition}

\bigskip

As in the case of superschemes, one has a built-in split exact sequence of $\widehat{\cal O}_X$-modules
$$
   0\; \longrightarrow\; {\cal A}_{\odd}\; \longrightarrow\; \widehat{\cal O}_X\;
         \longrightarrow\; {\cal A _{\even}}\; \longrightarrow\; 0
 $$
 and built-in morphisms
 $$
  \xymatrix{
    \widehat{X}_{\even}\; \ar@{^{(}->}[r]^-{\iota_{\widehat{X}_{\tinyeven}}}
	 & \;\widehat{X}\; \ar@{->>}[r]^-{\pi_{\widehat{X}_{\tinyeven}}}
	 & \;\widehat{X}_{\even}
   }
 $$
 of super $C^{\infty}$-schemes
 with $\pi_{\widehat{X}_{\tinyeven}}\circ \iota_{\widehat{X}_{\tinyeven}}
           =\Id_{\widehat{X}_{\tinyeven}}$ the identity map on $\widehat{X}_{\even}$. 		

\bigskip

\begin{definition} {\bf [general morphism between super $C^{\infty}$-schemes]}$\;$ {\rm
 A {\it general morphism} of super $C^{\infty}$-schemes
        from $(X, \widehat{\cal O}_X ={\cal A}_{\even}\oplus {\cal A}_{\odd})$
		to $(Y,\widehat{\cal O}_Y={\cal B}_{\even}\oplus {\cal B}_{\odd})$		
     is a pair $(f, f^{\sharp})$ as in Definition~1.4.7
  except that
	 $f^{\sharp}: \widehat{\cal O}_Y \rightarrow f_{\ast}\widehat{\cal O}_X$
       is only a map of sheaves of rings underlying the super  $C^{\infty}$-rings such that
	   the composition
	   $$
          \xymatrix{
            {\cal B}_{\even}\; \ar@{^{(}->}[r]^-{\iota_{{\cal B}_{\tinyeven}}}
	           & \;\widehat{\cal O}_Y\; \ar[r]^-{f^{\sharp}}
	           & \; f_{\ast}\widehat{\cal O}_X\; \ar[r]^-{ f_{\ast}\pi_{{\cal A}_{\tinyeven}}}
			   &\; f_{\ast}{\cal A}_{\even}
           }
       $$
	   is a map of sheaves of $C^{\infty}$-rings on $Y$. 	
  $f^{\sharp}$ may not preserve the ${\Bbb Z}/2$-grading.   	
}\end{definition}

\bigskip
 
\begin{remark}
$[$behind definition of morphisms and general morphisms for super ($C^{\infty}$-)schemes$\,]\;$ {\rm
 (1)
 A superscheme or super $C^{\infty}$-scheme can be defined as an equivalence class of gluing systems
        of rings as we did in [L-Y1] (D(1)) for Azumaya-type noncommutative spaces.
 Morphisms (resp.\ general morphisms) between superschemes or super $C^{\infty}$-schemes
  can thus be defined contravariantly as an equivalence class of gluing systems of superring-homomorphisms
   (resp.\ general superring-homomorphisms) as well.
 In ibidem, such a setting is a must in order to reveal the feature of D-branes under deformations
  (cf.$\:$ {\sc Figure}~2-3-2-1).
 For superschemes or super $C^{\infty}$-schemes, since the odd component of the structure is always nilpotent,
  the two settings are equivalent   and
 we present the one that is more standard-looking for ringed spaces.
   
 (2)
 Geometrically,
 for a map $f:\widehat{X}\rightarrow \widehat{Y}$ between super $C^{\infty}$-schemes
     that does not take $\widehat{X}_{\even}$ to $\widehat{Y}_{\even}$,
  we use its post-composition with $\widehat{Y}\rightarrow \widehat{Y}_{\even}$
    to retain a map
	   $\widehat{X}_{\even}\rightarrow \widehat{Y}_{\even}$ between $C^{\infty}$-schemes,
	 from which one determines whether $f$ should be considered as smooth.
 This is in the same spirit as for a map $g:X\rightarrow {\Bbb R}^n$,
   one uses the smoothness of all the compositions $\pr_j\circ g$,
     where $\pr_j:{\Bbb R}^n\rightarrow {\Bbb R}$ is the projection map to $j$-th component, $j=1,\,\cdots\,,\, n$,
   to define the smoothness of $g$.
 Passing to the graph of $g$ and expressed contravariantly in terms of function-rings,
  this gives us a hint as to how one should think of ``smoothness"
   for a map from an Azumaya super $C^{\infty}$-manifold to a super $C^{\infty}$-manifold
  (cf.\ Definition~2.1.4 and Definition~2.1.6).
  
 (3)
 See Sec.$\:$3.2 for a further remark on the notion of
   `{\sl $C^{\infty}$-admissible noncommutative rings} and {\sl morphisms} between them.		
 }\end{remark}

\bigskip

\begin{remark} $[$affine $C^{\infty}$ case$\,]\;$ {\rm
 The language of schemes and sheaves of modules helps one to think geometrically.
 However, since $C^{\infty}$-manifolds are affine $C^{\infty}$-schemes,
   we are in the situation of a noncommutative-but-algebraic-in-nature extension of affine $C^{\infty}$-geometry.
 Fundamental details still rely on the study of rings, ring-homomorphisms, and modules involved.
}\end{remark}

\bigskip

\section{A further study of the notion of smooth maps from an Azumaya/matrix supermanifold}

With the stringy motivation and mathematical background in Sec.$\:$1,
 we now turn to  the main subject of the notes:
 $C^{\infty}$-maps from an Azumaya/matrix super $C^{\infty}$-manifold with a fundamental module
 to a (real) super $C^{\infty}$-manifold.
This section is devoted to the proof of two main theorems (Theorem~2.1.5 \& Theorem~2.1.8)
  stated in Sec.$\:$2.1.

\bigskip

\subsection{The setup and the statement of two main theorems}

Denote by $\widehat{\Bbb R}_{[s]}$ the real Grassmann algebra with $s$ generators
 and $\widehat{\Bbb C}_{[s]}$ its complexification.
We may denote them respectively as $\widehat{\Bbb R}$ and $\widehat{\Bbb C}$
  when the number of generators can be kept implicit.

\bigskip

\begin{flushleft}
{\bf Azumaya/matrix super $C^{\infty}$-manifolds with a fundamental module}
\end{flushleft}
\begin{definition} {\bf [Azumaya/matrix super $C^{\infty}$-manifold with a fundamental module] }$\;$ {\rm
 Let
   \begin{itemize}
    \item[(1)] (Cf.$\;${\it topology $X$ underlying D-brane world-volume})\\[.6ex]
	 $X$ be a $C^{\infty}$-manifold (of some dimension $m$),
	  regarded also as
	   a $C^{\infty}$-scheme with the structure sheaf ${\cal O}_X$ of $C^{\infty}$-functions on $X$;
	 $\;{\cal O}_X^{\,\Bbb C}:= {\cal O}_X\otimes_{\Bbb R}{\Bbb C}$
	   the complexification of ${\cal O}_X$;
     	        	
    \item[(2)] (Cf.$\:${\it Chan-Paton bundle $E$ on D-brane world-volume $X$})\\[.6ex]
     $E$ be a complex smooth vector bundle over $X$ (of some rank$_{\Bbb C}$ $r$);\\
	 ${\cal E}$ be the sheaf of smooth sections of $E$,
	  which is naturally an ${\cal O}_X^{\,\Bbb C}$-module;
	
    \item[\LARGE $\cdot$] ({\it Azumaya/matrix-type noncommutative structure on $X$ associated to $E$})\\[.6ex]
	 $\End_{\Bbb C}(E)$ be the endomorphism bundle of $E$,
  	  which is isomorphic to $E\otimes_{\Bbb C}E^{\vee}$ canonically
	   (here, $E^{\vee}=$ the dual complex vector bundle of $E$);\\
     $\Endsheaf_{{\cal O}_X^{\,\Bbb C}}({\cal E})
	     \simeq {\cal E}\otimes_{{\cal O}_X^{\,\Bbb C}}{\cal E}^{\vee}$
	   be the sheaf of endomorphisms	of ${\cal E}$ as a ${\cal O}_X^{\,\Bbb C}$-module.
   \end{itemize}	
 Recall that these data define an {\it Azumaya/matrix $C^{\infty}$-manifold with a fundamental module}
  $$
    (X^{\!A\!z},{\cal E})\;
	  :=\; (X, {\cal O}_X^{A\!z}:= \Endsheaf_{{\cal O}_X^{\,\Bbb C}}({\cal E}), {\cal E})
  $$
  as a ringed space
     with the underlying topology $X$ and the structure sheaf
             the sheaf $\Endsheaf_{{\cal O}_X^{\,\Bbb C}}({\cal E})$ of
			  endomorphism algebras of ${\cal E}$.
 ([L-Y2] (D(11.1))).
 The noncommutative structure sheaf ${\cal O}_X^{A\!z}$ acts on ${\cal E}$ from the left
  via the built-in fundamental representation.
  
 Let, in addition,
    \begin{itemize}
	\item[(3)] ({\it Superification; cf.$\:$spinor bundle $S$ to incorporate world-volume supersymmetry})\\[.6ex]
	 $S$ be a real smooth vector bundle over $X$ (of some rank$_{\Bbb R}$ $s$);
	 $\;{\cal S}$ be the sheaf of smooth sections of $S$, 	
	 $\;S^{\,\Bbb C}:=S\otimes_{\Bbb R}\!{\Bbb C}$ and
	   ${\cal S}^{\,\Bbb C}:= {\cal S}\otimes_{{\cal O}_X}\!{\cal O}_X^{\,\Bbb C}$
	   be their complexification respectively;\\
    their dual are denoted respectively by
	  $S^{\vee}$, ${\cal S}^{\vee}$, $S^{\vee, {\Bbb C}}$, ${\cal S}^{\vee, {\Bbb C}}$.	

    \item[\LARGE $\cdot$]
	({\it ${\Bbb Z}/2$-graded $\,{\Bbb Z}/2$-commutative structure sheaf on $X$})\\[.6ex]	
	 $\bigwedge_{\Bbb R}^{\bullet}S^{\vee} := \oplus_{l=0}^s \bigwedge_{\Bbb R}^lS^{\vee} $
	   be the exterior ${\Bbb R}$-algebra bundle (= Grassmann ${\Bbb R}$-algebra bundle)
	      generated by $S^{\vee}$;
	 $\;\bigwedge_{{\cal O}_X}^{\bullet}{\cal S}^{\vee}
	    := \oplus_{l=0}^s \bigwedge_{{\cal O}_X}^l{\cal S}^{\vee}$
	   be the sheaf of exterior ${\cal O}_X$-algebras (= sheaf of Grassmann ${\cal O}_X$-algebras)
	      generated by ${\cal S}^{\vee}$;
     here,
	   $\bigwedge_{\Bbb R}^0S^{\vee}=S^{\vee}$ and
	   $\bigwedge_{{\cal O}_X}^0{\cal S}^{\vee}={\cal S}^{\vee}$ by convention;
	 their natural ${\Bbb Z}/2$-grading is specified by
	  $$
	   \begin{array}{lcl}
	    (\bigwedge_{\Bbb R}^{\bullet}S^{\vee})_{\even}\;
		     =\; \bigwedge_{\Bbb R}^{\even}S^{\vee}\,,
         &&	(\bigwedge_{\Bbb R}^{\bullet}S^{\vee})_{\odd}\;
		           =\; \mbox{$\bigwedge$}_{\Bbb R}^{\odd}S^{\vee}\,, \\[1.2ex]
		(\bigwedge_{{\cal O}_X}^{\bullet}{\cal S}^{\vee})_{\even}\;
		     =\;   \bigwedge_{{\cal O}_X}^{\even}{\cal S}^{\vee}\,,
	     && (\bigwedge_{{\cal O}_X}^{\bullet}{\cal S}^{\vee})_{\odd}\;
		     =\;  \bigwedge_{{\cal O}_X}^{\odd}{\cal S}^{\vee}\,;
       \end{array}			
	  $$\\[.6ex]
     $\widehat{\cal O}_X\; :=\;  \bigwedge_{{\cal O}_X}^{\bullet}{\cal S}^{\vee}\;$
        be the {\it ${\Bbb Z}/2$-graded $\,{\Bbb Z}/2$-commutative structure sheaf on $X$ determined by $S$}
  	    (or, equivalently, {\it by ${\cal S}^{\vee}$}; cf.\ Remark~2.1.3);
	 $\;\;\widehat{\cal O}_X^{\,\Bbb C}:= \widehat{\cal O}_X\otimes_{\Bbb R}{\Bbb C}$
	   its complexification; \\
      $\widehat{X}:= (X,\widehat{\cal O}_X)$ be the {\it supermanifold associated to $S$ on $X$};

    \item[\LARGE $\cdot$]
  ({\it ${\Bbb Z}/2$-graded Azumaya/matrix-type noncommutative structure sheaf on $X$})\\[.6ex]
    $\widehat{E}:= E\otimes_{\Bbb R}\bigwedge^{\bullet}_{\Bbb R}S^{\vee}$
	   be the {\it superification of $E$ by $S$};\\
    $\widehat{\cal E}:= {\cal E}\otimes_{{\cal O}_X}\!\!\widehat{\cal O}_X$
	   be the {\it superification} of ${\cal E}$ by ${\cal S}^{\vee}$;
    $$
	   \widehat{\cal O}_X^{A\!z}\;
	   :=\; {\cal O}_X^{A\!z}\otimes_{{\cal O}_X}\!\widehat{\cal O}_X\;
	    =\;  \Endsheaf_{{\cal O}_X^{\,\Bbb C}}({\cal E})
		       \otimes_{{\cal O}_X}\!\mbox{$\bigwedge$}^{\bullet}_{{\cal O}_X}{\cal S}^{\vee}
	$$
	 be the {\it ${\Bbb Z}/2$-graded Azumaya/matrix-type noncommutative structure sheaf on $X$ determined by the
	  (complex, real)-pair $(E,S)$ of bundles on $X$.}
   \end{itemize}
 The ringed space with a module
  $$
    (\widehat{X}^{\!A\!z},\widehat{\cal E})\;
	  :=\; (X, \widehat{\cal O}_X^{A\!z}, \widehat{\cal E})
  $$
  is called an {\it Azumaya/matrix super $C^{\infty}$-manifold with a fundamental module}.
			
 The noncommutative structure sheaf $\widehat{\cal O}_X^{A\!z}$ acts on $\widehat{\cal E}$ from the left
   via the built-in fundamental representation.
 This action commutes with the built-in right action of $\widehat{\cal O}_X^{\,\Bbb C}$ on $\widehat{\cal E}$.
 Thus, with $\widehat{\cal E}$ as a right $\widehat{\cal O }_X^{\Bbb C}$-module,
   $$
      \widehat{\cal O}_X^{A\!z}\;
	    \simeq\; \Endsheaf_{\widehat{\cal O}_X^{\,\Bbb C}}(\widehat{\cal E})
   $$
   canonically.
}\end{definition}

\bigskip

\begin{notation} {\bf [basic]}$\;$ {\rm
 A few elementary statements are made to bring out all basic notations.
  \begin{itemize}
    \item[\LARGE $\cdot$]
     With $\widehat{\cal E}$ as a right $\widehat{\cal O}_X^{\,\Bbb C}$-module and
	     its dual $\widehat{\cal  E}^{\vee}
	                       := \Homsheaf_{\widehat{\cal O}_X}(\widehat{\cal E},\widehat{\cal O}_X)$
         as a left $\widehat{\cal O}_X^{\,\Bbb C}$-module, 						
      $\widehat{\cal O}_X^{A\!z}
        \simeq   \widehat{\cal E}\otimes_{\widehat{\cal O}_X^{\,\Bbb C}}  \widehat{\cal E}^{\vee}$
      also canonically.
	
   \item[\LARGE $\cdot$]	
     In terms of bundles,
	  $\Endsheaf_{\widehat{\cal O}_X^{\,\Bbb C}}(\widehat{\cal E})$
	   is the sheaf of $C^{\infty}$-sections of the bundle of
	    endomorphisms (acting from the left) of fibers of $\widehat{E}$
		as right $\widehat{\Bbb C}_{[s]}$-modules
		with respect to a local trivialization of $S$ that fixes an isomorphism of fibers of
		$(\bigwedge^{\bullet}_{\Bbb R}S^{\vee})^{\Bbb C}$ with $\widehat{\Bbb C}_{[s]}$.
     Note that $\Aut(\widehat{\Bbb C}_{[s]})$ is nontrivial (cf.$\:$Corollary~2.2.1.2). 		
     With a slight abuse of notation, we will denote this endomorphism bundle by
	   $\End_{\widehat{\Bbb C}_{[s]}}(\widehat{E})$,
	   or $\End_{\widehat{\Bbb C}}(\widehat{E})$ when $s$ is implicit.
	
   \item[\LARGE $\cdot$]
    Recall the notation from Definition~2.1.1.
    By construction,
     $\widehat{\cal O}_X^{A\!z}
	    ={\cal O}_X^{A\!z}
		       \oplus
			    {\cal O}_X^{A\!z}\otimes_{{\cal O}_X}\bigwedge^{\ge 1}{\cal S}^{\vee}$,
      with
	   $({\cal O}_X^{A\!z}\otimes_{{\cal O}_X}\bigwedge^{\ge 1}{\cal S}^{\vee})^{s+1}
	     =0$.
    For $\widehat{m}$
	   a section of $\Endsheaf_{\widehat{\cal O}_X^{\,\Bbb C}}(\widehat{\cal E})$,
	 let $\widehat{m}= m_{(0)}+ \widehat{m}_{(\ge 1)}$ be the corresponding decomposition.
	Note that $(\widehat{m}_{(\ge 1)})^{s+1}=0$.
	
   \item[\LARGE $\cdot$]	
    The short exact sequence of ${\cal O}_X^{\Bbb C}$-modules
	  $$
	    0\;\longrightarrow\;
		      {\cal O}_X^{A\!z}\otimes_{{\cal O}_X}
			        \mbox{$\bigwedge$}^{\ge 1}{\cal S}^{\vee}\;
			 \longrightarrow\; \widehat{\cal O}_X^{A\!z}\;
			 \longrightarrow\; {\cal O}_X^{A\!z}\; \longrightarrow\; 0
	  $$
	 splits by the built-in inclusion ${\cal O}_X^{A\!z}\hookrightarrow \widehat{\cal O}_X^{A\!z}$.
   One has thus built-in morphisms
    $$
      \xymatrix{
        X^{\!A\!z}\; \ar@{^{(}->}[r]^-{\iota_{X^{\!A\!z}}}
	     & \;\widehat{X}^{A\!z}\; \ar@{->>}[r]^-{\pi_{X^{A\!z}}}
	     & \;X^{\!A\!z}
       }
    $$
     of ringed spaces
     with $\pi_{X^{\!A\!z}}\circ \iota_{X^{\!A\!z}}
           =\Id_{X^{\!A\!z}}$ the identity map on $X^{\!A\!z}$. 		
 \end{itemize}
}\end{notation}

\bigskip

\begin{remark} $[\,S$ and ${\cal S}^{\vee}$ in defining $\widehat{\cal O}_X\,]\;$
{\rm
 There is a subtle point that may be puzzling to differential geometers (but is most natural to algebraic geometers;
   cf.$\:$[B-F: Sec.$\:$1], [Fu: Appendix B.3], [Ha: Chap.$\:$II: Exercise 5.18] ).
 Namely,
  \begin{itemize}
   \item[\bf Q.] {\it
    Why is $\widehat{\cal O}_X$ defined as
      $\bigwedge^{\bullet}_{{\cal O}_X}\!{\cal S}^{\vee}$,
	  rather than as $\bigwedge^{\bullet}_{{\cal O}_X}\!{\cal S}$,
	  for the supermanifold associated\\ to $S$ on $X$?}
  \end{itemize}
 In the commutative/bosonic algebraic case, the total space of a vector bundle $V$ over a scheme $Z$
  is a scheme that is affine over $Z$ whose scheme structure is given by
   $\boldSpec(\Sym^{\bullet}_{{\cal O}_Z}({\cal V }^{\vee}))$,
     where
	   ${\cal V}$ is the sheaf of local sections of $V$ and
	   ${\cal V}^{\vee}:=\Homsheaf_{{\cal O}_Z}({\cal V},{\cal O}_Z)$ the dual sheaf.
 Only when it is defined this way does one have the correct functorial property that
   a bundle map $V_1\stackrel{f}{\rightarrow} V_2$ over $Z$
     corresponds contravariantly to an ${\cal O}_Z$-algebra-homomorphism
	 ${\cal O}_{V_2}\stackrel{f^{\sharp}}{\rightarrow} {\cal O}_{V_1}$.
 Here, we adopt the same, though in an anti-commuting situation.	
 Indeed, this is also consistent with physicists' intuition on supermanifolds.
 Though odd elements in the structure sheaf are nilpotent and, hence, cannot change the underlying topology,
  physicists tend to think of a supermanifold
  as a ``manifold with some directions parameterized by anticommuting coordinates".
 In this intuition, $\widehat{X}$ is simply the ``total space of $S$ with the fiber directions anticommutitized".
 To describe such directions, one needs anticommuting coordinate functions,
  which are nothing but local sections of $\Hom_{\Bbb R}(S, \underline{\Bbb R})= S^{\vee}$
  but anticommutitized via the built-in embedding of $S^{\vee}$ into the Grassmann-algebra bundle
   $\bigwedge^{\bullet}_{\Bbb R} S^{\vee}$.
 Here, $\underline{\Bbb R}$ is the constant real  line bundle of rank $1$ on $X$.
  
 This also explains
   why we call $\widehat{\cal O}_X$ the structure sheaf determined by $S$ in the language of bundles,
   but {\it by the dual sheaf ${\cal S}^{\vee}$} in the language of sheaves:
 The former refers to the supermanifold $\widehat{X}$
   as ``the total space of $S$ but with anticommuting directions along fibers of $S$ over $X$",
  while the latter refers to the additional anticommuting functions from sections in ${\cal S}^{\vee}$
     in order to extend the sheaf of rings ${\cal O}_X$ to the sheaf of rings $\widehat{\cal O}_X$
	 on the topological space $X$.
}\end{remark}

\bigskip

\begin{flushleft}
{\bf Two Main Theorems on
          maps from $(\widehat{X}_{[s_1]},\widehat{E})$ to $\widehat{Y}_{[s_2]}$ }
\end{flushleft}
Let $\widehat{Y}_{[s_2]}$ be a real super $C^{\infty}$-manifold.

\bigskip

\begin{definition}
{\bf [$C^{\infty}$-admissible ring-homomorphism to
           $C^{\infty}(\End_{\widehat{\Bbb C}_{[s_1]}}(\widehat{E}))$]}$\;$ {\rm
 A ring-homomorphism
  $$
	  \xymatrix{
	   C^\infty(\End_{\widehat{\Bbb C}_{[s_1]}}(\widehat{E}))
	      &&& C^{\infty}(\widehat{Y}_{[s_2]})\ar[lll]_-{\widehat{\varphi}^{\sharp}}
	   }
  $$
  over ${\Bbb R}\hookrightarrow {\Bbb C}$ is said to be {\it $C^{\infty}$-admissible}
 if it extends canonically to the following commutative diagram of ring-homomorphisms
  $$
		 \xymatrix{
	       C^{\infty}(\End_{\widehat{\Bbb C}_{[s_1]}}(\widehat{E}))
			     &&& C^{\infty}(\widehat{Y}_{[s_2]}) \ar[lll]_-{\widehat{\varphi}^{\sharp}}
			                                      \ar@{_{(}->}^-{pr_{\widehat{Y}_{[s_2]}}^{\sharp}}[d]   \\			
			   \rule{0ex}{1.2em}\;\;C^{\infty}(\widehat{X}_{[s_1]})\;\;
			                                                    \ar@{^{(}->}[u]
			                                                    \ar@{^{(}->}[rrr]_-{pr_{\widehat{X}_{[s_1]}}^{\sharp}}
				 &&& C^{\infty}(\widehat{X\times Y}_{[s_1,s_2]})
				            \ar[lllu]_-{\tilde{\widehat{\varphi}}^{\sharp}}		
		}
	 $$
	 such that,
	   if denoting
         $$
		   A_{\widehat{\varphi}}\;
			    :=\; C^{\infty}(\widehat{X})\langle \Image(\widehat{\varphi}^{\sharp})\rangle\;
			    :=\; \Image(\tilde{\widehat{\varphi}}^{\sharp})\;\;\;
			   \subset\; C^{\infty}(\End_{\widehat{\Bbb C}_{[s_1]}}(\widehat{E}))
		  $$
		 and
		  $$
		     A_{\widehat{\varphi},0}\;
			   :=\;  \tilde{\widehat{\varphi}}^{\sharp}
			            (C^{\infty}(\widehat{X\times Y}_{[s_1,s_2]})_{\even})\;\;\;
			   \subset\; A_{\widehat{\varphi}}\,,
		  $$
   then
     the underlying commutative diagram of ring-homomorphisms
  	    $$
		 \xymatrix{
	       A_{\widehat{\varphi}}
			     &&& C^{\infty}(\widehat{Y}_{[s_2]}) \ar[lll]_-{\widehat{\varphi}^{\sharp}}
			                                      \ar@{_{(}->}^-{pr_{\widehat{Y}_{[s_2]}}^{\sharp}}[d]   \\			
			   \rule{0ex}{1.2em}\;\;C^{\infty}(\widehat{X}_{[s_1]})\;\;
			                                              \ar@{^{(}->}[u]
			                                              \ar@{^{(}->}[rrr]_-{pr_{\widehat{X}_{[s_1]}}^{\sharp}}
				 &&& C^{\infty}(\widehat{X\times Y}_{[s_1,s_2]})
				              \ar@{->>}[lllu]_-{\tilde{\widehat{\varphi}}^{\sharp}}		
		}
	   $$
	  restricts to a commutative diagram of $C^{\infty}$-ring-homomorphisms
	   $$
		 \xymatrix{
		  & A_{\widehat{\varphi},0}
			     &&& C^{\infty}(\widehat{Y}_{[s_2]})_{\even}
				              \ar[lll]_-{\widehat{\varphi}^{\sharp}|_{\tinyeven}}
			                  \ar@{_{(}->}^-{pr_{\widehat{Y}_{[s_2]}}^{\sharp}|_{\tinyeven}}[d]   \\			
		  &   \rule{0ex}{1.2em}\;\;C^{\infty}(\widehat{X}_{[s_1]})_{\even}\;\;
			                  \ar@{^{(}->}[u]
			                  \ar@{^{(}->}[rrr]_-{pr_{\widehat{X}_{[s_1]}}^{\sharp}|_{\tinyeven}}
				 &&& C^{\infty}(\widehat{X\times Y}_{[s_1,s_2]})_{\even}
				              \ar@{->>}[lllu]_-{\tilde{\widehat{\varphi}}^{\sharp}|_{\tinyeven}}		 &.
		}
	   $$
	  Here
	    $A_{\widehat{\varphi},0}$ is regarded as a  quotient $C^{\infty}$-ring
		of $C^{\infty}(\widehat{X\times Y}_{[s_1,s_2]})_{\even}$.		
}\end{definition}

\medskip

\begin{theorem} {\bf [every ring-homomorphism in question $C^{\infty}$-admissible]}$\;$
 Every ring-homomorphism
  $$
    \widehat{\varphi}^{\sharp}\;:\; C^{\infty}(\widehat{Y}_{[s_2]})\;
	  \longrightarrow\; C^{\infty}(\End_{\widehat{\Bbb C}_{[s_1]}}(\widehat{E}))
  $$
   over ${\Bbb R}\hookrightarrow {\Bbb C}$
 is $C^{\infty}$-admissible.
\end{theorem}

\bigskip

Theorem~2.1.5
 is a super generalization of [L-Y4: Theorem~3.1.1] (D(11.3.1));	
it justifies the following definition:
	 	
\bigskip

\begin{definition} {\bf [$C^{\infty}$-map from Azumaya/matrix supermanifold]}$\;$ {\rm
 A {\it $C^{\infty}$-map} (synonymously, {\it smooth map})
  $$
    \widehat{\varphi}\;:\;
	  (\widehat{X}_{[s_1]}^{A\!z}, \widehat{E})\;\longrightarrow\; \widehat{Y}_{[s_2]}
  $$
  from an Azumaya/matrix super $C^{\infty}$-manifold with a fundamental module
    $(\widehat{X}_{[s_1]}^{A\!z},\widehat{E})$ to a super $C^{\infty}$-manifold $\widehat{Y}_{[s_2]}$
  is {\sl defined	contravariantly} by a ring-homomorphism between function-rings
  $$
    \widehat{\varphi}^{\sharp}\;:\;
	 C^{\infty}(\widehat{Y}_{[s_2]})\; \longrightarrow\;
	 C^{\infty}(\End_{\widehat{\Bbb C}_{[s_1]}}(\widehat{E}))\,.
  $$
}\end{definition}

\bigskip

Thus,
  the notion of a
    `smooth map from an Azumaya/matrix supermanifold to a real manifold or supermanifold',
  	   first introduced in [L-Y3] (D(11.2)),
	is completely parallel to its counterpart in the realm of algebraic geometry in
	 [L-Y1] (D(1)) and [L-L-S-Y] (D(2))
     for the description of fundamental stacked D-branes in the algebro-geometric situation.

\bigskip

\begin{remark} $[$underlying $C^\infty$-map from $X^{\!A\!z}\,]\;$ {\rm
 Recall Example~1.4.5 and Notation~2.1.2.
 Then,
     every $C^\infty$-map $\widehat{\varphi}:\widehat{X}^{A\!z}\rightarrow \widehat{Y}$
	 has an underlying $C^{\infty}$-map $\varphi:X^{\!A\!z}\rightarrow Y$, defined by the following compositions
   $$
    \xymatrix{
      & \;\widehat{X}^{A\!z}\;    \ar[rr]^-{\widehat{\varphi}}
	     && \;\widehat{Y}\; \ar@{->>}[d]^-{\pi_Y}  \\
      & \rule{0ex}{1.2em}X^{\!A\!z}
	       \ar@{^{(}->}[u]^-{\iota_{X^{\!A\!z}}}
		   \ar@{.>}[rr]^-{\varphi}
	     &&  \;Y\;    & \hspace{-2em}.
    }
   $$
 As a consequence of Sec.$\:$1.2,
   one may also define equivalently the smoothness of $\widehat{\varphi}$
   in terms of the smoothness of the underlying $\varphi$.
}\end{remark}
	
\bigskip	
	
The following theorem gives a super generalization of [L-Y4: Theorem~3.2.1] (D(11.3.1)):
	
\bigskip

\begin{theorem}
{\bf [$C^{\infty}$-map from Azumaya/matrix supermanifold to ${\Bbb R}^{n|s_2}$]}$\;$
 Let
  $(\widehat{X}_{[s_1]}^{A\!z},\widehat{E}) $ be an Azumaya/matrix super $C^{\infty}$-manifold
    with a fundamental module  and
 consider the super $C^{\infty}$-manifold 	
  ${\Bbb R}^{n|s_2}$ with global coordinates
  $(y^1,\,\cdots\,, y^n\,|\, \vartheta^1,\,\cdots\,,\,\vartheta^{s_2} )$.
 Recall Notation~2.1.2.
 Let
  $$
    \widehat{\eta}\;:\;
	  \left\{
	  \begin{array}{lll}
	    y^i  & \longmapsto
	           & \widehat{m}_i\,=\, m_{i,(0)}+ \widehat{m}_{i,(\ge 1)} \\
        \vartheta^l  & \longmapsto
               & \Theta_l		
	  \end{array}\right.\;
	         \in\, C^{\infty}(\End_{\widehat{\Bbb C}_{[s_1]}}(\widehat{E}))\,,
  $$			
   for $i =1,\,\ldots\,,n\,,\; l=1,\,\ldots\,,\, s_2$,
 be an assignment  such that
  \begin{itemize}
   \item[(1)]
     $\widehat{m}_i\widehat{m}_j\;=\;\widehat{m}_j\widehat{m}_i\,$,
	 $\;\widehat{m}_i\Theta_l=\Theta_l\widehat{m}_i\,$,
	 $\;\Theta_l\Theta_{l^{\prime}}= - \Theta_{l^{\prime}}\Theta_l\,$,
	 for all $i,\, j,\, l,\, l^{\prime}$;
	
   \item[(2)]
    for every $p\in X$,
	 the eigenvalues of the restriction
	   $m_{i,(0)}(p)\in \End_{\Bbb C}(E|_p)\simeq M_{r\times r}({\Bbb C})$
	   are all real.
 \end{itemize}
 Then,
  $\widehat{\eta}$ extends uniquely to a $C^{\infty}$-admissible ring-homomorphism
  $$
    \widehat{\varphi}_{\widehat{\eta}}^{\sharp}\;:\;
	 C^{\infty}({\Bbb R}^{n|s_2})\;
	 \longrightarrow\; C^{\infty}(\End_{\widehat{\Bbb C}_{[s_1]}}(\widehat{E}))
  $$
  over ${\Bbb R}\hookrightarrow{\Bbb C}$ and, hence,
  defines a $C^{\infty}$-map
   $\widehat{\varphi}_{\widehat{\eta}}:
     (\widehat{X}_{[s_1]}^{A\!z}, \widehat{E})
	  \rightarrow {\Bbb R}^{n|s_2}$.
\end{theorem}

\bigskip

\bigskip

The proof of these two theorems are given in Sec.$\:$2.4 after the preparation in Sec.$\:$2.2 and Sec.$\:$2.3.

\bigskip

\subsection{Preliminaries on endomorphisms and primary decompositions}

Some basic facts in linear algebra concerning linear transformations and the primary decomposition of a vector space
  under a linear transformation are generalized
 to endomorphisms of a free module over a complex Grassmann algebra   and
 to the stalks of a locally free sheaf on a super $C^{\infty}$-scheme $\widehat{X}$.

\bigskip

\subsubsection{Endomorphisms of a free module over a complex Grassmann algebra}

Consider first the case when $X$ is a point and
let
 \begin{itemize}
  \item[\LARGE $\cdot$]
   $\widehat{\Bbb R}_{[s]} :={\Bbb R}[\theta^1,\cdots,\, \theta^s]^{\anticommuting}$;
   $\widehat{\Bbb C}_{[s]}:= \widehat{\Bbb R}_{[s]}\otimes_{\Bbb R}{\Bbb C}$
	 be its complexification;
   $\widehat{\frak m}_{[s]}:=(\theta^1,\,\cdots\,,\,\theta^s)$ the unique maximal ideal
     of $\widehat{\Bbb R}_{[s]}$;
	
  \item[\LARGE $\cdot$]
   $E\simeq {\Bbb C}^{\oplus r}$; \hspace{2em}
    $\widehat{E}:= E\otimes_{\Bbb R}\widehat{\Bbb R}_{[s]}
	     \simeq \widehat{\Bbb C}_{[s]}^{\;\oplus r}$
	 be the free module over $\widehat{\Bbb C}_{[s]}$ of rank $r$; and
	
  \item[\Large $\cdot$]	
    $\End_{\widehat{\Bbb C}_{[s]}}(\widehat{E})\simeq M_{r\times r}(\widehat{\Bbb C}_{[s]})$
	 be the algebra of endomorphisms of $\widehat{E}$ as a right $\widehat{\Bbb C}_{[s]}$-module.
   Here,
    $M_{r\times r}(\widehat{\Bbb C}_{[s]})$
	is the ring of $r\times r$-matrices over $\widehat{\Bbb C}_{[s]}$,
	which acts on $ \widehat{\Bbb C}_{[s]}^{\;\oplus r}$ from the left.
 \end{itemize}	
Recall
 Definition~2.1.1, Notation~2.1.2,
  the canonical isomorphism
    $\End_{\widehat{\Bbb C}_{[s]}}(\widehat{E})
        \simeq \End_{\Bbb C}(E)\otimes_{\Bbb R}\widehat{\Bbb R}_{[s]}$,     and
  the canonical decomposition
     $\widehat{m}= m_{(0)} + \widehat{m}_{(\ge 1)}$
       for an element $\widehat{m}\in \End_{\widehat{\Bbb C}_{[s]}}(E)$
     with the property that
	   $m_{(0)}\in \End_{\Bbb C}(E)$ and $(\widehat{m}_{(\ge 1)})^{s+1}=0$.
	
Three themes on $\End_{\widehat{\Bbb C}_{[s]}}(\widehat{E})$ that are relevant to our study
 are presented in this subsubsection.
The first follows from a direct computation
 and the second and the third follow from an adaptation of Linear Algebra.

\bigskip

\begin{flushleft}
{\bf The automorphism group $\Aut_{\widehat{\Bbb C}_{[s]}}(\widehat{E})$ of $\widehat{E}$}
\end{flushleft}

\begin{sslemma}{\bf [invertible elements of $\End_{\widehat{\Bbb C}_{[s]}}(\widehat{E})$]}$\;$
 An $\widehat{m}\in \End_{\widehat{\Bbb C}_{[s]}}(\widehat{E})$ is left-invertible
    (resp.\ right-invertible)
  if and only if $m_{(0)}$ is invertible in $\End_{\Bbb C}(E)$.
 When $\widehat{m}$ is (either left- or right-) invertible, its left inverse and its right inverse coincide.
\end{sslemma}

\begin{proof}
 From the decomposition $\widehat{m}= m_{(0)}  + \widehat{m}_{(\ge 1)}$
   with $(\widehat{m}_{(\ge 1)})^{s+1}=0$,
 it is straightforward to check that
   $\widehat{m}$ is either left- or right-invertible in $\End_{\widehat{\Bbb C}_{[s]}}(\widehat{E})$
   if and only if $m_{(0)}$ is invertible in $\End_{\Bbb C}(E)$.
 In which case, either inverse is given by
 $$
   \widehat{m}^{-1}\;
    =\; m_{(0)}^{-1}
	       -  m_{(0)}^{-1}\widehat{m}_{(\ge 1)}m_{(0)}^{-1}
	         \left(1-  \widehat{m}_{(\ge 1)}m_{(0)}^{-1}
			              + ( \widehat{m}_{(\ge 1)}m_{(0)}^{-1}   )^2
			   -\,\cdots\,+\, (-1)^s ( \widehat{m}_{(\ge 1)}m_{(0)}^{-1}   )^s
			  \right)\,.
 $$
\end{proof}

\bigskip

\begin{sscorollary} {\bf [$\Aut_{\widehat{\Bbb C}_{[s]}}(\widehat{E})$]}$\;$
 The automorphism group $\Aut_{\widehat{\Bbb C}_{[s]}}(\widehat{E})$
      of $\widehat{E}$ as a right $\widehat{\Bbb C}_{[s]}$-module
  is given by
 $\Aut_{\Bbb C}(E)\oplus \End_{\Bbb C}(E) \otimes_{\Bbb C}\widehat{\frak m}_{[s]} $,
 with the group multiplication induced from its built-in embedding in
 $\End_{\widehat{\Bbb C}_{[s]}}(\widehat{E})$.
\end{sscorollary}

\bigskip

\begin{flushleft}
{\bf Primary decomposition of $\widehat{E}$
          under an $\widehat{m}\in \End_{\widehat{\Bbb C}_{[s]}}(\widehat{E})$}
\end{flushleft}
\begin{sslemma-definition} {\bf [characteristic polynomial]}$\;$
 For $\widehat{m}\in \End_{\widehat{\Bbb C}_{[s]}}(\widehat{E})$,
 let $\chi_{m_{(0)}}:= \determinant (t\cdot \Id_{r\times r} - m_{(0)}) \in {\Bbb C}[t]$
  be the characteristic polynomial of $m_{(0)}$.
 Define
  $$
     \chi_{\widehat{m}}\;=\;  (\chi_{m_{(0)}})^{s+1}\,.
  $$
 Then
   $$
     \chi_{\widehat{m}}(\widehat{m})=0\,.
   $$
 {\rm We shall call $\chi_{\widehat{m}}$ the {\it characteristic polynomial} of $\widehat{m}$}.
\end{sslemma-definition}
  
\begin{proof}
 Since $\chi_{m_{(0)}}(m_{(0)})= 0$,
  one has
   $\; \chi_{m_{(0)}}(\widehat{m})\;
	  =\;   (\chi_{m_{(0)}}(\widehat{m}))_{(\ge 1)}$.
 It follows that 	
  $$
   \chi_{\widehat{m}}(\widehat{m})
    = \left( (\chi_{m_{(0)}}(\widehat{m}))_{(\ge 1)}\right)^{s+1}
	=\; 0\,.
  $$
  
\end{proof}
 
\bigskip

Let
 $\;\chi_{m_{(0)}}\;=\; (t-\lambda_1)^{d_1}\,\cdots\,(t-\lambda_l)^{d_l}$,
   with $\lambda_1,\,\cdots\,,\, \lambda_l$ all distinct.
Then,
 $$
  \chi_{\widehat{m}}\;
   =\; (t-\lambda_1)^{(s+1)d_1}\,\cdots\,(t-\lambda_l)^{(s+1)d_l}\,.	
 $$
Define
 $$
   g_i\;=\;  \frac{\chi_{\widehat{m}}}{(t-\lambda_i)^{(s+1)d_i}}\; \in\; {\Bbb C}[t]\,,
   \hspace{2em}\mbox{for $i=1,\,\ldots\,,\,l$}\,.
 $$
Then,  $g_1,\,\cdots\,,\, g_l$ are relatively prime and hence there exist $h_1,\,\cdots\,,\, h_l \in {\Bbb C}[t]$
 such that
 $$
   h_1g_1\,+\,\cdots\,+\, h_lg_l\;=\; 1\,.
 $$
Let
 $$
   \widehat{e}_i\;
     :=\; (h_ig_i)(\widehat{m})\;\in\; \End_{\widehat{\Bbb C}_{[s]}}(\widehat{E})\,,
   \hspace{2em}\mbox{for $i=1,\,\ldots\,,\, l$}\,.
 $$
Then,
 
\bigskip

\begin{sslemma-definition} {\bf [complete set of orthogonal idempotents associated to $\widehat{m}$]}$\;$
 The collection $\widehat{e}_1,\,\cdots\,,\, \widehat{e}_l$
    form a complete set of orthogonal idempotents of the algebra
  $\End_{\widehat{\Bbb C}_{[s]}}(\widehat{E})$ over $\widehat{\Bbb C}_{[s]}$.
 Namely,
  $$
    \begin{array}{lcl}
	 \mbox{\rm (complete)}
       &&  \widehat{e}_1\,+\,\cdots\,+\, \widehat{e}_l\;=\; \Id_{\widehat{E}}\,; \\[1.2ex]
	 \mbox{\rm (orthogonal)}
	   && \widehat{e}_i\widehat{e}_j\;=\;0
	           \hspace{2em}\mbox{for $i\ne j$, $\;\;i,j=1,\,\ldots\,,\, l$}\,; \\[1.2ex]	
	 \mbox{\rm (idempotent)}
	   && \widehat{e}_i^{\,2}\;=\; \widehat{e}_i \hspace{2em}\mbox{for $i=1,\,\ldots\,l$}\,.
	\end{array}
  $$
 {\rm We shall call $\{\widehat{e}_1,\,\cdots\,,\,\widehat{e}_l\}$
   a {\it complete set of orthogonal idempotents of $\End_{\widehat{\Bbb C}_{[s]}}(\widehat{E})$
   associated to $\widehat{m}\in \End_{\widehat{\Bbb C}_{[s]}}(\widehat{E})$} }.
\end{sslemma-definition}

\begin{proof}
 Completeness and orthogonality follow from
  $$
   \begin{array}{c}
     h_1g_1\,+\,\cdots\,+\, h_lg_l\;=\; 1\,; \\[1.2ex]
	\mbox{$\chi_{\widehat{m}}\;$ divides $\;(h_ig_i)(h_jg_j)$}
	    \hspace{2em}\mbox{for $i\ne j$, $\;\;i,j=1,\,\ldots\,,\, l$}
   \end{array}
  $$
  respectively.
 Idempotency follows from completeness plus orthogonality.
  
\end{proof}

\bigskip

\begin{ssdefinition-lemma} {\bf [primary decomposition of $\widehat{E}$ associated to $\widehat{m}$]}$\;$ {\rm
 The decomposition
  $$
     \widehat{E}\; =\; \widehat{e}_1\widehat{E}+ \,\cdots\,+ \widehat{e}_l\widehat{E}
  $$
  is a direct-sum decomposition, call a {\it primary decomposition} of $\widehat{E}$
  {\it associated to $\widehat{m}\in \End_{\widehat{\Bbb C}_{[s]}}(\widehat{E})$}.
  {\it Each summand $\widehat{e_i}\widehat{E}$, $i=1,\,\ldots\,,l$,  in this decomposition
            is a free $\widehat{\Bbb C}_{[s]}$-module.}
}\end{ssdefinition-lemma}
 
\begin{proof}
 Consider
   the canonical decomposition $\widehat{e}_i=e_{i,(0)}+\widehat{e}_{i,(\ge 1)}$,
   $i=1,\,\ldots\,,l$.
 Then observe that
   $\{e_{1,(0)},\,\cdots\,,\, e_{l,(0)}\}$ is a complete set of orthogonal idempotents of
   $\End_{\Bbb C}(E)$ associated to $m_{(0)}\in \End_{\Bbb C}(E)$.
 Let
   $$
     E\; =\; e_{1,(0)}E + \,\cdots\,+ e_{l,(0)}E
   $$
   be the corresponding primary decomposition of $E$.
 In this case, this is a direct-sum decomposition of a vector space by its sub-vector-spaces.
 
 Let
   $r_i:=\dimm_{\Bbb C}(e_{i,(0)}E)$
   and
   $$
     (\xi_{1,1},\,\cdots\,,\,\xi_{1,r_1},\;
	 \cdots\,,\; \xi_{l,1},\,\cdots\,,\,\xi_{l,r_l})\,,
   $$
   with $r_1+\cdots+r_l=r$, be a basis of $E$ such that
      $(\xi_{i,1},\,\cdots\,,\,\xi_{i,r_i})$ is a basis of $e_{i,(0)}E$,  $i=1,\,\ldots\,,\, l$.
 Then by construction,
   $$
     \widehat{e}_i (\xi_{i,1}),\,\cdots\,,\,   \widehat{e}_i(\xi_{i,r_i})\;
	  \in\;  \widehat{e}_i\widehat{E}\,, \hspace{2em}\mbox{for $\;i=1,\,\ldots\,,\, l$}\,.
   $$
  We claim that
   \begin{itemize}
    \item[\LARGE $\cdot$] {\it
	 $(\widehat{e}_1(\xi_{1,1}),\,\cdots\,,\,   \widehat{e}_1(\xi_{1,r_1}),\;
	            \cdots\,,\;   \widehat{e_l}(\xi_{l,1}),\,\cdots\,,\, \widehat{e}_l(\xi_{l,r_l}))$
	  is a basis of $\widehat{E}$ as a free $\widehat{\Bbb C}_{[s]}$-module.
	  }
   \end{itemize}
  Once this claim is justified, the lemma then follows.
     
 Let
   $E\simeq {\Bbb C}^{\oplus r}$	
        and, hence, $\widehat{E}\simeq \widehat{\Bbb C}_{[s]}^{\;\oplus r}$
    be the isomorphisms specified by the basis\\
       $\;(\xi_{1,1},\,\cdots\,,\,\xi_{1,r_1},\;
	   \cdots\,,\; \xi_{l,1},\,\cdots\,,\,\xi_{l,r_l})$
     of $E$ (as a ${\Bbb C}$-vector space)
     and, hence, of $\widehat{E}$ (as a $\widehat{\Bbb C}_{[s]}$-module).
 Under this isomorphism, 	
  express an element of $\widehat{E}$ as a column vector with entries in $\widehat{\Bbb C}_{[s]}$.
 Then
   \begin{itemize}
    \item[\LARGE $\cdot$]
	  $(\widehat{e}_1(\xi_{1,1}),\,\cdots\,,\,   \widehat{e}_1(\xi_{1,r_1}),\;
	            \cdots\,,\;   \widehat{e_l}(\xi_{l,1}),\,\cdots\,,\, \widehat{e}_l(\xi_{l,r_l}))$
	    is a basis of $\widehat{E}$ as a free $\widehat{\Bbb C}_{[s]}$-module
     if and only if the $r\times r$ matrix
      $[\widehat{e}_1(\xi_{1,1}),\,\cdots\,,\,   \widehat{e}_1(\xi_{1,r_1}),\;
	            \cdots\,,\;   \widehat{e_l}(\xi_{l,1}),\,\cdots\,,\, \widehat{e}_l(\xi_{l,r_l})]$
	  is invertible in $M_{r\times r}(\widehat{\Bbb C}_{[s]})$.
   \end{itemize}
 Since by construction,
   \begin{eqnarray*}
     &&  [\widehat{e}_1(\xi_{1,1}),\,\cdots\,,\,   \widehat{e}_1(\xi_{1,r_1}),\;
	           \cdots\,,\;
			  \widehat{e_l}(\xi_{l,1}),\,\cdots\,,\, \widehat{e}_l(\xi_{l,r_l})]_{(0)}\\[.6ex]
    &=&  [\xi_{1,1},\,\cdots\,,\,\xi_{1,r_1},\;  \cdots\,,\; \xi_{l,1},\,\cdots\,,\,\xi_{l,r_l}]\;
               =\; \Id_{r\times r}\;\in\;  M_{r\times r}({\Bbb C})\,,								
   \end{eqnarray*}
   which is invertible in $M_{r\times r}({\Bbb C})$,
  if follows from Lemma 2.2.1.1 that\\
   $[\widehat{e}_1(\xi_{1,1}),\,\cdots\,,\,   \widehat{e}_1(\xi_{1,r_1}),\;
	            \cdots\,,\;   \widehat{e_l}(\xi_{l,1}),\,\cdots\,,\, \widehat{e}_l(\xi_{l,r_l})]$
	is invertible in $M_{r\times r}(\widehat{\Bbb C}_{[s]})$.
 This proves the claim	and, hence, the lemma.
    
\end{proof}

\bigskip

\begin{ssremark} $[$block diagonal form of $\widehat{m}\,]\;$ {\rm
 Continuing the above notations.
 By construction,
  $\widehat{e}_i$ is a projection map from $\widehat{E}$ to $\widehat{e}_i\widehat{E}$;
  $\widehat{m}$ leaves each direct summand $\widehat{e}_i\widehat{E}$ invariant;   and
   $$
     \widehat{m}\;=\;   \widehat{m}\widehat{e}_1 +\, \cdots\, + \widehat{m}\widehat{e}_l\,.
   $$
 Thus, with respect to a basis of $\widehat{E}$
   that comes from an ordered collection of bases of
     $\widehat{e}_1\widehat{E},\,\cdots\,,\, \widehat{e}_l\widehat{E}$
	  as free $\widehat{\Bbb C}_{[s]}$-modules,
 $\widehat{m}$ is represented in a block-diagonal form, with one block for each $\lambda_i$, $i=1,\,\ldots\,,\, l$.
}\end{ssremark}

\bigskip

\begin{ssremark}$[$when $\lambda_i$'s are all real$\,]\;$ {\rm
 When all the eigenvalues $\lambda_1,\,\cdots\,,\,\lambda_l$ are real,
 all the polynomials
  $\chi_{m_{(0)}}$, $\chi_{\widehat{m}}$, $g_i$'s, and $h_i$'s
  that appear in the above discussion are in the polynomial ring ${\Bbb R}[t]$ of real coefficients.
}\end{ssremark}

\bigskip

\begin{flushleft}
{\bf Primary decomposition of $\widehat{E}$ under a commuting system of endomorphisms}
\end{flushleft}
When one has
 a commuting system of endomorphisms on $\widehat{E}$ (as a right $\widehat{C}_{[s]}$-module)
 $$
   \widehat{m}_1\,,\; \cdots\,,\; \widehat{m}_n\;
     \in\;  \End_{\widehat{\Bbb C}_{[s]}}(\widehat{E})\,,
	 \hspace{2em}\mbox{with
	   $\; \widehat{m}_i\widehat{m}_j = \widehat{m}_j\widehat{m}_i\;$
       for all $i,\, j$}\,,	
 $$
let
 $$
   \widehat{e}_{i,1}\,,\;\cdots\,,\; \widehat{e}_{i, l_i}
 $$
 be a complete set of orthogonal idempotents of $\End_{\widehat{\Bbb C}_{[s]}}(\widehat{E})$
  associated to $\widehat{m}_i$.
Then
 since  $\widehat{e}_{i,j}$ is a polynomial in $\widehat{m}_i$  and
  $\widehat{m}_i$'s commute with each other,
 $$
    \widehat{e}_{i,j} \widehat{e}_{i^{\prime}, j^{\prime}}\;
	  =\;  \widehat{e}_{i^{\prime},j^{\prime}}\widehat{e}_{i,j}	
 $$
 for all $(i,j),\, (i^{\prime}, j^{\prime})$.
It follows that the expansion
 $$
   (\widehat{e}_{1,1}+ \,\cdots\, + \widehat{e}_{1,l_1})\,\cdots\,
   (\widehat{e}_{n,1}+ \,\cdots\, + \widehat{e}_{n,l_n})\;
   =\;   \widehat{e}_1 +\, \cdots\, + \widehat{e}_l\,,
 $$
 where $\widehat{e}_j= \widehat{e}_{1, \cdot}\,\cdots\,\widehat{e}_{n,\cdot^{\prime}}$
   runs through nonzero summands from the expansion of the product,
  is independent of the order of the factors in the product   and
  gives another complete set of orthogonal idempotents of
     $\End_{\widehat{\Bbb C}_{[s]}}(\widehat{E})$.
It has the following properties:
 \begin{itemize}
  \item[(1)]
   $\widehat{m}_i\widehat{e}_j\;=\; \widehat{e}_j\widehat{m}_i\,$
   for all $i=1,\,\ldots\,,\, n$ and $j=1,\,\cdots\,,\, l$.
   
  \item[(2)]
   The decomposition
      $\widehat{E}= \widehat{e}_1\widehat{E}+ \,\cdots\,+\widehat{e}_l\widehat{E}$
     is a direct-sum decomposition by free $\widehat{\Bbb C}_{[s]}$-modules,
	 with each direct summand $\widehat{e}_i\widehat{E}$ invariant under $\widehat{m}_j$
	 for all $j$.
   Thus, all $\widehat{m}_j$'s are in the block-diagonal form
     with respect to a basis of $\widehat{E}$  that comes from an ordered collection of bases of
	 $\widehat{e}_1\widehat{E}\,, \cdots\,,\,  \widehat{e}_l\widehat{E}$.

  \item[(3)]
   Recall the decomposition $\widehat{m}_j = m_{j,(0)}+ \widehat{m}_{j,(\ge 1)}$.
   Then each block in Property (2) is associated to a unique $n$-tuple
    $(\lambda^1_{ i_1},\,\cdots\,,\,\lambda^n_{i_n})$,
	where $\lambda^j_{k_j }$ is an eigenvalue of $m_{j,(0)}$.
 \end{itemize}

\bigskip

\begin{ssdefinition}{\bf [primary decomposition under a commuting system]}$\;$ {\rm
 The above direct-sum decomposition
   $\; \widehat{E}\; =\;  \widehat{e}_1\widehat{E}+ \,\cdots\,+\widehat{e}_l\widehat{E}\;$
  by free $\widehat{\Bbb C}_{[s]}$-modules
  is called a {\it primary decomposition} of $\widehat{E}$
  {\it under the commuting system of endomorphisms}
   $\widehat{m}_1,\,\cdots\,,\, \widehat{m}_n
      \in \End_{\widehat{\Bbb C}_{[s]}}(\widehat{E})$.
}\end{ssdefinition}

\bigskip

\subsubsection{Generalization of Sec.$\:$2.2.1
         to $C^{\infty}(\End_{\widehat{\Bbb C}_{[s]}}(\widehat{E}))$ for general $X$}

Given $(\widehat{X}_{[s]}^{A\!z},\widehat{E})$	an Azumaya/matrix super $C^{\infty}$-manifold
  with a fundamental module of rank $r$,
let
 $\widehat{m}\in C^{\infty}(\End_{\widehat{\Bbb C}_{[s_1]}}(\widehat{E}))$
 and recall the decomposition
  $\;\widehat{m} = m_{(0)}\,+\, \widehat{m}_{(\ge 1)}\;$ from Notation~2.1.2.
Then it follows from Lemma~2.2.1.1 and Corollary~2.2.1.2 that
      
\bigskip
  
\begin{sslemma}
{\bf [invertible elements of $C^{\infty}(\End_{\widehat{\Bbb C}_{[s]}}(\widehat{E}))$]}$\;$
 An
  $\widehat{m}\in
    C^{\infty}(\End_{\widehat{\Bbb C}_{[s]}}(\widehat{E}))$ is left-invertible
    (resp.\ right-invertible)
  if and only if $m_{(0)}$ is invertible in $C^{\infty}(\End_{\Bbb C}(E))$.
 When $\widehat{m}$ is (either left- or right-) invertible, its left inverse and its right inverse coincide   and
  is given by
  $$
   \widehat{m}^{-1}\;
    =\; m_{(0)}^{-1}
	       -  m_{(0)}^{-1}\widehat{m}_{(\ge 1)}m_{(0)}^{-1}
	         \left(1-  \widehat{m}_{(\ge 1)}m_{(0)}^{-1}
			              + ( \widehat{m}_{(\ge 1)}m_{(0)}^{-1}   )^2
			   -\,\cdots\,+\, (-1)^s ( \widehat{m}_{(\ge 1)}m_{(0)}^{-1}   )^s
			  \right)\,.
   $$
\end{sslemma}		
		
\bigskip

\begin{sscorollary}
{\bf [$C^{\infty}(\Aut_{\widehat{\Bbb C}_{[s]}}(\widehat{E}))$]}$\;$
 The automorphism group $C^{\infty}(\Aut_{\widehat{\Bbb C}_{[s]}}(\widehat{E}))$
      of $\widehat{E}$ as a right $\widehat{\Bbb C}_{[s]}$-module
  is given by
 $C^{\infty}
   (\Aut_{\Bbb C}(E)
         \oplus \End_{\Bbb C}(E) \otimes_{\Bbb R}\bigwedge^{\ge 1}_{\Bbb R}S^{\vee})$,
 with the group multiplication induced from its built-in embedding in
 $C^{\infty}(\End_{\widehat{\Bbb C}_{[s]}}(\widehat{E}))$.
\end{sscorollary}
		
\bigskip
  
Let
  $$
    \chi_{m_{(0)}}\;:=\; \determinant(t\cdot \Id_{r\times r}- m_{(0)}) \;
	 \in\; C^{\infty}(X)^{\Bbb C}[t]
  $$
 be the characteristic polynomial of $m_{(0)}$.
Define
  $$
    \chi_{\widehat{m}}\;=\;  (\chi_{m_{(0)}})^{s+1} \; \in\; C^{\infty}(X)^{\Bbb C}[t]\,.
  $$
Then the identity $\chi_{\widehat{m}}(\widehat{m})=0$ still holds.
We'll call $\chi_{\widehat{m}}$
 the {\it characteristic polynomial} (with coefficients in $C^{\infty}(X)^{\Bbb C}$) of $\widehat{m}$.
  
For $p\in X$,
 suppose that the characteristic polynomial
  $$
    \chi_{m_{(0)}}|_p\;
	 :=\;  \chi_{m_{(0)}(p)}\;
	  =\;  \determinant(t\cdot\Id_{r\times r }- m_{(0)}(p))\;
	  =\; (t-\lambda_1)^{d_1}\,\cdots\,(t-\lambda_l)^{d_l}\;\;\in\; {\Bbb C}[t]
  $$
    of $m_{(0)}(p)$
  has $l$-many distinct roots.
Then
  there exist polynomials $f_1,\,\cdots\,,\, f_l \in C^{\infty}(X)_{(p)}[t]$ in $t$
   with coefficients germs of  complex-valued smooth functions at $p$
 such that
    \begin{itemize}
	  \item[(1)]
	   For any $p^{\prime}\in X$ in a small enough neighborhood of $p$, 	
  	   the sets of roots, one for each polynomials
	    $f_1|_{p^{\prime}},\,\cdots\,, f_l|_{p^{\prime}}\in {\Bbb C}[t]$,
	    are disjoint from each other in ${\Bbb C}$.

      \item[(2)]		
	   $f_i|_p = (t-\lambda_i)^{(s+1)d_i}$, for $i=1,\,\ldots\,, l$.

      \item[(3)]
       $\chi_{\widehat{m}}$ factors into a product
         $\chi_{\widehat{m}}= f_1\,\cdots\, f_l$ as germs at $p$.
    \end{itemize}
Define
 $$
    g_i\;=\; \frac{\chi_{\widehat{m}}}{f_i}\;
	\in\; C^{\infty}(X)^{\Bbb C}_{(p)}[t]\,,
	\hspace{2em}\mbox{for $i=1,\,\ldots\,, l$}\,.	
 $$
Then, $g_1,\,\cdots\,,\, g_l$ are relatively prime in $C^{\infty}(X)^{\Bbb C}_{(p)}$.
Since the Euclid algorithm is an algebraic procedure and, by construction, the coefficient of the top-degree term
 of $g_i$ is invertible in $C^{\infty}(X)^{\Bbb C}_{(p)}$ for all $i$,
 the Euclid algorithm with respect to the $t$-degree remains to work for $g_1,\,\cdots\,, g_l$.
Thus,  there exist $h_1,\,\cdots\,,\, h_l \in C^{\infty}(X)^{\Bbb C}_{(p)}[t]$
 such that
 $$
   h_1g_1\,+\,\cdots\,+\, h_lg_l\;=\; 1\,.
 $$
Let
 $$
   \widehat{e}_i\;
     :=\; (h_ig_i)(\widehat{m})\;
	          \in\; C^{\infty}(\End_{\widehat{\Bbb C}_{[s]}}(\widehat{E}))_{(p)}\,,
   \hspace{2em}\mbox{for $i=1,\,\ldots\,,\, l$}\,.
 $$
Then,  the same argument as in the proof of Lemma/Definition~2.2.1.4 gives
 
\bigskip

\begin{sslemma-definition} {\bf [complete set of orthogonal idempotents associated to $\widehat{m}$]}$\;$
 The collection $\widehat{e}_1,\,\cdots\,,\, \widehat{e}_l$
    form a complete set of orthogonal idempotents of
  $C^{\infty}(\End_{\widehat{\Bbb C}_{[s]}}(\widehat{E}))_{(p)}$.
 Namely,
  $$
    \begin{array}{lcl}
	 \mbox{\rm (complete)}
       &&  \widehat{e}_1\,+\,\cdots\,+\, \widehat{e}_l\;=\; \Id_{\widehat{E}_{(p)}}\,; \\[1.2ex]
	 \mbox{\rm (orthogonal)}
	   && \widehat{e}_i\widehat{e}_j\;=\;0
	           \hspace{2em}\mbox{for $i\ne j$, $\;\;i,j=1,\,\ldots\,,\, l$}\,; \\[1.2ex]	
	 \mbox{\rm (idempotent)}
	   && \widehat{e}_i^{\,2}\;=\; \widehat{e}_i \hspace{2em}\mbox{for $i=1,\,\ldots\,l$}\,.
	\end{array}
  $$
 {\rm We shall call $\{\widehat{e}_1,\,\cdots\,,\,\widehat{e}_l\}$
   a {\it complete set of orthogonal idempotents of
    $C^{\infty}(\End_{\widehat{\Bbb C}_{[s]}}(\widehat{E}))_{(p)}$
   associated to the germ of
    $\widehat{m}\in C^{\infty}(\End_{\widehat{\Bbb C}_{[s]}}(\widehat{E}))$ at $p$} }.
\end{sslemma-definition}

\bigskip

Once having this complete system of orthogonal idempotents,
 all the constructions in Sec.$\:$2.2.1 go through as the level of germs at $p$.
Which we summarize below.

\bigskip

\begin{ssdefinition} {\bf [primary decomposition of $\widehat{E}$ associated to $\widehat{m}$]}$\;$ {\rm
 The decomposition
  $$
     \widehat{E}|_U\; =\; \widehat{e}_1\widehat{E}|_U+ \,\cdots\,+ \widehat{e}_l\widehat{E}|_U
  $$
  over a small neighborhood $U$ of $p\in X$
    on which $\widehat{E}|_U$ is trivial and $\widehat{e}_i$'s are all defined and nowhere zero
  is a direct-sum decomposition, call a {\it primary decomposition} of $\widehat{E}$
  {\it associated to $\widehat{m}\in \End_{\widehat{\Bbb C}_{[s]}}(\widehat{E})$
  locally around $p$}.
 Note that
   this is a decomposition of $\widehat{E}|_U$ by free $\widehat{\cal O}_U$-modules.
}\end{ssdefinition}

\bigskip

\begin{ssremark} $[$block diagonal form of $\widehat{m}\,]\;$ {\rm
 Continuing the above notations.
 By construction,
  $\widehat{e}_i$ is a projection map from $\widehat{E}|_U$ to $\widehat{e}_i\widehat{E}|_U$;
  $\widehat{m}|_U$ leaves each direct summand $\widehat{e}_i\widehat{E}|_U$ invariant;   and
   $$
     \widehat{m}|_U\;
	  =\;   \widehat{m}|_U\widehat{e}_1 +\, \cdots\, + \widehat{m}|_U\widehat{e}_l\,.
   $$
 Thus, with respect to a basis of $\widehat{E}|_U$
   that comes from an ordered collection of bases of
     $\widehat{e}_1\widehat{E}|_U$, $\cdots\,$, $\widehat{e}_l\widehat{E}|_U$,
 $\widehat{m}|_U$ is represented in a block-diagonal form, with one block for each distinct
   eigenvalue $\lambda_i$ of $m_{(0)}(p)$, $i=1,\,\ldots\,,\, l$.
}\end{ssremark}

\bigskip

\begin{ssremark}$[$the case all eigenvalues real$\,]\;$ {\rm
 If for all $p^{\prime} \in X$ the eigenvalues of  $m_{(0)}(p^{\prime})$ are real,
 then
  all the polynomials ($\chi_{m_{(0)}}$, $\chi_{\widehat{m}}$, $g_i$'s, and $h_i$'s) in $t$
  that appear in the above discussion are in the polynomial ring $C^{\infty}(X)_{(p)}[t]$.
}\end{ssremark}
 
\bigskip
 
\begin{sslemma-definition}{\bf [primary decomposition under a commuting system]}$\;$
 Let
 $$
   \widehat{m}_1\,,\; \cdots\,,\; \widehat{m}_n\;
     \in\;   C^{\infty}(\End_{\widehat{\Bbb C}_{[s]}}(\widehat{E}))\,,
	 \hspace{2em}\mbox{with
	   $\; \widehat{m}_i\widehat{m}_j = \widehat{m}_j\widehat{m}_i\;$
       for all $i,\, j$}
 $$
 be a commuting system of endomorphisms on $\widehat{E}$
  (as a right ${\cal O}_{\widehat{X}_{[s]}}$-module).
 Then, for $p\in X$, there exists a neighborhood $U$ of $p$
  such that
   there exists a complete set of orthogonal idempotents
   $\widehat{e}_1,\,\cdots\,,\, \widehat{e}_l$
   of $C^{\infty}(\End_{\widehat{\Bbb C}_{[s]}}(\widehat{E}|_U))$.
   with the following properties: (By shrinking $U$ if necessary, assume that $\widehat{E}|_U$ is trivial.)
   \begin{itemize}
    \item[(1)]
     $\widehat{m}_i|_U\widehat{e}_j\;=\; \widehat{e}_j\widehat{m}_i|_U\,$
     for all $i=1,\,\ldots\,,\, n$ and $j=1,\,\cdots\,,\, l$.
   
    \item[(2)]
     The decomposition
	    $\widehat{E}|_U= \widehat{e}_1\widehat{E}|_U+ \,\cdots\,+\widehat{e}_l\widehat{E}|_U$
       is a direct-sum decomposition by free $\widehat{\cal O}_U$-modules,
	   with each direct summand $\widehat{e}_i\widehat{E}|_U$ invariant under $\widehat{m}_j$
	   for all $j$.
     Thus, all $\widehat{m}_j|_U$'s are in the block-diagonal form
       with respect to a basis of $\widehat{E}|_U$  that comes from an ordered collection of bases of
	   $\widehat{e}_1\widehat{E}|_U\,, \cdots\,,\,  \widehat{e}_l\widehat{E}|_U$.

    \item[(3)]
     Recall the decomposition $\widehat{m}_j = m_{j,(0)}+ \widehat{m}_{j,(\ge 1)}$.
     Then each block in Property (2) is associated to a unique $n$-tuple
      $(\lambda^1_{ i_1},\,\cdots\,,\,\lambda^n_{i_n})$,
 	  where $\lambda^j_{k_j }$ is an eigenvalue of $m_{j,(0)}$.
   \end{itemize}
 
 {\rm The above decomposition
   $\; \widehat{E}|_U\;
     =\;  \widehat{e}_1\widehat{E}|_U+ \,\cdots\,+\widehat{e}_l\widehat{E}|_U\;$
    is called a {\it primary decomposition} of $\widehat{E}$
  {\it over a small enough neighborhood of $p\in X$ under the commuting system of endomorphisms}
   $\widehat{m}_1,\,\cdots\,,\, \widehat{m}_n
      \in C^{\infty}(\End_{\widehat{\Bbb C}_{[s]}}(\widehat{E}))$.}
\end{sslemma-definition}

\bigskip	

\subsection{$C^{\infty}$-maps from an Azumaya/matrix superpoint to a real supermanifold}

We prove in this subsection Theorem~2.1.5 and Theorem~2.1.8
 when
   $X=p$ is a point and
   $\widehat{X}=\widehat{p}_{[s_1]}:= \Spec(\widehat{\Bbb C}_{[s_1]}) $
            is a superpoint.

\bigskip
 
\subsubsection{Proof of Theorem~2.1.5 when $X$ is a point}
 
For
   $X=p$ a point and
   $\widehat{X}=\widehat{p}_{[s_1]}:= \Spec(\widehat{\Bbb C}_{[s_1]}) $ a superpoint,
let
  $$
    \widehat{\varphi}^{\sharp}\;:\;
	  C^{\infty}(\widehat{Y}_{[s]})\; \longrightarrow\;
	  \End_{\widehat{\Bbb C}_{[s_1]}}(\widehat{E})
	    \simeq M_{r\times r}(\widehat{\Bbb C}_{[s_1]})       	
  $$
 be a ring-homomorphism over ${\Bbb R}\hookrightarrow{\Bbb C}$.		
Since here $\widehat{\Bbb C}_{[s_1]}$ acts on $\widehat{E}$ from the right,
 $\widehat{\Bbb C}_{[s_1]}$  commutes with
   $\widehat{\varphi}^{\sharp}(C^{\infty}(\widehat{Y}_{[s_2]}))$
   in $\End_{\widehat{\Bbb C}_{[s_1]}}(\widehat{E})$.
Since
 $C^{\infty}(\widehat{p\times Y}_{[s_1,s_2]}) $
       ($\simeq C^{\infty}(\widehat{Y}_{[s_1,s_2]})$ canonically)
	 is a split-exact, locally free extension of $C^{\infty}(\widehat{Y}_{[s_2]})$ by anticommuting variables,
 $\widehat{\varphi}^{\sharp}$	extends to a commutative diagram of ring-homomorphisms
     $$
		 \xymatrix{
	      & \End_{\widehat{\Bbb C}_{[s_1]}}(\widehat{E})
			     &&& C^{\infty}(\widehat{Y}_{[s_2]}) \ar[lll]_-{\widehat{\varphi}^{\sharp}}
			                                      \ar@{_{(}->}^-{pr_{\widehat{Y}_{[s_2]}}^{\sharp}}[d]   \\			
		  &   \rule{0ex}{1.2em}\;\;\widehat{\Bbb C}_{[s_1]}\;\;
			                                                    \ar@{^{(}->}[u]
			                                                    \ar@{^{(}->}[rrr]_-{pr_{\widehat{p}_{[s_1]}}^{\sharp}}
				 &&& C^{\infty}(\widehat{p\times Y}_{[s_1,s_2]})
				            \ar[lllu]_-{\tilde{\widehat{\varphi}}^{\sharp}}		    &.
		}
	 $$
Since	
   $A_{\widehat{\varphi},0}
      := \tilde{\widehat{\varphi}}^{\sharp}
	            (C^{\infty}(\widehat{p\times Y}_{[s_1,s_2]})_{\even})$
   in the current situation is a finite-dimensional ${\Bbb R}$-algebra,
 the underlying diagram of ring-homomorphisms
	   $$
		 \xymatrix{
		  & A_{\widehat{\varphi},0}
			     &&& C^{\infty}(\widehat{Y}_{[s_2]})_{\even}
				              \ar[lll]_-{\widehat{\varphi}^{\sharp}|_{\tinyeven}}
			                  \ar@{_{(}->}^-{pr_{\widehat{Y}_{[s_2]}}^{\sharp}|_{\tinyeven}}[d]   \\			
		  &   \rule{0ex}{1.2em}\;\;\widehat{\Bbb C}_{[s_1], \even}\;\;
			                  \ar@{^{(}->}[u]
			                  \ar@{^{(}->}[rrr]_-{pr_{\widehat{p}_{[s_1]}}^{\sharp}|_{\tinyeven}}
				 &&& C^{\infty}(\widehat{p\times Y}_{[s_1,s_2]})_{\even}
				              \ar@{->>}[lllu]_-{\tilde{\widehat{\varphi}}^{\sharp}|_{\tinyeven}}		 &.
		}
	   $$
	is automatically a diagram is $C^{\infty}$-ring-homomorphism.
This proves Theorem~2.1.5 in the case when $X$ is a point.	

\bigskip

\subsubsection{Proof of Theorem~2.1.8 when $X$ is a point}

For
   $X=p$ a point and
   $\widehat{X}=\widehat{p}_{[s_1]}:= \Spec(\widehat{\Bbb C}_{[s_1]}) $ a superpoint,
recall the global coordinates $(y^1,\,\cdots\,, y^n\,|\, \vartheta^1,\,\cdots\,,\,\vartheta^{s_2} )$
   for the super $C^{\infty}$-manifold ${\Bbb R}^{n|s_2}$   and
let
   $$
    \widehat{\eta}\;:\;
	  \left\{
	  \begin{array}{lll}
	    y^i  & \longmapsto
	           & \widehat{m}_i\,=\, m_{i,(0)}+ \widehat{m}_{i,(\ge 1)} \\
        \vartheta^{l^{\prime}}  & \longmapsto
               & \Theta_{l^{\prime}}		
	  \end{array}\right.\;
	         \in\,  \End_{\widehat{\Bbb C}_{[s_1]}}(\widehat{E})\,
			              \simeq\, M_{r\times r}(\widehat{\Bbb C}_{[s_1]}) \,,
  $$			
   for $i =1,\,\ldots\,,n\,,\; l^{\prime}=1,\,\ldots\,,\, s_2$,
 be an assignment  such that
  \begin{itemize}
   \item[(1)]
     $\widehat{m}_i\widehat{m}_j\;=\;\widehat{m}_j\widehat{m}_i\,$,
	 $\;\widehat{m}_i\Theta_l=\Theta_l\widehat{m}_i\,$,
	 $\;\Theta_l\Theta_{l^{\prime}}= - \Theta_{l^{\prime}}\Theta_l\,$,
	 for all $i,\, j,\, l,\, l^{\prime}$;
	
   \item[(2)]
	 the eigenvalues of $m_{i,(0)}\in \End_{\Bbb C}(E)\simeq M_{r\times r}({\Bbb C})$
	  are all real.
 \end{itemize}
Recall Sec.$\:$2.2.1 and
let
 $$
   \widehat{e}_1\,,\; \cdots\,,\; \widehat{e}_{k_0}
 $$
  be a complete set of orthogonal idempotents in $\End_{\widehat{\Bbb C}_{[s_1]}}(\widehat{E})$
  associated to the commuting system $\widehat{m}_1,\,\cdots\,,\, \widehat{m}_n$,
 and
 $$
   \widehat{E}\;=\;  \widehat{V}_1\oplus \,\cdots\, \oplus \widehat{V}_{k_0}\;
    :=\;  \widehat{e}_1\widehat{E}\,+\, \cdots\,+\, \widehat{e}_{k_0} \widehat{E}_{k_0}
 $$
 be the corresponding primary decomposition of $\widehat{E}$ by free $\widehat{\Bbb C}_{[s_1]}$-modules.
Associated to each $\widehat{V}_j$ is a
 $$
   q_j\; :=\; (\lambda_j^1, \,\cdots\,,\, \lambda_j^n)\in {\Bbb R}^n\,,
 $$
 where $\lambda_j^i$ is an eigenvalue of $\widehat{m}_i$.

Denote the tuple $(y^1,\,\cdots\,, y^n)$ by $\boldy$.
Let $\widehat{\eta}|_{\scriptsizeboldy}$ be the assignment
 $$
   \widehat{\eta}|_{\scriptsizeboldy}\;:\;
     y^i\;\longmapsto\; \widehat{m}_i\,.
 $$
Then
 $\widehat{\eta}|_{\scriptsizeboldy}$ extends uniquely to a ring-homomorphism
 $$
   \varphi_{\widehat{\eta}|_{\tinyboldy}}^{\sharp}\;:\;
     C^{\infty}({\Bbb R}^n)\;
	 \longrightarrow\; \End_{\widehat{\Bbb C}_{[s_1]}}(\widehat{E})
 $$
 from the composition of ring-homomorphisms
    $$
	  \xymatrix{
	     \;C^{\infty}({\Bbb R}^n)\;
		       \ar@{.>}[rr]^-{\varphi_{\widehat{\eta}|_{\tinyboldy}}^{\sharp}}
	           \ar[d]_-{\times_{j=1}^{k_0}T_{q_j}^{((r-1)(s_1+1)-1)} }
	       &&
		         \; \End_{\widehat{\Bbb C}_{[s_1]}}(\widehat{V}_1)
				    \times \,\cdots\,\;\times \End_{\widehat{\Bbb C}_{[s_1]}}(\widehat{V}_{k_0})\;
		           \subset\; \End_{\widehat{\Bbb C}_{[s_1]}}(\widehat{E})    \\
         \times_{j=1}^{k_0}
	          \frac{{\Bbb R}[y^1-\lambda_j^1,\,\cdots\,, y^n-\lambda_j^n] }
	                       {(y^1-\lambda_j^1,\,\cdots\,,\, y^n-\lambda_j^n)^{(r-1)(s_1+1)}}
	 	    	\ar[rru]_-{\underline{\varphi}^{\sharp}\;
				                       =\; (\, \underline{\varphi}^{\sharp}_1,\,\cdots\,,\,
        									       \underline{\varphi}^{\sharp}_{k_0}\,) }    &&&,
	  }
    $$	
  where
   \begin{itemize}
    \item[\LARGE $\cdot$]
	  $T_{q_j}^{((r-1)(s_1+1)-1)}$ is the map
		 `{\sl taking Taylor polynomial (of elements in $C^{\infty}({\Bbb R}^n)$) at $q_j$
		      with respect to coordinate $(y^1,\,\cdots\,,\, y^n)$ up to and including degree $(r-1)(s_1+1)-1$}', and
	
    \item[\LARGE $\cdot$]
      $$
	    \underline{\varphi}^{\sharp}_j\; :\;  		
		  	   \frac{{\Bbb R}[y^1-\lambda_j^1,\,\cdots\,, y^n-\lambda_j^n]}
	                    {(y^1-\lambda_j^1,\,\cdots\,,\, y^n-\lambda_j^n)^{(r-1)(s_1+1)}}\;
          \longrightarrow\;  \End_{\widehat{\Bbb C}_{[s_1]}}(\widehat{V}_j)\,,							
	  $$ 	
	  is the ${\Bbb R}$-algebra homomorphism generated by sending
	   $y^i\mapsto  \widehat{m}_i\widehat{e_j}$, $i=1,\,\ldots\,,\,n$.
   \end{itemize}
 
Equip
    $\frac{{\Bbb R}[y^1-\lambda_j^1,\,\cdots\,, y^n-\lambda_j^n]}
	                    {(y^1-\lambda_j^1,\,\cdots\,,\, y^n-\lambda_j^n)^{(r-1)(s_1+1)}}$   and
    $\Image(\underline{\varphi}^{\sharp}_j)$						
 with the canonical $C^{\infty}$-ring structure.
 Then all of
    $T_{q_j}^{((r-1)(s_1+1)-1)}$  and
	$\underline{\varphi}^{\sharp}_j$, $j=1,\,\ldots\,,\,k_0$,
  become $C^{\infty}$-ring-homomorphisms.	
Let $A_{\varphi_{\eta}}:= \Image (\varphi_{\widehat{\eta}|_{\tinyboldy}}^{\sharp})$
   be equipped with the canonical $C^{\infty}$-ring structure.
Then
  the ring-homomorphism
  $$
     \varphi^{\sharp}_{\widehat{\eta}|_{\tinyboldy}}\;:\;
	   C^{\infty}({\Bbb R}^n)\; \longrightarrow\; A_{\varphi_{\widehat{\eta}|_{\tinyboldy}}}
  $$
 is also a $C^{\infty}$-ring-homomorphism.

Once $\varphi_{\widehat{\eta}|_{\tinyboldy}}^{\sharp}$ is constructed,
 its unique extension to the following diagram of ring-homomorphisms under the additional assignment
 $$
  \begin{array}{cll}
   \,\theta^l\;\,\longmapsto\;\, \theta^l\,\,\,,                                                               && l\; = 1\,,\,\ldots\,,\, s_1\,,  \\
   \vartheta^{l^{\prime}}\;\longmapsto \; \Theta_{l^{\prime}}\,,     && l^{\prime}=1\,,\,\ldots\,,\, s_2\,,
   \end{array}
 $$
 is immediate and unique:
 $$
   \xymatrix{
      \;C^{\infty}(\widehat{\Bbb R}^n_{[s_1,s_2]})
	        := C^{\infty}({\Bbb R})
			        [\theta^1,\,\cdots\,,\theta^{s_1}\,;\,
			          \vartheta^1,\,\cdots\,, \vartheta^{s_2}]^{\anticommuting}\;
          \ar[ddrrr]^-{\tilde{\widehat{\varphi}}_{\widehat{\eta}}^{\sharp}}	  \\
      \;\rule{0ex}{1.2em}C^{\infty}(\widehat{\Bbb R}^n_{[s_2]})
	    := C^{\infty}({\Bbb R}^n)[\vartheta^1,\,\cdots\,, \vartheta^{s_2}]^{\anticommuting}\;
               \ar[drrr]|-{\;\;\widehat{\varphi}_{\widehat{\eta}}^{\sharp}\;\;}	
               \ar@{^{(}->}[u]		  \\
	  \;C^{\infty}({\Bbb R}^n)\rule{0ex}{1.2em}\;
		       \ar[rrr]_-{\varphi_{\widehat{\eta}|_{\tinyboldy}}^{\sharp}}
			   \ar@{^{(}->}[u]
	       &&&  \; \End_{\widehat{\Bbb C}_{[s_1]}}(\widehat{E})\;    \\				
	}
 $$	
since
  both ring-extensions
   $$
     C^{\infty}({\Bbb R}^n)\;\hookrightarrow\; C^{\infty}(\widehat{\Bbb R}_{[s_2]})
	  \hspace{2em}\mbox{and}\hspace{2em}
	 C^{\infty}(\widehat{\Bbb R}^n_{[s_2]})\;
	   \hookrightarrow\; C^{\infty}(\widehat{\Bbb R}_{[s_1,s_2]})
   $$
  are split-exact, algebraic type, free extension of rings, though by additional anticommuting variables.
 
This proves Theorem~2.1.8 when $X$ is a point.

\bigskip
				
See {\sc Figure}~2-3-2-1 for the geometry behind;
cf.\ [L-Y2: {\sc Figure}~3-4-1] (D(11.1)).
%

 \begin{figure} [htbp]
  \bigskip
  \centering

  \includegraphics[width=0.60\textwidth]{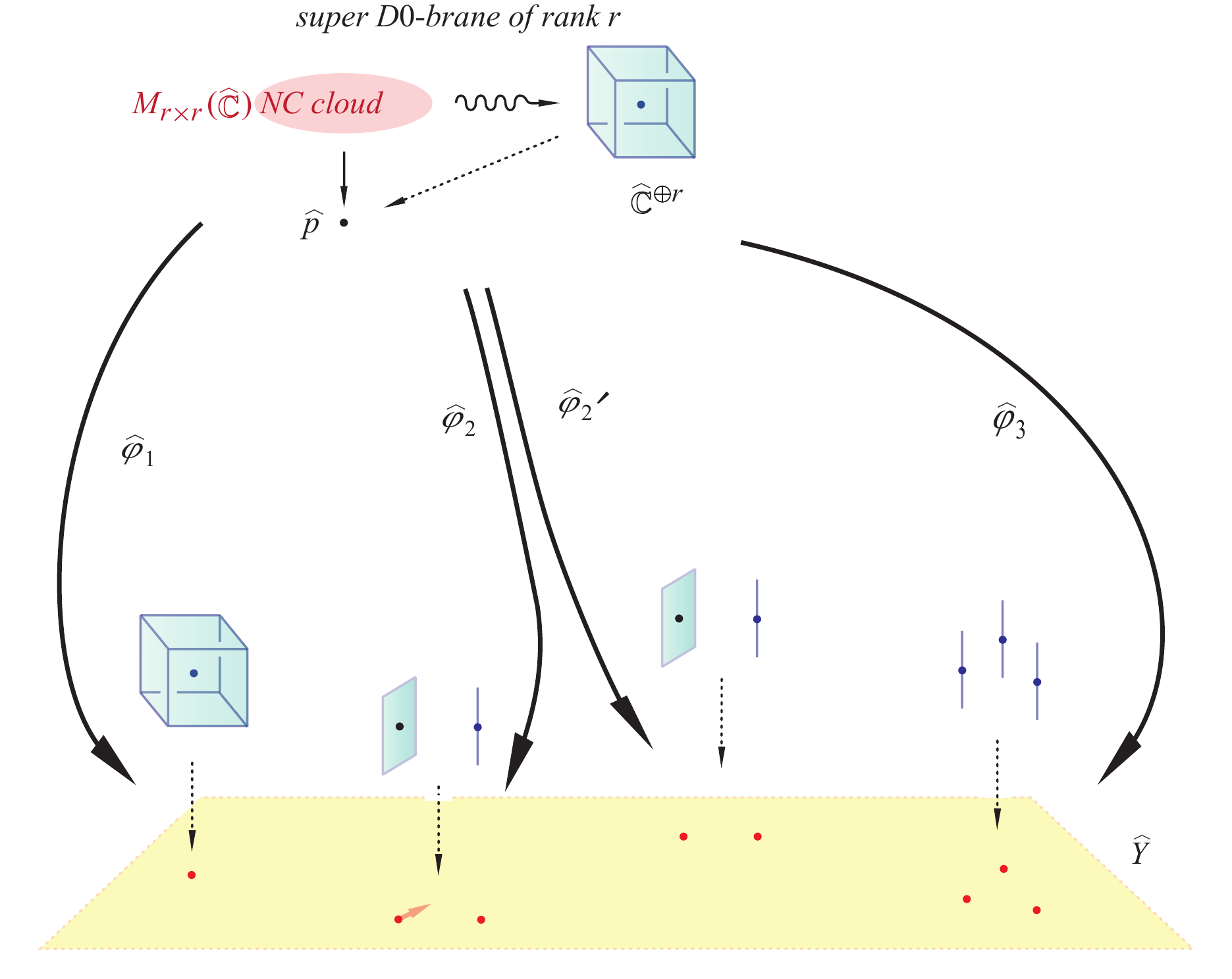}

  \bigskip
  \centerline{\parbox{13cm}{\small\baselineskip 12pt
   {\sc Figure}~2-3-2-1.
    Four examples of $C^{\infty}$-maps
     $\widehat{\varphi}:(\widehat{p}^{A\!z}, \widehat{\Bbb C}^{\oplus r})
	     \rightarrow \widehat{Y}$
      from an Azumaya/matrix superpoint with a fundamental module
	     to a super $C^{\infty}$-manifold $\widehat{Y}$
	  are illustrated.
	The nilpotency of the image scheme $\Image\widehat{\varphi}$ in $\widehat{Y}$
	   is bounded by $(r-1)(s_1+1)$.	
    In the figure, the push-forward of the fundamental module in each example is also indicated.
       }}
  \bigskip
 \end{figure}

\bigskip

\subsection{$C^{\infty}$-maps from an Azumaya/matrix supermanifold to a real supermanifold}

We now proceed to prove Theorem~2.1.5 and Theorem~2.1.8 in their full statement.
The required technical ingredients to generalize both theorems in the $C^{\infty}$ case
    in [L-Y4: Sec.$\:$3] (D(11.3.1))
 to the current super $C^{\infty}$ case are provided in Sec.$\:$2.2 and Sec.$\:$2.3.
The rest follows a similar argument to that in ibidem.
Some details are given here for the completeness of the discussion
 and also for bringing out objects to be used in other parts of the project.

\bigskip						

\subsubsection{Proof of Theorem~2.1.5}

\begin{flushleft}
{\it Step $(a):$
        The only natural candidate for the extension to $C^{\infty}(X\times \widehat{Y}_{[s_2]})$}
\end{flushleft}											
Let
  $\widehat{\varphi}^{\sharp}: C^{\infty}(\widehat{Y}_{[s_2]})
         \rightarrow C^{\infty}(\End_{\widehat{\Bbb C}_{[s_1]}}(\widehat{E}))$
   be a ring-homomorphism over ${\Bbb R}\hookrightarrow {\Bbb C}$.
 Regard the endomorphism bundle
   $\End_{\widehat{\Bbb C}_{[s_1]}}(\widehat{E})$ over $X$
   also as an $X$-family of ${\Bbb C}$-algebras
    $\{\End_{\widehat{\Bbb C}_{[s_1]}}(\widehat{E}|_p) \}_{p\in X}$
  and
 consider the ${\Bbb C}$-algebra
  $C^{-\infty}(\End_{\widehat{\Bbb C}_{[s_1]}}(\widehat{E}))$
   of sections of $\End_{\widehat{\Bbb C}_{[s_1]}}(\widehat{E})\rightarrow X$
   without assuming any continuity or regularity conditions.
Then
 $\widehat{\varphi}^{\sharp}$ extends canonically to the ring-homomorphism
 $$
   \xymatrix @R=-.2ex @C=6em {
    \check{\widehat{\varphi}}^{\sharp}\;:\;
    C^{\infty}(X\times \widehat{Y}_{[s_2]})\;   \ar[r]		
	  &  \hspace{1em}C^{-\infty}(\End_{{\Bbb C}_{[s_1]}}(\widehat{E}))\hspace{1em}     \\
    \hspace{5.2em}\widehat{f}\hspace{3.2em}\ar@{|->}[r]
	  &  \;\left\{p \,\longmapsto\,
                          \left.\left(\widehat{\varphi}^{\sharp}
						                     (\widehat{f}|_{\{p\}\times \widehat{Y}_{[s_2]}})
						            \right)\right|_p	
						 \right \}_{p\in X}	
	 }
 $$
    over ${\Bbb R}\hookrightarrow{\Bbb C}$.
By construction, it fits into the following commutative diagram
  $$
   \xymatrix{
    C^{-\infty}(\End_{\widehat{\Bbb C}_{[s_1]}}(\widehat{E}))
	  & \;C^{\infty}(\End_{{\Bbb C}_{[s_1]}}(\widehat{E})) \ar@{_{(}->}[l]
	  &&  C^{\infty}(\widehat{Y}_{[s_2]})
	                    \ar[ll]_-{\widehat{\varphi}^{\sharp}}
						\ar@{_{(}->}[d]^-{pr_{\widehat{Y}_{[s_2]}}^{\sharp}}  \\	
    &&&  C^{\infty}(X\times \widehat{Y}_{[s_2]})
	                        \ar[lllu]^-{\check{\widehat{\varphi}}^{\sharp}}
	}
  $$
   of ring-homomorphisms, where both inclusions in the diagram are naturally built-in.
Which extends further to the following commutative diagram of ring-homomorphisms
 $$
   \xymatrix{
   & \hspace{1ex}
      C^{-\infty}(\End_{\widehat{\Bbb C}_{[s_1]}}(\widehat{E}))
	     &  \hspace{1ex}
		     C^{\infty}(\End_{\widehat{\Bbb C}_{[s_1]}}(\widehat{E}))
		             \ar @{_{(}->}[l]
	   &&& \rule{0ex}{3ex}\hspace{1ex}
	            C^{\infty}(\widehat{Y}_{[s_2]})
				                \ar[lll]_-{\widehat{\varphi}^{\sharp}}
	                            \ar @{_{(}->}[d]^{pr_{\widehat{Y}_{[s_2]}}^{\sharp}}        \\	
   & & \rule{0ex}{2.4ex}
	     \;C^{\infty}(X)\; \ar @{^{(}->}[rrr]_-{pr_X^{\sharp}} \ar@{^{(}->}[u]
	    &&& C^{\infty}(X\times \widehat{Y}_{[s_2]}) 	
		          \ar @/^1ex /  [ullll]_(.4){\check{\widehat{\varphi}}^{\sharp}}   &,
	}
  $$
  where
    $\pr_X: X\times \widehat{Y}_{[s_2]}\rightarrow X$ and
	  $\pr_{\widehat{Y}_{[s_2]}}: X\times \widehat{Y}_{[s_2]}
	      \rightarrow \widehat{Y}_{[s_2]}$
    are the projection maps and
   the inclusion
      $C^{\infty}(X)\hookrightarrow
           C^{\infty}(\End_{\widehat{\Bbb C}_{[s_1]}}(\widehat{E}))$
    follows from the composition of
	the built-in inclusion
	  $C^\infty(\End_{\Bbb C}(E))\hookrightarrow
	      C^\infty(\End_{\widehat{\Bbb C}_{[s_1]}}(\widehat{E}))$  and
	the inclusion of the center $C^{\infty}(X)^{\Bbb C}$ of $C^{\infty}(\End_{\Bbb C}(E))$.

\bigskip

\begin{flushleft}
{\it Step $(b):$ Smoothness of $\check{\widehat{\varphi}}^{\sharp}$
                                   over $X$ via Malgrange Division Theorem}
\end{flushleft}											
To understand whether $\check{\widehat{\varphi}}^{\sharp}$ takes its values in
   $C^{\infty}(\End_{\widehat{\Bbb C}_{[s_1]}}(\widehat{E}))$,
 one needs to know how
   $$
      \check{\widehat{\varphi}}^{\sharp}|_{\{p\}\times \widehat{Y}_{[s_2]}}\;:\;
       C^{\infty}(\widehat{Y}_{[s_2]})\;
		  \longrightarrow\; \End_{\widehat{\Bbb C}_{[s_1]}}(\widehat{E}|_p)
   $$
	varies as $p$ varies along $X$.
This leads us to studying the germs of $\check{\widehat{\varphi}}^{\sharp}$ over $X$ as given below.

\bigskip

\begin{ssdefinition} {\bf [spectral locus/subscheme
       of $\widehat{\varphi}^{\sharp}$ in $X\times \widehat{Y}_{[s_2]}$]}$\;$ {\rm
 Recall
   the built-in inclusions
     $C^{\infty}(Y) \subset C^{\infty}(\widehat{Y}_{[s_2]})$ and
	 $C^{\infty}(X\times Y)\subset C^{\infty}(X\times \widehat{Y}_{[s_2]})$,	
	 and
   the decomposition $\widehat{m}= m_{(0)}+ \widehat{m}_{(\ge 1)}$
       for $\widehat{m}\in C^{\infty}(\End_{\widehat{\Bbb C}_{[s_1]}}(\widehat{E}))$
	   from Notation~2.1.2.
 For $f\in C^{\infty}(\widehat{Y}_{[s_2]})$,
  denote by $\widehat{\varphi}^{\sharp}(f)_{(0)}$, in $ C^{\infty}(\End_{\Bbb C}(E))$,
    the $(0)$-component of
	 $\widehat{\varphi}^{\sharp}(f)
	    \in C^{\infty}(\End_{\widehat{\Bbb C}_{[s_1]}}(\widehat{E}))$.
 Let
    $\check{I}_{\widehat{\varphi}}$ be the ideal of $C^{\infty}(X\times \widehat{Y}_{[s_2]})$
     generated by the set
     $$
        \{\,\determinant (f\cdot\Id_E- \widehat{\varphi}^{\sharp}(f)_{(0)})^{s_1+1} \,
		       |\, f\in C^{\infty}(Y)\,\}
     $$
     of elements in $C^{\infty}(X\times Y)$, where $\Id_E$ is the identity map on $E$.
   $\check{I}_{\widehat{\varphi}}$  defines a super $C^{\infty}$-subscheme
      $\check{\varSigma}_{\widehat{\varphi}}$ of $X\times \widehat{Y}_{[s_2]}$,
    called interchangeably the {\it spectral locus}
	or the {\it spectral subscheme} of $\widehat{\varphi}^{\sharp}$
    in $X\times \widehat{Y}_{[s_2]}$.
}\end{ssdefinition}

\bigskip

Basic properties of $\check{\varSigma}_{\widehat{\varphi}}$
 that follow immediately from the defining ideal $\check{I}_{\widehat{\varphi}}$
  are listed below:
  \begin{itemize}
   \item[\LARGE $\cdot$]
    {\it $\check{\varSigma}_{\widehat{\varphi}}$ is finite over $X$}
 	 in the sense that, for all $p\in X$,
	 the preimage $\pr_X^{-1}(p)$  of the morphism
	    $\pr_X:\check{\varSigma}_{\widehat{\varphi}}\rightarrow X$
	   from the restriction of the projection map $X\times \widehat{Y}_{[s_2]}\rightarrow X$
	  is a $0$-dimensional super $C^{\infty}$-scheme
	  with the function-ring given by
	          a finite-dimensional ${\Bbb Z}/2$-graded ${\Bbb Z}/2$-commutative ${\Bbb R}$-algebra. 	
	
   \item[\LARGE $\cdot$]
    A comparison with the study of ring-homomorphisms from $C^{\infty}({\Bbb R}^n)$ to
	 $M_{r\times r}({\Bbb C})$ in [L-Y2: Sec.$\:$3.2] (D(11.1)),
	 together with the super supplement in Sec.$\:$2.2,
	 implies that
	 \begin{itemize}
	  \item[-$\;$]
       $\check{\widehat{\varphi}}^{\sharp}(\check{I}_{\widehat{\varphi}})\;=\; 0\,$.
	
	  \item[-$\;$]\it
	   for all $\widehat{f}\in C^{\infty}(X\times \widehat{Y}_{[s_2]})$,
       $\check{\widehat{\varphi}}^{\sharp}(\widehat{f})
	       \in C^{-{\infty}}(\End_{\widehat{\Bbb C}_{[s_1]}}(\widehat{E}))$
	   depends only on the restriction of $\widehat{f}$ on the super $C^{\infty}$-subscheme
	    $\check{\varSigma}_{\widehat{\varphi}}\subset X\times \widehat{Y}_{[s_2]}$.	
	 \end{itemize}
  \end{itemize} 	
Cf.\ {\sc Figure}~2-4-1-1.

 \begin{figure} [htbp]
  \bigskip
  \centering

  \includegraphics[width=0.55\textwidth]{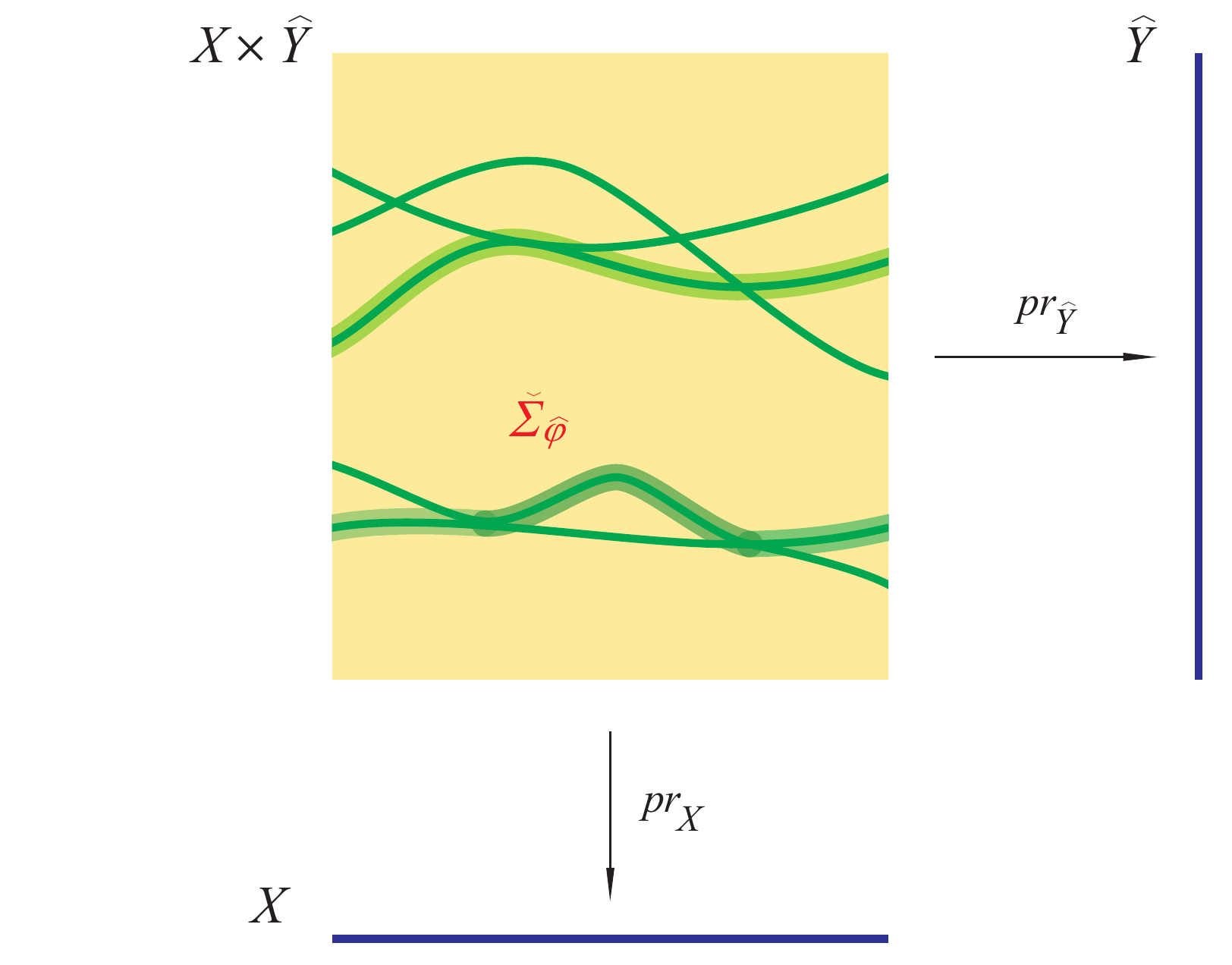}

  \bigskip
  \bigskip
  \centerline{\parbox{13cm}{\small\baselineskip 12pt
   {\sc Figure}~2-4-1-1.
     The spectral subscheme $\check{\varSigma}_{\widehat{\varphi}}$
	   (in green color, with the green shade indicating the nilpotent structure/cloud
	       on $\check{\varSigma}_{\widehat{\varphi}}$)
	   in $X\times \widehat{Y}_{[s_2]}$
	  associated to a ring-homomorphism
	   $\widehat{\varphi}^{\sharp}:C^{\infty}(\widehat{Y}_{[s_2]})
	      \rightarrow C^{\infty}(\End_{\widehat{\Bbb C}_{[s_1]}}(\widehat{E}))$.
	 It is a super $C^{\infty}$-scheme that is finite over $X$.
       }}
  \bigskip
 \end{figure}	

 Recall the morphism $\pr_X: \check{\varSigma}_{\widehat{\varphi}}\rightarrow X$.
 Let $p\in X$.
 Then
  since $\pr_X^{-1}(p)$ is $0$-dimensional,
  there exists an open neighborhood $U$ of $p$ such that
   $\pr_X^{-1}(U)$
     is contained in an open subset $U\times \widehat{V}$ of $X\times \widehat{Y}_{[s_1]}$,
	 where $\widehat{V}$ is an open set of $\widehat{Y}$
 	 that is diffeomorphic to ${\Bbb R}^{n|s_2}$ with $n=\dimm Y$.
 Under the diffeomorphism $\widehat{V}\simeq {\Bbb R}^{n|s_2}$,
  let
   $(y^1,\,\cdots\,,\, y^n)$  be coordinates on the underlying smooth manifold $V$ of $\widehat{V}$   and
   $C^{\infty}(\widehat{V})
              \simeq  C^{\infty}({\Bbb R}^n)[\vartheta_1,\,\cdots\,,\, \vartheta_{s_2}]^{\anticommuting}$.
 Let
  $$
    (\pr_X^{-1}(p))_{red}\;=\; \{q_1,\,\cdots\,,\,q_l\}
  $$
  be the set of closed points in $\pr_X^{-1}(p)\subset  \check{\varSigma}_{\widehat{\varphi}}$.
 (For notation,
     $q_j=(p;  (q_j^1,\,\cdots\,,\,q_j^n))\in U\times V$
	 in the coordinate system $(y^1,\,\cdots\,,\,y^n)$ on $V$.)
 Consider the auxiliary super $C^{\infty}$-subscheme
  $$
    \check{\varSigma}_{(y^1,\,\cdots\,,\,y^n)}\; \subset \; U\times \widehat{V}
  $$
  defined by the ideal
  $$
   \mbox{
    $\check{I}_{(y^1,\,\cdots\,,\,y^n)}\;
	 :=\;  (g_1,\,\cdots\,, g_n)\;\subset\; C^{\infty}(U\times \widehat{V})\,,\;\;$
	 where
     $\;g_i\;:=\;
	      \determinant(y^i\cdot\Id_E- \widehat{\varphi}^{\sharp}(y^i)_{(0)})^{s_1+1}\,$.
	}
  $$
 Then,
  $$
    \check{\varSigma}_{\widehat{\varphi}}\,\cap\, U\times \widehat{V}\;
	 \subset\;  \check{\varSigma}_{(y^1,\,\cdots\,,\,y^n)}\,.
  $$
 Since
   $\check{\widehat{\varphi}}^{\sharp}(\widehat{f})$
	    for $\widehat{f}\in C^{\infty}(X\times \widehat{Y}_{[s_2]})$
     is defined pointwise over $X$    and
	 depends only on $\widehat{f}|_{\check{\varSigma}_{\widehat{\varphi}}}$,
  $\check{\widehat{\varphi}}^{\sharp}$
    is tautologically defined on $C^{\infty}(U\times \widehat{V})$ over $U$ as well
  and, again, one has
  $$
    \check{\widehat{\varphi}}^{\sharp}(\check{I}_{(y^1,\,\cdots\,,\,y^n)})\; =\; 0\,.
  $$
  
  Now let
  $$
    d_{i,1},\,\cdots\,,\,  d_{i,s}
  $$
  be the regularity of $g_i$ along the $y^i$-coordinate direction at $q_1,\,\cdots,\,\, q_l$ respectively
  (cf.\ [Br: 6.1 Definition]).
  I.e.\
  $$
    g_i(q_j)\;=\; \partial_ig_i(q_j) \; =\; \cdots\;
	 =\; \partial_i^{d_{i,j}-1}g_i(q_j)\;=\; 0
	\hspace{1em}\mbox{while}\hspace{1em}
	\partial_i^{d_{i,j}}g_i(q_j)\;\ne \; 0\,.	
  $$
  Here, $\partial_i := \partial/\partial y^i$.
 Then, it follows from the Malgrange Division Theorem ([Mal]; see also [Br], [Mat1], [Mat2], [Ni]) that
  \begin{itemize}
    \item[]\it
    the germ of $\widehat{f}\in C^{\infty}(X\times \widehat{Y}_{[s_2]})$ at $q_j$ admits a normal form
	 $$
	    \widehat{f}=  \widehat{f}_0^{\:(q_j)}\,+\, \widehat{f}_1^{\:(q_j)}
	 $$
	with
	 $$
	  \begin{array}{c}
	    \widehat{f}_0^{\:(q_j)}\;
		 \in\; C^{\infty}(U)[y^1,\,\cdots\,,\, y^n]
		                                           [\vartheta^1,\,\cdots\,,\,\vartheta^{s_2}]^{\anticommuting}\\[.6ex]
        \hspace{12em}\mbox{of $\;(y^1,\,\cdots\,,\,y^n)$-degree}\;
        \le(d_{1,j}-1,\,\cdots\,,\, d_{n,j}-1)\\[2ex]
       \hspace{2em}\mbox{and}\hspace{2em}
       \widehat{f}_1^{\:(q_j)}\;\in \; \check{I}_{(y^1,\,\cdots\,,\,y^n)}\,. \hspace{5em}	   		
	  \end{array}
	 $$
  \end{itemize}
 After
   shrinking the neighborhood $U$ of $p\in X$ further, if necessary,  and
   capping $\widehat{f}_0^{\:(q_j)}$ (still denoted by $\widehat{f}_0^{\:(q_j)}$)
  by a smooth cutoff function with support a disjoint union of small enough coordinate balls
  around $q_j$, $j=1,\,\ldots\,,\,l$,
 $$
  \begin{array}{l}
   \check{\widehat{\varphi}}^{\sharp}(\widehat{f})|_{U}\;
    =\;   \check{\widehat{\varphi}}^{\sharp}
               (\,\mbox{$\sum_{j=1}^l \widehat{f}_0^{\:(q_j)}$}\,)\;
    \in\; C^{\infty}(U)
            [\widehat{\varphi}^{\sharp}(y^1),\,\cdots\,,\, \widehat{\varphi}^{\sharp}(y^n)]
            [\widehat{\varphi}^{\sharp}(\vartheta^1),\,\cdots\,,\,
			 \widehat{\varphi}^{\sharp}(\vartheta^{s_2})]^{\anticommuting}   \\[1.2ex]
    \hspace{12.5em}
    \subset\; C^{\infty}(\End_{\widehat{\Bbb C}_{[s_1]}}(\widehat{E}|_U))
  \end{array}
 $$
 since
   \begin{itemize}
    \item[\LARGE $\cdot$]
	 $\check{\widehat{\varphi}}^{\sharp}(h)\;=\; \widehat{\varphi}^{\sharp}(h) $
	    for all $h\in C^{\infty}(\widehat{Y}_{[s_2]})
		                \subset C^{\infty}(X\times \widehat{Y}_{[s_2]})$,\\
	   particularly for $y^1,\,\cdots\,,\, y^n$ and $\vartheta^1,\,\cdots\,,\vartheta^{s_2}$;
	
	\item[\LARGE $\cdot$]
	 $\check{\widehat{\varphi}}^{\sharp}(h)\;=\; h\cdot \Id_{\widehat{E}}$
 	   for all $h\in C^{\infty}(X)$,
	  where $\Id_{\widehat{E}}$ is the identity map on $\widehat{E}$.
   \end{itemize}
Since smoothness is a local (indeed, infinitely infinitesimal) property, 	
   smoothness of $\check{\widehat{\varphi}}^{\sharp}(\widehat{f})$
    for all $\widehat{f}\in C^{\infty}(X\times \widehat{Y}_{[s_2]})$
   follows.
This shows that
	 $\Image(\check{\widehat{\varphi}}^{\sharp})
	   \subset C^{\infty}(\End_{\widehat{\Bbb C}_{[s_1]}}(\widehat{E}))$  and
one has a commutative diagram of ring-homomorphisms
 $$
   \xymatrix{
    & \;C^{\infty}(\End_{{\Bbb C}_{[s_1]}}(\widehat{E}))
	  &&&  C^{\infty}(\widehat{Y}_{[s_2]})
	                    \ar[lll]_-{\widehat{\varphi}^{\sharp}}
						\ar@{_{(}->}[d]^-{pr_{\widehat{Y}_{[s_2]}}^{\sharp}}  \\	
    & &&&  C^{\infty}(X\times \widehat{Y}_{[s_2]})
	                        \ar[lllu]^-{\check{\widehat{\varphi}}^{\sharp}}    &.
	}
  $$

\bigskip

\begin{flushleft}
{\it Step $(c):$ Final canonical extension to $C^{\infty}(\widehat{X\times Y}_{[s_1,s_2]})$}
\end{flushleft}											
Recall the canonical embedding
 $C^{\infty}(\widehat{X}_{[s_1]})
    \subset  C^{\infty}(\End_{\widehat{\Bbb C}_{[s_1]}}(\widehat{E}))$
 by the right ${\cal O}_{\widehat{X}_{[s_1]}}$-module structure of 	$\widehat{\cal E}$.
Since
   \begin{itemize}
     \item[\LARGE $\cdot$]
     $\Image \check{\widehat{\varphi}}^{\sharp}
        \subset C^{\infty}(\End_{\widehat{\Bbb C}_{[s_1]}}(\widehat{E}))$
      acts on $\widehat{E}$ from the left and hence commutes with $C^{\infty}(\widehat{X}_{[s_1]})$	
	   and
	
	 \item[\LARGE $\cdot$]
      the ring-extension
        $C^{\infty}(X\times \widehat{Y}_{[s_2]})\hookrightarrow
          C^{\infty}(\widehat{X\times Y}_{[s_1,s_2]})$
        is split-exact, locally free, and of ${\Bbb Z}/2$ algebraic type,
   \end{itemize}
 $\check{\widehat{\varphi}}^{\sharp}$ extends canonically to a ring-homomorphism
 $$
  \tilde{\widehat{\varphi}}^{\sharp}\;:\;
    C^{\infty}(\widehat{X\times Y}_{[s_1,s_2]})\;
	\longrightarrow\; C^{\infty}(\End_{\widehat{\Bbb C}_{[s_1]}}(\widehat{E}))
 $$
  that makes the following diagram of ring-homomorphisms commute:
 $$
   \xymatrix{
    & \;C^{\infty}(\End_{{\Bbb C}_{[s_1]}}(\widehat{E}))
	  &&&  C^{\infty}(\widehat{Y}_{[s_2]})
	                    \ar[lll]_-{\widehat{\varphi}^{\sharp}}
						\ar@{_{(}->}[d]^-{pr_{\widehat{Y}_{[s_2]}}^{\sharp}}  \\	
    & &&&  C^{\infty}(X\times \widehat{Y}_{[s_2]})
	                        \ar[lllu]^-{\check{\widehat{\varphi}}^{\sharp}}
                            \ar@{_{(}->}[d]							\\
    & \;  \rule{0ex}{2em}\;C^{\infty}(\widehat{X}_{[s_1]})\;
                              \ar@{^{(}->}[uu]
			                  \ar@{^{(}->}[rrr]_-{pr_{\widehat{X}_{[s_1]}}^{\sharp}}	
        &&& \;C^{\infty}(\widehat{X\times Y}_{[s_1,s_2]})\;
                              \ar @/^2ex/ [llluu]^-{\tilde{\widehat{\varphi}}^{\sharp}}    		&.	
	}
  $$

\bigskip

\begin{flushleft}
{\it Step $(d):$ $C^{\infty}$-admissibility}
\end{flushleft}		
Let
 $A_{\widehat{\varphi},0}
   := \tilde{\widehat{\varphi}}^{\sharp}
         (C^{\infty}(\widehat{X\times Y}_{[s_1,s_2]})_{\even})$.	
Then, 			
 as a consequence of the Hadamard's Lemma,
 the $C^{\infty}$-ring structure on $C^{\infty}(\widehat{X\times Y}_{[s_1,s_2]})_{\even}$
 always descends, via $\tilde{\widehat{\varphi}}^{\sharp}$,
 to a $C^{\infty}$-ring structure on $A_{\widehat{\varphi},0}$
 that is compatible with the underlying ring-structure of $A_{\widehat{\varphi},0}$.
In this way, one obtains a commutative diagram
   $$
	 \xymatrix{
		  A_{\widehat{\varphi},0}
			     &&& C^{\infty}(\widehat{Y}_{[s_2]})_{\even}
				              \ar[lll]_-{\widehat{\varphi}^{\sharp}|_{\tinyeven}}
			                  \ar@{_{(}->}^-{pr_{\widehat{Y}_{[s_2]}}^{\sharp}|_{\tinyeven}}[d]   \\			
		    \rule{0ex}{1.2em}\;\;C^{\infty}(\widehat{X}_{[s_1]})_{\even}\;\;
			                  \ar@{^{(}->}[u]
			                  \ar@{^{(}->}[rrr]_-{pr_{\widehat{X}_{[s_1]}}^{\sharp}|_{\tinyeven}}
				 &&& C^{\infty}(\widehat{X\times Y}_{[s_1,s_2]})_{\even}
				              \ar@{->>}[lllu]_-{\tilde{\widehat{\varphi}}^{\sharp}|_{\tinyeven}}		 
		}
   $$
  of $C^{\infty}$-ring-homomorphisms.
This shows that $\widehat{\varphi}^{\sharp}$ is $C^{\infty}$-admissible and proves the theorem.

\bigskip

Before leaving this subsubsection, we introduce a terminology and a notation for future use.

\bigskip

\begin{ssdefinition}
{\bf [spectral locus/subscheme of $\widehat{\varphi}^{\sharp}$
            in $\widehat{X\times Y}_{[s_1,s_2]}$]}$\;$ {\rm
 Recall
  the spectral subscheme $\check{\varSigma}_{\widehat{\varphi}}$
    of $\widehat{\varphi}^{\sharp}$ in $X\times \widehat{Y}_{[s_2]}$  and
  the built-in morphism $\widehat{X\times Y}_{[s_1,s_2]}\rightarrow X\times \widehat{Y}_{[s_2]}$
      of super $C^{\infty}$-schemes.
The preimage $\varSigma_{\widehat{\varphi}}$ of $\check{\varSigma}_{\widehat{\varphi}}$
 under the above morphism is called
 the {\it spectral locus/subscheme of $\widehat{\varphi}^{\sharp}$
            in $\widehat{X\times Y}_{[s_1,s_2]}$}.
}\end{ssdefinition}

\bigskip											
											
\subsubsection{Proof of Theorem~2.1.8}

Given $\widehat{\eta}$ in the statement of the theorem,
  it follows from Sec.$\:$2.3.2 that
   for all $p\in X$, the assignment from restriction
   $$
     \widehat{\eta}|_p\;:\;
	  \left\{
	  \begin{array}{lll}
	    y^i  & \longmapsto
	           & \widehat{m}_i(p) \\
        \vartheta^l  & \longmapsto
               & \Theta_l(p)		
	  \end{array}\right.\;
	         \in\,   \End_{\widehat{\Bbb C}_{[s_1]}}(\widehat{E}|_p)\,,
   $$			
     for $i =1,\,\ldots\,,n\,,\; l=1,\,\ldots\,,\, s_2$,	
   extends uniquely to a ring-homomorphism
   $$
      \widehat{\varphi}^{\sharp}_{\widehat{\eta}|_p}\;:\;
	   C^{\infty}({\Bbb R}^{n|s_2})=
	     C^{\infty}({\Bbb R}^n)[\vartheta^1,\,\cdots\,,\, \vartheta^{s_2}] \;
	    \longrightarrow\;   \End_{\widehat{\Bbb C}_{[s_1]}}(\widehat{E}|_p)
   $$
   over ${\Bbb R}\hookrightarrow {\Bbb C}$ that is $C^{\infty}$-admissible over $p$.
 As $p$ varies,  $\widehat{\eta}$ extends uniquely to a ring-homomorphism
    $$
     \widehat{\varphi}_{\widehat{\eta}}^{\sharp}\;:\;
  	  C^{\infty}({\Bbb R}^{n|s_2})\; \longrightarrow\;
	  C^{-{\infty}}(\End_{\widehat{\Bbb C}_{[s_1]}}(\widehat{E}))
    $$
    over ${\Bbb R}\hookrightarrow {\Bbb C}$.
 The same construction as Step (a) in the proof of Theorem~2.1.5
   extends $\widehat{\varphi}_{\widehat{\eta}}^{\sharp}$ further and uniquely to a ring-homomorphism
     $$
	   \check{\widehat{\varphi}}_{\widehat{\eta}}^{\sharp}\;:\;
	      C^{\infty}(X\times {\Bbb R}^{n|s_2})\; \longrightarrow\;
		  C^{-{\infty}}(\End_{\widehat{\Bbb C}_{[s_1]}}(\widehat{E}))
	 $$
	 over ${\Bbb R}\hookrightarrow {\Bbb C}$
	that fits into the following commutative diagram
	 $$
      \xymatrix{
        \hspace{1ex}
         C^{-\infty}(\End_{\widehat{\Bbb C}_{[s_1]}}(\widehat{E}))
	       &  \hspace{1ex}C^{\infty}(\End_{\widehat{\Bbb C}_{[s_1]}}(\widehat{E}))
		                                          \ar @{_{(}->}[l]
	      &&& \rule{0ex}{3ex}\hspace{1ex}
	               C^{\infty}({\Bbb R}^{n|s_2})
				         \ar@/_2em/[llll]_-{\widehat{\varphi}_{\widehat{\eta}}^{\sharp}}
	                     \ar @{_{(}->}[d]^{pr_{{\Bbb R}^{n|s_2}}^{\sharp}}        \\	
       & \rule{0ex}{2.4ex}
	        \;C^{\infty}(X)\; \ar @{^{(}->}[rrr]_-{pr_X^{\sharp}} \ar@{^{(}->}[u]
	       &&& C^{\infty}(X\times {\Bbb R}^{n|s_2}) 	
		                    \ar @/^1ex /  [ullll]_(.4){\check{\widehat{\varphi}}_{\widehat{\eta}}^{\sharp}}\;,
	   }
     $$
	of ring-homomorphisms while satisfying the condition that
	 $$
        \check{\widehat{\varphi}}_{\widehat{\eta}}^{\sharp}
		    |_{\{p\}\times {\Bbb R}^{n|s_2} }\;
	   =\; \widehat{\varphi}_{\widehat{\eta}|_p}^{\sharp}\;
		:\;  C^{\infty}({\Bbb R}^{n|s_2})\;
		       \longrightarrow\; \End_{\widehat{\Bbb C}_{[s_1]}}(\widehat{E}|_p)\,,
     $$
     for all $p\in X$.
	
 The same argument as Step (b) in the proof of Theorem~2.1.5,
      using the Malgrange Division Theorem,															  
    implies that indeed
     $\check{\widehat{\varphi}}_{\widehat{\eta}}^{\sharp}$
   	 takes values in $C^{\infty}(\End_{\widehat{\Bbb C}_{[s_1]}}(\widehat{E}))$.
 Thus, so does $\widehat{\varphi}_{\widehat{\eta}}^{\sharp}$.	
 As in Step (c) there,
  $\check{\widehat{\varphi}}_{\widehat{\eta}}^{\sharp}$ extends finally to
  a ring-homomorphism
   $$
     \tilde{\widehat{\varphi}}_{\widehat{\eta}}^{\sharp}\;:\;
	  C^{\infty}(\widehat{X\times {\Bbb R}}_{[s_1,s_2]})\;
	  \longrightarrow\;  C^{\infty}(\End_{{\Bbb C}_{[s_1]}}(\widehat{E}))
   $$
  that fits into the following diagram of ring-homomorphisms
  $$
   \xymatrix{
    & \;C^{\infty}(\End_{{\Bbb C}_{[s_1]}}(\widehat{E}))
	  &&&  C^{\infty}({\Bbb R}^{n|s_2})
	                    \ar[lll]_-{\widehat{\varphi}_{\widehat{\eta}}^{\sharp}}
						\ar@{_{(}->}[d]^-{pr_{{\Bbb R}^{n|s_2}}^{\sharp}}  \\	
    & &&&  C^{\infty}(X\times {\Bbb R}^{n|s_2})
	                        \ar[lllu]^-{\check{\widehat{\varphi}}_{\widehat{\eta}}^{\sharp}}
                            \ar@{_{(}->}[d]							\\
    & \;  \rule{0ex}{2em}\;C^{\infty}(\widehat{X}_{[s_1]})\;
                              \ar@{^{(}->}[uu]
			                  \ar@{^{(}->}[rrr]_-{pr_{\widehat{X}_{[s_1]}}^{\sharp}}	
        &&& \;C^{\infty}(\widehat{X\times {\Bbb R}^n}_{[s_1,s_2]})\;
                              \ar @/^2ex/ [llluu]^-{\tilde{\widehat{\varphi}}_{\widehat{\eta}}^{\sharp}}    &.	
	}
  $$ 	
  
 Let
  $A_{\widehat{\varphi}_{\widehat{\eta}},0}
     := \tilde{\widehat{\varphi}}_{\widehat{\eta}}^{\sharp}
         (C^{\infty}(\widehat{X\times {\Bbb R}^n}_{[s_1,s_2]})_{\even})$,
  with the quotient $C^{\infty}$-ring structure.
Then, same as Step (d) there,
  the following commutative diagram of $C^{\infty}$-ring-homomorphisms
  $$
	 \xymatrix{
		  A_{\widehat{\varphi}_{\widehat{\eta}},0}
			     &&& C^{\infty}({\Bbb R}^{n|s_2})_{\even}
				              \ar[lll]_-{\widehat{\varphi}_{\widehat{\eta}}^{\sharp}|_{\tinyeven}}
			                  \ar@{_{(}->}^-{pr_{{\Bbb R}^{n|s_2}}^{\sharp}|_{\tinyeven}}[d]   \\			    
		    \rule{0ex}{1.2em}\;\;C^{\infty}(\widehat{X}_{[s_1]})_{\even}\;\;
			                  \ar@{^{(}->}[u]
			                  \ar@{^{(}->}[rrr]_-{pr_{\widehat{X}_{[s_1]}}^{\sharp}|_{\tinyeven}}
				 &&& C^{\infty}(\widehat{X\times {\Bbb R}^n}_{[s_1,s_2]})_{\even}
				              \ar@{->>}
							   [lllu]_-{\tilde{\widehat{\varphi}}_{\widehat{\eta}}^{\sharp}|_{\tinyeven}}
		}
   $$
  justifies that
   $\widehat{\varphi}_{\widehat{\eta}}^{\sharp}$ is $C^{\infty}$-admissible.

This proves the theorem.

\bigskip

\section{Remarks on fermionic D-branes and on ``noncommutative $C^{\infty}$-rings" after the study}

Some remarks are given in this section to connect to future works.

\bigskip

\subsection{Fermionic D-branes as dynamical objects \`{a} la RNS or GS fermionic strings}

With the improved understanding of the notion of
  `morphisms from an Azumaya/matrix super $C^{\infty}$-manifold to a super $C^{\infty}$-manifold',
one can now spell out how a dynamical fermionic stacked D-brane in string theory can be described in our language,
 following Polchinski-meeting-Grothendieck:
  
\bigskip

\begin{definition-prototype}
 {\bf [fermionic D-branes \`{a} la RNS or GS fermionic string]}$\;$ {\rm
 Recall the setting in Sec.$\:$2.1.
 \;(1)
 Let
   $\widehat{Y}
      =(Y, \widehat{\cal O}_Y:=  \bigwedge^{\bullet}_{{\cal O}_Y}{\cal S}_2^{\vee})$
    be a super $C^{\infty}$-manifold,
   with the underlying $C^{\infty}$-manifold $Y$ equippend with a Riemannian or Lorentzian metric
      (depending on the context).
 Then,
   a {\it fermionic D-brane} (synonymously, {\it super D-brane}) {\it on $\widehat{Y}$}
     (or {\it fermionic D-brane world-volume}, synonymously {\it super D-brane world-volume}
         on $\widehat{Y}$ when appropriately formulated for $Y$ Lorentzian)
   consists of the following data:
   $$
     (X,\: E,\: S_1,\:  \widehat{\nabla},\:
	     \widehat{\varphi}: \widehat{X}^{\!A\!z}\rightarrow \widehat{Y})\,,
   $$
  where (cf.$\:$ Definition~2.1.1)
  %
  \begin{itemize}
   \item[\Large $\cdot$]
    $X$ is a $C^{\infty}$-manifold
	(with a Riemannian or Lorentzian metric, depending on the context),
		
   \item[\Large $\cdot$]	
   $E$ is a smooth complex vector bundle on $X$ (of rank$_{\Bbb C}$ $r$),
   
   \item[\Large $\cdot$]	
    $S_1$ is a smooth real vector bundle on $X$ (of rank$_{\Bbb R}$ $s_1$),
	
   \item[\Large $\cdot$]
    $(\widehat{X}^{\!A\!z},\widehat{\cal E})$
      is the Azumaya/matrix super $C^\infty$-manifold with a fundamental module
	  specified by $(E,S)$,

   \item[\Large $\cdot$]
     $\widehat{\nabla}$ is a connection on $\widehat{E}$,
   
   \item[\Large $\cdot$]	
    $\,\widehat{\varphi}: \widehat{X}^{\!A\!z}\rightarrow \widehat{Y}\,$					
     is a $C^{\infty}$-map from $\widehat{X}^{\!A\!z}$ to $\widehat{Y}$,
	 defined contravariantly by a ring-homomorphism
	 $$
       \widehat{\varphi}^{\sharp}\;:\; C^{\infty}(\widehat{Y})\; \longrightarrow\;
	     C^{\infty}(\End_{\widehat{\Bbb C}_{[s_1]}}(\widehat{E}))
     $$	
   	  over ${\Bbb R}\hookrightarrow{\Bbb C}$,
  \end{itemize}
 such that
   \begin{itemize}
	\item[\Large $\cdot$]
	 The pair $(\widehat{\varphi},\widehat{\nabla})$
	  satisfies a set of mathematical and/or physical constraints;\\
	  cf.$\:$Remark~3.1.2.
  \end{itemize}						

 \centerline{\includegraphics[width=0.80\textwidth]{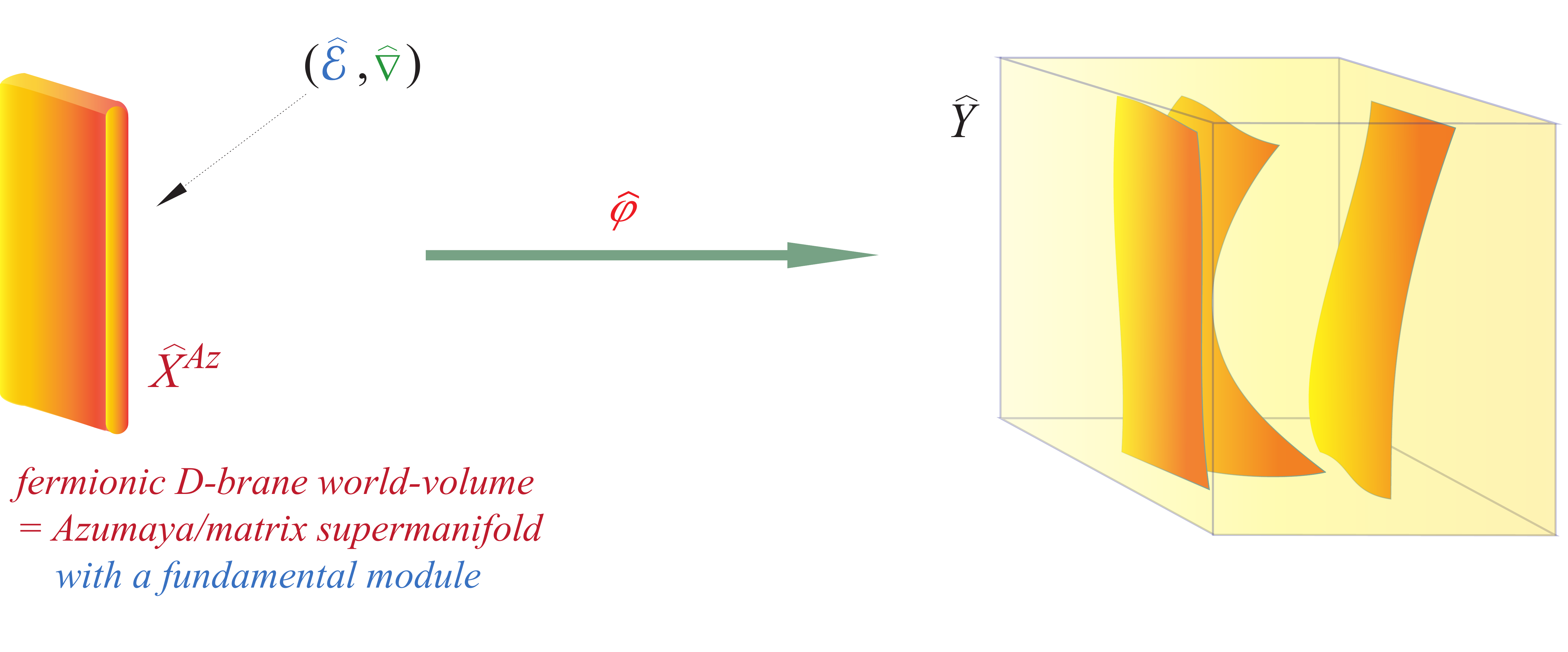}}
  
 \smallskip
  
 (2) For a {\it fermionic D-brane in the Ramond-Neveu-Schwarz formulation},
  we require that $S_1$ be a (direct sum of) spinor bundle(s) on $X$ in Item (1).
 In this case, the super target-space $\widehat{Y}$ can be just the $C^{\infty}$-manifold $Y$
  (i.e.$\:S_2$ can be set to the zero-bundle);
  and to incorporate (map, mappino)-pair into one single field
      on which the world-volume supersymmetry can act,
    $\widehat{\varphi}^{\sharp}$ is forced to violate the ${\Bbb Z}/2$-grading.
 Furthermore, the pair $(\widehat{\varphi}^{\sharp},\widehat{\nabla})$
  must satisfy additional  constraint equations to match with representations of
  (global or localized) supersymmetry algebra on the domain $\widehat{X}$;
  cf.$\:$Remark~3.1.2.
 
 \smallskip	
 
 (3) For a {\it fermionic D-brane in the Green-Schwarz formulation},
  we require that $S_2$ be a (direct sum of) spinor bundle(s) on $Y$ in Item (1).
  In this case,
    $S_1$ on $X$, while required by Algebraic Geometry, can be taken as auxiliary   and
	needs not to come from spinor bundles on $X$.
  One may require $\widehat{\varphi}^{\sharp}$ be ${\Bbb Z}/2$-grading-preserving.
  Again, the pair $(\widehat{\varphi}^{\sharp},\widehat{\nabla})$
    must satisfy additional constraint equations to match with representations of
  (global or localized) supersymmetry algebra on the target $\widehat{Y}$;
  cf.$\:$Remark~3.1.2.
 }\end{definition-prototype}

\bigskip

\begin{remark} $[$constraints on $(\widehat{\varphi}^{\sharp},\widehat{\nabla})]\;\:${\rm
 The constraints on $(\widehat{\varphi}^{\sharp},\widehat{\nabla})$ come from two sources.
 The first set of constraints comes from
   an admissible condition on $(\widehat{\varphi}^{\sharp},\widehat{\nabla})$
   so that covariant tensors on $\widehat{Y}$ can be pulled back to covariant tensors on $\widehat{X}$
   --- an issue one always has to face and resolve
         in the construction of an action functional for $(\widehat{\varphi}^{\sharp},\widehat{\nabla})$'s.
 Such an admissible condition may turn out to have a physical meaning;		
  cf.$\:$[L-Y5] (D(13.1)), [L-Y6] (D(13.2.1)), [L-Y7] (D(13.3)).
		
 The second set of constraints
   arises in both the Ramond-Neveu-Schwarz formulation and the Green-Schwarz formulation
   from the fact that $(\widehat{\varphi}^{\sharp},\widehat{\nabla})$
    in general has far more independent component fields
	than as dictated by representations of the supersymmetry algebra in question.
 Such redundant degrees of freedom have to be removed by imposing suitable constraint equations on
   $(\widehat{\varphi}^{\sharp},\widehat{\nabla})$. 		
 The details of such constraint equations depend on the supersymmetry algebra
   and have to be investigated case by case
   by the dimension of the space(-time) in question and by the total number of supersymmetries involved;
 cf.$\;$[West], [W-B].
 Though of physical origin, they may turn out to have mathematical/geometrical meaning. 
}\end{remark}
 
\bigskip

The stage is now set ([L-Y3] (D(11.2)) \& the current notes);
 the cast (Definition~2.1.1 \& Definition~2.1.6 \& Definition-Prototype~3.1.1) is in position
  (Theorem~2.1.5 \& Theorem~2.1.8);
 related bosonic exercises were practiced ([L-Y5] (D(13.1)) \& [L-Y7] (D(13.3)));
let the fermionic play begin.

\bigskip

\subsection{Remark on the notion of `$C^{\infty}$-admissible noncommutative rings'}

Before leaving the notes, we give a remark on the notion of `{\sl $C^{\infty}$-admissible noncommutative rings}'
 as a byproduct of our study.

A $C^{\infty}$-ring is always commutative by the very nature of the $C^{\infty}$-ring structure.
However, the proof of [L-Y4: Theorem 3.2.1] (D(11.3.1)) and Theorem~2.1.8
 in the current notes suggests
 a natural notion of `{\sl $C^{\infty}$-admissible noncommutative rings}',
  or --- with a slight abuse and possible confusion of language --- ``{\it noncommutative $C^{\infty}$-rings}".

\bigskip

\begin{definition} {\bf [$C^{\infty}$-admissible noncommutative ring]}$\;$ {\rm
 An (associative, unital) ring is said to be {\it $C^{\infty}$-admissible} if
   \begin{itemize}
    \item[(1)]
     For all $n\in {\Bbb N}$ and for every finite set of commuting elements $r_1,\,\cdots\,,\, r_n\in R$,
	  an element in $R$, denoted by
	    $$
		   f(r_1,\,\cdots\,, r_n)\,,
		$$
      is uniquely defined and commutes with $r_1,\,\cdots\,,\, r_n$,	
	  for all $f\in C^{\infty}({\Bbb R}^n)$.
	
	\item[(2)]
	 Under Condition (1),
	   let $r_1,\,\cdots\,,\, r_n$, $n\in {\Bbb N}$, be a set of commuting elements in $R$.
	 Then
	   $$
	      f_1(r_1,\,\cdots\,,\, r_n)\,,\,\cdots\,,\, f_m(r_1,\,\cdots\,,\, r_n)
	   $$
	   commute also with each other
	   for all $f_1,\,\cdots\,,\, f_m\in C^{\infty}({\Bbb R}^n)$.
	
	\item[(3)]
	 Under Condition (1) and Condition (2),
	   let $f_1,\,\cdots\,,\, f_m\in C^{\infty}({\Bbb R}^n)$, $g\in C^{\infty}({\Bbb R}^m)$, and
        set $h=g(f_1,\,\cdots\,,\, f_m)\in C^{\infty}({\Bbb R}^n)$ be the composition of $g$ with\\
		$(f_1,\,\cdots\,,\, f_m)\in C^{\infty}({\Bbb R}^n\rightarrow {\Bbb R}^m)$.
     Then,
       $$
          h(r_1,\,\cdots\,,\,r_n )
		    = g(f_1(r_1,\,\cdots\,, r_n),\,\cdots\,,\, f_m(r_1,\,\cdots\,,\, r_n))
       $$	
	   for any set of commuting elements $r_1,\,\cdots\,,\, r_n\in R$, $n\in {\Bbb N}$.
	
   \item[(4)]
     Under Condition (1), the following normalization condition is set.
	 For all $n\in {\Bbb N}$ and $1\le j\le n$,
	  define $\pi_j:{\Bbb R}^n\rightarrow {\Bbb R}$ by $\pi_j(x_1,\,\cdots\,,\, x_n)=x_j$.
     Then, for any set of commuting elements $r_1,\,\cdots\,,\, r_n\in R$,
       $$
         \pi_j(r_1,\,\cdots\,,\, r_n)\;=\; r_j\,.
       $$	   	 	
   \end{itemize}
   
  A {\it morphism} between $C^{\infty}$-admissible rings $R$ and $S$ is a ring-homomorphism
   $\rho:R\rightarrow S$ such that
    $$
      f(\rho(r_1),\,\cdots\,,\, \rho(r_n))\;  =\;  \rho( f(r_1,\,\cdots\,,\, r_n))
	$$
	for all $f\in C^{\infty}({\Bbb R}^n)$, $n\in {\Bbb N}$.
}\end{definition}
 
\bigskip

\noindent
In words, a $C^{\infty}$-admissible ring is a ring such that whenever the $C^{\infty}$-structure has a chance
 to apply to its subset of elements, then it applies and works consistently.

\bigskip

\begin{definition} {\bf [$C^{\infty}$-hull]}$\;$ {\rm
 For $R$ an (associative, unital) $C^{\infty}$-admissible ring,
   let $\Lambda$ be a finite set $\Lambda$ of commuting elements in $R$.
 The {\it $C^\infty$-hull} of $\Lambda$ in $R$ is defined to be the subset
  $$
    C^\infty\mbox{\it -Hull}_{\,R}(\Lambda)\;
	 :=\;
	   \{f(r_1,\,\cdots\,,\, r_l)\, |\,
	           f\in C^{\infty}({\Bbb R}^l),\,
			   r_1,\,\cdots\,,\, r_l \in \Lambda,\,
			   l\in {\Bbb N}
			\}\;\subset \; R\,.
  $$
}\end{definition}

\bigskip
  
A $C^\infty$-hull in $R$ inherits a $C^{\infty}$-ring structure from the $C^{\infty}$-admissibility of $R$.
By construction,
 $C^\infty\mbox{\it -Hull}_{\,R}(\Lambda)$
   is the minimal $C^{\infty}$-subring of $R$ that contains $\Lambda$.
In terms of this,
 a morphism $\rho:R\rightarrow S$ between $C^\infty$-admissible rings
 is a ring-homomorphism $\rho: R \rightarrow S$
  that restricts to $C^{\infty}$-ring-homomorphims
   $C^\infty\mbox{\it -Hull}_{\,R}(\Lambda)
       \rightarrow C^\infty\mbox{\it -Hull}_{\,S}(\rho(\Lambda))$
  of $C^{\infty}$-hulls.
 
Similar argument to the proof of [L-Y4: Theorem 3.2.1] (D(11.3.1)) and Theorem~2.1.8 in the current notes
 shows that all the noncommutative rings that have appeared in our study of D-branes so far
 $$
   C^{\infty}(\End_{\Bbb C}(E))\,,\;\;
   C^{\infty}(\widehat{X}_{[s_1]})\,,\;\;
   C^{\infty}(\End_{\widehat{\Bbb C}_{[s_1]}}(\widehat{E}))\,,\;\;
   C^{\infty}(\widehat{Y}_{[s_2]})\,,\;\;
   C^{\infty}(\widehat{X\times Y}_{[s_1,s_2]})
 $$
  are $C^{\infty}$-admissible in the sense of Definition~3.2.1 above
 and that
  all the morphisms considered are morphisms between $C^{\infty}$-admissible rings in that sense.

\newpage
\baselineskip 13pt
{\footnotesize

\vspace{1em}

\noindent
chienhao.liu@gmail.com, 
 chienliu@cmsa.fas.harvard.edu; \\
yau@math.harvard.edu

}


\begin{thebibliography}{AAaaaa}
\bibitem[Ap]{} T.M.$\:$Apostol,
 {\sl Mathematical analysis}, 2nd ed., Addison-Wesley, 1974.

\bibitem[Br]{} Th.$\:$Br\"{o}cker,
 {\sl Differentiable germs and catastrophes},
 translated from the German, last chapter and bibliography by L. Lander.
 London Math.\ Soc.\ Lect.\ Note Ser.\ 17.
 Cambridge Univ.\ Press, 1975.

\bibitem[B-F]{} K.$\:$Behrend and B.$\:$Fantechi,
 {\it The intrinsic normal cone},
 {\sl Invent.\ Math.}\ {\bf 128} (1997), 45--88.
  (arXiv:alg-geom/9601010)
  
\bibitem[Du]{} E.J.$\:$Dubuc,
 {\it $C^{\infty}$-schemes},
 {\sl Amer.\ J.\ Math.}\ {\bf 103} (1981), 683--690.

\bibitem[DV-K-M]{} M.$\:$Dubois-Violette, R.$\:$Kerner, and J.$\:$Madore,
 {\it Super matrix geometry},
 {\sl Class.\ Quantum Grav.}\ {\bf 8} (1991), 1077--1089.
 
\bibitem[Ei]{} D.$\:$Eisenbud,
 {\sl Commutative algebra - with a view toward algebraic geometry},
 GTM 150, Springer, 1995.

\bibitem[E-H]{} D.$\:$Eisenbud and J.$\:$Harris,
 {\sl The geometry of schemes},
 GTM~197, Springer, 2000.

\bibitem[Fu]{} W.$\:$Fulton,
 {\sl Intersection theory},
  Ser.$\:$Mod.$\:$Surv.$\:$Math.$\:$2, Springer, 1984.

\bibitem[G-S]{} M.B.\ Green and J.H.\ Schwarz,
 {\it Supersymmetrical dual string theory},
 {\sl Nucl.\ Phys.}\ {\bf B181} (1981), 502--530.
  
\bibitem[G-S-W]{} M.B.$\:$Green, J.H.$\:$Schwarz, and E.$\:$Witten,
 {\sl Superstring theory}, {\sl vol.$\:$1}$\,$: {\sl Introduction};
 {\sl vol.$\:$2}$\,$: {\sl Loop amplitudes, anomalies, and phenomenology},
 Cambridge Univ.\ Press, 1987.
 
\bibitem[Ha]{} R.$\:$Hartshorne,
 {\sl Algebraic geometry},
 GTM 52, Springer, 1977.

\bibitem[H-K]{} K.$\:$Hoffman and R.$\:$Kunze,
 {\sl Linear algebra}, 2nd ed., Prentice-Hall, 1971.

\bibitem[Jo]{} D.$\:$Joyce,
 {\it Algebraic geometry over $C^{\infty}$-rings},
 arXiv:1001.0023 [math.AG].

\bibitem[Ko]{} A.$\:$Kock,
 {\sl Synthetic differential geometry}, 2nd ed.,
   London Math.\ Soc.\ Lect.\ Note Ser.\  333,
   Cambridge Univ.\ Press, 2006.

\bibitem[L-Y1]{} C.-H.$\:$Liu and S.-T.$\:$Yau,
  {\it Azumaya-type noncommutative spaces and morphism therefrom:
       Polchinski's D-branes in string theory from Grothendieck's
       viewpoint},
  arXiv:0709.1515 [math.AG]. (D(1))

\bibitem[L-L-S-Y]{} S.$\:$Li, C.-H.$\:$Liu, R.$\:$Song, S.-T.$\:$Yau,
 {\it Morphisms from Azumaya prestable curves with a fundamental module
       to a projective variety:
      Topological D-strings as a master object for curves},
  arXiv:0809.2121 [math.AG]. (D(2))

\bibitem[L-Y2]{} C.-H.$\:$Liu and S.-T.$\:$Yau,
  {\it D-branes and Azumaya/matrix noncommutative differential geometry,
 I: D-branes as fundamental objects in string theory  and differentiable maps
    from Azumaya/matrix manifolds with a fundamental module to real manifolds},
  arXiv:1406.0929 [math.DG]. (D(11.1))

\bibitem[L-Y3]{} --------,
 {\it D-branes and Azumaya/matrix noncommutative differential geometry,
 II: Azumaya/matrix supermanifolds and differentiable maps therefrom
      - with a view toward dynamical fermionic D-branes in string theory},
  arXiv:1412.0771 [hep-th]. (D(11.2))

\bibitem[L-Y4]{} --------,
 {\it Further studies on the notion of differentiable maps from Azumaya/matrix manifolds, I. The smooth case},
  arXiv:1508.02347 [math.DG]. (D(11.3.1))

\bibitem[L-Y5]{} --------,
 {\it Dynamics of D-branes I.
          The non-Abelian Dirac-Born-Infeld action, its first variation, and the equations of motion for D-branes
		  --- with remarks on the non-Abelian Chern-Simons/Wess-Zumino term},
  arXiv:1606.08529 [hep-th]. (D(13.1))

\bibitem[L-Y6]{} --------,
 {\it More on the admissible condition on differentiable maps $\varphi: (X^{\!A\!z},E;\nabla)\rightarrow Y$
         in  the construction of the non-Abelian Dirac-Born-Infeld action $S_{DBI}(\varphi,\nabla)$},
 arXiv:1611.09439 [hep-th]. (D(13.2.1))		
  
\bibitem[L-Y7]{} --------, 
 {\it Dynamics of D-branes II. The standard action
         - an analogue of the Polyakov action for (fundamental, stacked) D-branes}, 
 arXiv:1704.03237 [hep-th]. (D(13.3))		
  
\bibitem[L-Y8]{} --------,
 manuscript in preparation.
  
\bibitem[Mal]{} B.$\:$Malgrange,
%
%
 {\sl Ideals of differentiable functions},
  Oxford Univ.\ Press, 1966.
  
\bibitem[Man]{} Y.I.$\:$Manin,
 {\sl Gauge field theory and complex geometry},
  Springer, 1988.  
  
\bibitem[Mat1]{} J.N.$\:$Mather,
 {\it Stability of $C^{\infty}$-mappings. I. The division theorem},
    {\sl Ann.\ Math.}\  {\bf 87} (1968), 89--104;
 {\it III. Finitely determined mapgerms},
    {\sl   I.H.E.S.\ Publ.\ Math.}\  {\bf 35} (1968), 279--308.

\bibitem[Mat2]{} --------,
 {\it On Nirenberg's proof of Malgrange's preparation theorem},
  {\sl Proceedings of Liverpool Singularities Symposium I} (1969/1970),
   C. T. C. Wall ed., 116--120,
   Lect.\ Notes Math.\ 192, Springer, 1971.

\bibitem[M-R]{} I.$\:$Moerdijk and G.E.$\:$Reyes,
 {\sl Models for smooth infinitesimal analysis},
  Springer, 1991.

\bibitem[Ni]{} L.$\:$Nirenberg,
 {\it A proof of the Malgrange preparation theorem},
 {\sl Proceedings of Liverpool Singularities Symposium I} (1969/1970),
   C. T. C. Wall ed., 97--105,
   Lect.\ Notes Math.\ 192, Springer, 1971.

\bibitem[N-S]{} A.$\:$Neveu and J.H.$\:$Schwarz,
 {\it Factorizable dual model of pions},
 {\sl Nucl.\ Phys.}\ {\bf B31} (1971), 86--112.   
   
\bibitem[Po]{} J.$\:$Polchinski,
 {\sl String theory},
 vol.$\:$I$\,$: {\sl An introduction to the bosonic string};
 vol.$\:$II$\,$: {\sl Superstring theory and beyond},
 Cambridge Univ.\ Press, 1998.

\bibitem[Ra]{} P.$\:$Ramond,
 {\it Dual theory for free fermions},
 {\sl Phys.\ Rev.}\ {\bf D3} (1971), 2415--2418. 
 
\bibitem[S-W]{} S.$\:$Shnider and R.O.$\:$Wells, Jr.,
 {\sl Supermanifolds, super twister spaces and super Yang-Mills fields},
   S\'{e}minaire Math.\ Sup\'{e}r.\ 106. Press.\ Univ.\ Montr\'{e}al, 1989.

\bibitem[West]{} P.$\:$West, 
 {\sl Introduction to supersymmetry and supergravity},
  extended 2nd ed., 
 World Scientific, 1990.
 
\bibitem[Westra]{} D.B.$\:$Westra, 
 {\sl Superrings and supergroups}, 
  Ph.D.\ thesis, Universit\"{a}t Wien, October 2009. 
 
\bibitem[Wi]{} E.$\:$Witten,
 {\it Notes on supermanifolds and integration}, 
   arXiv:1209.2199 [hep-th].
   
\bibitem[W-B]{} J.$\:$Wess and J.$\:$Bagger,
 {\sl Supersymmetry and suergravity},
  Princeton Univ.\ Press, 1992.   
\end{thebibliography}
\end{document}